\documentclass[11pt]{article} 

\usepackage[utf8]{inputenc} 

\usepackage{geometry} 
\usepackage{graphicx} 
\graphicspath{ {./images/} }

\usepackage{booktabs} 
\usepackage{array} 
\usepackage{paralist} 
\usepackage{verbatim} 
\usepackage{subfig} 
\usepackage[toc,page]{appendix}
\usepackage{hhline}

\usepackage{fancyhdr} 
\usepackage[nottoc,notlof,notlot]{tocbibind} 
\usepackage[titles,subfigure]{tocloft} 

\usepackage{tikz}
\usetikzlibrary{positioning, fit, shapes}

\usepackage{amssymb}
\usepackage{amsmath}
\usepackage{amsfonts}
\usepackage{mathtools}
\usepackage{planarforest}
\usepackage[ruled,vlined]{algorithm2e}
\usepackage{upgreek}
\usepackage[mathscr]{eucal}
\usepackage{algorithmic}
\usepackage[hidelinks]{hyperref}
\usepackage{longtable}
\usepackage{makecell}

\newtheorem{theorem}{Theorem}[section]
\newtheorem{lemma}[theorem]{Lemma}
\newtheorem{proposition}[theorem]{Proposition}
\newtheorem{corollary}[theorem]{Corollary}
\newtheorem{corollary-definition}[theorem]{Corollary-Definition}

\newtheorem{definition}[theorem]{Definition}
\newtheorem{remark}[theorem]{Remark}
\newtheorem{ex}[theorem]{Example}
\newtheorem{conjecture*}{Conjecture}

\newcommand{\qed}{\hfill$\square$}

\newcommand{\graft}{\curvearrowright}

\newcommand{\subs}{\triangleright}

\newcommand{\ins}{\blacktriangleright}
\newcommand{\gl}{\diamond}

\newcommand{\dprime}{{\prime\prime}}

\newcommand{\ET}{\mathcal{ET}}
\newcommand{\EF}{\mathcal{EF}}
\newcommand{\AF}{\mathcal{AF}}
\newcommand{\AT}{\mathcal{AT}}

\newcommand{\CF}{\mathcal{CF}}
\newcommand{\CEFone}{{\mathcal{CEF}^1}}
\newcommand{\CEF}{\mathcal{CEF}}
\newcommand{\EA}{\mathcal{EA}}
\newcommand{\EAT}{\mathcal{EAT}}
\newcommand{\PEAF}{\Prim(\EAF)}
\newcommand{\EAF}{\mathcal{EAF}}
\newcommand{\blank}{{-}}
\newcommand{\tildegraft}{\mathbin{\tilde{\graft}}}

\newcommand{\one}{\ensuremath{\mathbf{1}}\xspace}

\newcommand{\E}{\ensuremath{\mathbb{E}}\xspace}

\newcommand{\N}{\ensuremath{\mathbb{N}}\xspace}

\newcommand{\R}{\ensuremath{\mathbb{R}}\xspace}
\newcommand{\T}{\ensuremath{\mathbb{T}}\xspace}

\renewcommand{\AA}{\ensuremath{\mathcal{A}}\xspace}

\newcommand{\CC}{\ensuremath{\mathcal{C}}\xspace}

\newcommand{\EE}{\ensuremath{\mathcal{E}}\xspace}
\newcommand{\FF}{\ensuremath{\mathcal{F}}\xspace}

\newcommand{\JJ}{\ensuremath{\mathcal{J}}\xspace}

\newcommand{\LL}{\ensuremath{\mathcal{L}}\xspace}
\newcommand{\MM}{\ensuremath{\mathcal{M}}\xspace}
\newcommand{\NN}{\ensuremath{\mathcal{N}}\xspace}

\renewcommand{\SS}{\ensuremath{\mathcal{S}}\xspace}
\newcommand{\TT}{\ensuremath{\mathcal{T}}\xspace}

\newcommand{\XX}{\ensuremath{\mathcal{X}}\xspace}

\DeclareMathOperator{\Div}{div}

\DeclareMathOperator{\Aut}{Aut}

\DeclareMathOperator{\id}{id}

\DeclareMathOperator{\Span}{Span}
\DeclareMathOperator{\supp}{supp}
\DeclareMathOperator{\Prim}{Prim}
\DeclareMathOperator{\Char}{Char}

\usepackage{tikz-cd}

\def\ts{\thinspace}
\newcommand{\abs}[1]{\left\vert #1\right\vert}

\newcommand\restrict[1]{\raisebox{-.5ex}{$|$}_{#1}}

\newcolumntype{C}[1]{>{\centering\let\newline\\\arraybackslash\hspace{0pt}}m{#1}}

\newenvironment{proof}[1][Proof]{\begin{trivlist}
\item[\hskip \labelsep {\bfseries #1}]}{\qed \end{trivlist}}

\title{
Hopf algebra structures for the backward error analysis of ergodic stochastic differential equations
}

\author{
Eugen Bronasco\textsuperscript{1} and Adrien Laurent\textsuperscript{2}
}

\begin{document}
\footnotetext[1]{
Université de Genève, Section de mathématiques, Genève, Switzerland. Eugen.Bronasco@unige.ch.}
\footnotetext[2]{
Univ Rennes, INRIA (Research team MINGuS), IRMAR (CNRS UMR 6625) and ENS Rennes,
France. Adrien.Laurent@INRIA.fr.}

\maketitle

\begin{abstract}
While backward error analysis does not generalise straightforwardly to the strong and weak approximation of stochastic differential equations, it extends for the sampling of ergodic dynamics.
The calculation of the modified equation relies on tedious calculations and there is no expression of the modified vector field, in opposition to the deterministic setting.
We uncover in this paper the Hopf algebra structures associated to the laws of composition and substitution of exotic aromatic S\nobreakdash-series, relying on the new idea of clumping. We use these algebraic structures to provide the algebraic foundations of stochastic numerical analysis with S\nobreakdash-series, as well as an explicit expression of the modified vector field as an exotic aromatic B\nobreakdash-series.

\smallskip

\noindent
{\it Keywords:\,} geometric numerical integration, exotic aromatic series, Butcher series, clumped forests, Hopf algebra, stochastic differential equations, backward error analysis, invariant measure, ergodicity.
\smallskip

\noindent
{\it AMS subject classification (2020):\,}
Primary:
16T05, 
41A58, 
60H35. 
Secondary:
37M25, 
65L06, 
70H45. 
%
\end{abstract}



\section{Introduction}

Consider overdamped Langevin dynamics on the~$d$-dimensional torus~$\T^d$ of the following form,
\begin{equation}
\label{equation:Langevin_Rd}
dX(t)=f(X(t))dt +dW(t),
\end{equation}
where~$f=-\nabla V$ is a gradient vector field deriving from a $\CC^\infty$ potential $V$,~$W$ is a standard $d$\nobreakdash-dimensional Brownian motion on a probability space equipped with a filtration~$(\FF_t)_{t>0}$ and fulfilling the standard assumptions.
There are three main types of approximations of~\eqref{equation:Langevin_Rd}. A strong approximation focuses on approaching~$X(t)$ for a given realisation of~$W$. A weak approximation approaches averages of functionals of the solution~$\E[\phi(X(t))]$ where~$\phi\in \CC^\infty(\T^d,\R)$ is a smooth test function.
An approximation for the invariant measure approximates averages of functionals at the equilibrium for ergodic systems.
This paper develops new algebraic tools for the creation and study of integrators for the approximation of overdamped Langevin dynamics in~$\T^d$ and also on manifolds~$\MM$, in the weak sense and for the invariant measure.

For the study of ordinary differential equations~$y'=f(y)$, backward error analysis~\cite{Hairer06gni} rewrites a one-step numerical integrator~$y_{n+1}=\Phi_h(y_n)$ as the exact solution of a modified problem~$\tilde{y}'=\tilde{f}(\tilde{y})$. The modified vector field~$\tilde{f}$ typically writes as a formal series in~$f$ and its partial derivatives (a Butcher series) and the properties of the integrator (order, preservation of invariants or measures, symplecticity,...) can be read on~$\tilde{f}$.
Backward error analysis was in particular instrumental in the understanding of the long time energy preservation property of symplectic methods~\cite{Benettin94oth}.
In the stochastic setting, it is known since~\cite{Shardlow06mef} that backward error analysis does not generalise straightforwardly for the Euler-Maruyama method in the strong or weak sense.
The approach is successful in~\cite{Zygalakis11ote} for the Milstein scheme, but only up to order two.
In~\cite{Debussche12wbe}, the idea of backward error analysis is generalised in the context of the approximation for the invariant measure: assume that the exact and numerical dynamics are ergodic, that is, the solution of~\eqref{equation:Langevin_Rd} behaves in long time according to a probability density~$\rho_\infty$:
\[\lim_{T\to\infty}\frac{1}{T}\int_0^T \phi(X(t)) dt= \int_{\T^d} \phi(x) \rho_\infty(x) dx \quad \text{almost surely},\]
and similarly, the integrator follows in long time a density~$\rho^h$, where~$h$ is the timestep.
Then, using the idea of backward error analysis, the paper~\cite{Debussche12wbe} provides an explicit expansion of the invariant measure~$\rho^h$ of the integrator of the form
$$\rho^h=\rho_\infty+h\rho_1+h^2\rho_2+\dots$$
The idea was generalised for the creation of modified equations in~\cite{Abdulle14hon}: given an equation of the form~\eqref{equation:Langevin_Rd} and an integrator~$X_{n+1}=\Phi_h(X_n)$ satisfying mild assumptions, there exists a (truncated) modified vector field~$\tilde{f}^{[N]}=f+hf_1+\ldots+h^Nf_N$ for any $N$,
such that the integrator applied to the following modified problem has order $N$ for sampling the invariant measure of~\eqref{equation:Langevin_Rd}:
\[
d\tilde{X}(t)=\tilde{f}^{[N]}(\tilde{X}(t))dt +dW(t).
\]
The modified vector field~$\tilde{f}=f+hf_1+\ldots$ is iteratively defined in~\cite{Abdulle14hon} as a formal power series in $h$ in relation to the measure $\rho^h$ of the integrator: \[
\LL^*\rho_1=-\Div(f_1\rho_\infty),
\quad
\LL^*\rho_2=-\Div(f_1\rho_1)-\Div(f_2\rho_\infty),\quad \dots
\]
where $\LL$ is the generator of equation \eqref{equation:Langevin_Rd}. Note that $\tilde{f}$ is not uniquely defined by these identities.
The calculations of the modified vector field~$\tilde{f}$ are intricate, so that the algebraic tool of exotic Butcher series was introduced in~\cite{Laurent20eab} to simplify the approach.
While the original approach of~\cite{Abdulle14hon} works at any order, there is no proof in~\cite{Laurent20eab} that the calculations rewrite with exotic B\nobreakdash-series beyond order three. In addition, there is no explicit expression of the modified vector field~$\tilde{f}$.
This paper answers these problems by providing the Hopf algebra foundations of exotic series and by using them to give a new explicit expression of the modified vector field~$\tilde{f}$ as an exotic B\nobreakdash-series at any order.



Butcher-series were first introduced in~\cite{Butcher72aat,Hairer74otb} (see also the textbooks~\cite{Hairer06gni,Butcher16nmf,Butcher21bsa} and the review~\cite{McLachlan17bsa}) for the study of order conditions for Runge-Kutta methods in numerical analysis.
They were later applied successfully to a variety of different fields such as geometric numerical integration~\cite{Hairer06gni}, quantum field theory~\cite{Connes98har}, rough paths~\cite{Gubinelli10ror,Hairer15gvn}, or to stochastic numerical analysis.
We mention in particular the aromatic extension of B\nobreakdash-series introduced in~\cite{Chartier07pfi,Iserles07bsm} for the study of volume-preserving integrators. They allow to compute the divergence of a B\nobreakdash-series, and were later studied in~\cite{McLachlan16bsm, MuntheKaas16abs, Bogfjellmo19aso, Floystad20tup, Bogfjellmo22uat, Laurent23tab, Laurent23tld} for their algebraic, geometric, and numerical properties.
For stochastic numerical analysis, we mention in particular the early works~\cite{Burrage96hso,Komori97rta,Burrage00oco} that first introduced stochastic trees and B\nobreakdash-series for the strong convergence of SDEs, and the works by Rö{\ss}ler~\cite{Rossler04ste,Rossler06rta,Rossler06rkm,Rossler10ste} and Debrabant and Kv{\ae}rn{\o}~\cite{Debrabant08bsa,Debrabant10rkm,Debrabant11cos} for the design and analysis of high order strong and weak integrators on a finite time interval, as well as the works~\cite{Anmarkrud17ocf,Kupper12ark}.
The different stochastic extensions of B\nobreakdash-series were recently formalised and unified into the exotic aromatic S\nobreakdash-series formalism~\cite{Laurent20eab, Laurent21ocf, Laurent21ata, Bronasco22ebs} for the study of stochastic numerical analysis in the weak sense and for the invariant measure. It was then shown in~\cite{Laurent23tue} that the exotic aromatic formalism is a very natural extension of B\nobreakdash-series and aromatic B\nobreakdash-series as they all satisfy similar universal geometric properties.

We present in this paper the algebraic and combinatorial structures related to the exotic aromatic S\nobreakdash-series and study their concrete applications in stochastic numerical analysis.
This is not straightforward as the Hopf algebra structures used in the deterministic case~\cite{Chartier05asl,Chartier10aso} do not extend naturally to the stochastic case.
We introduce the concept of decorated aromatic S\nobreakdash-series, defined using decorated aromatic forests, to simplify the study of the algebraic structures related to the exotic aromatic S\nobreakdash-series. Our investigation involves the study of the D-algebra structure~\cite{Munthe-Kaas-Wright-Dalgebra} over the decorated aromatic forests, and we employ this structure to introduce a Hopf algebroid~\cite{Lu_Hopf_algebroids,RinehartBialgebra}, which we refer to as the Grossman-Larson Hopf algebroid. This Hopf algebroid is intricately linked to the composition law of decorated aromatic S\nobreakdash-series.
The substitution law does not generalise straightforwardly with decorated aromatic S\nobreakdash-series and the main difficulty comes from the aromas, as first observed in~\cite{Bogfjellmo19aso} in the context of aromatic forests.
We introduce the concept of clumped forests, which represent monomials of aromatic trees and form the universal enveloping algebra of aromatic trees. The difference between aromatic forests and clumped forests lies in the aromas that are attached to the rooted components in the latter case:~$\forest{(b),b} \cdot \forest{b[b]} \neq \forest{b}\cdot\forest{(b),b[b]}$.
These clumped forests are essential in establishing the algebraic foundations required for the substitution law, in the spirit of~\cite{Calaque11tih,Lundervold11hao,Lundervold13bea,Lundervold15oas,Rahm22aoa}. Subsequently, we proceed to construct the space of exotic aromatic forests using the decorated aromatic forests we introduced earlier. We use our analysis of algebraic structures over decorated aromatic forests to present the composition~\cite{Bronasco22ebs} and substitution law for the exotic aromatic S\nobreakdash-series.


The Hopf algebra structures related to exotic aromatic S\nobreakdash-series allow us to elegantly formalise stochastic numerical analysis theory.
The composition law gives the methodology to derive order conditions for
arbitrarily high orders, and to study the accuracy of the composition of different numerical integrators and of postprocessors~\cite{Vilmart15pif}.
The substitution law allows us to derive a new explicit expression of a modified vector field that writes as an exotic B\nobreakdash-series for backward error analysis and for creating modified equations in the Euclidean case. This simplifies greatly the computation of the modified vector field and the exotic B\nobreakdash-series expression ensures that the modified vector field satisfies important natural geometric properties such as orthogonal equivariance~\cite{Laurent23tue}.


The structure of the paper is the following.
Section~\ref{section:Decorated aromatic forests} presents a comprehensive summary of the new algebraic structure of exotic aromatic forests, relying on the novel idea of clumping and on two Hopf algebras of trees.
The numerical applications are gathered in Section~\ref{section:numerical_analysis_application} and include the algebraic formalisation of stochastic order theory in the weak sense and for the invariant measure, the application of the composition law to the composition of integrators and postprocessors, and the application of the substitution law to backward error analysis and modified equations.
We then present in Section~\eqref{sec:algebraic_structure} the detailed study of the algebraic structures associated to exotic aromatic forests.

\section{Exotic aromatic and clumped forests}
\label{section:Decorated aromatic forests}

We consider graphs~$\pi=(V,E,S)$ where~$V$ is a finite set of vertices and~$E\subset V\times V$ is a set of directed edges. The empty graph is included and written~$\mathbf{1}$. If~$e = (v,w) \in E$, the edge~$e$ is going from the source~$v$ to the target~$w$,~$v$ is a predecessor of~$w$, and~$w$ is a successor of~$v$.
The stolons\footnote{In botany, stolons are horizontal connections that link the base of two plants, allowing a plant to clone itself. Strawberry plants are an example of plants with stolons.} in~$S\subset V\times V/\langle(x,y)-(y,x)\rangle$ link some of the vertices without successors by pairs, but these vertices are in at most one stolon.
The vertices that do not have a successor and are not part of stolons are called roots and the vertices that are not in stolons and do not have predecessors are called leaves.
The graphs have two kinds of connected components: the ones that have a root, called trees, and the other ones, called aromas. The aromas either have a stolon or have a cycle, that is a list of vertices~$\{v_1,\dots,v_n\}$ where~$v_{i+1}$ is the successor of~$v_i$ and~$v_1$ is the successor of~$v_n$.
We call aromatic forests, gathered in the set~$AF$, the equivalence classes of such graphs, where two graphs are equivalent if there exists a bijection between their sets of vertices and edges that preserve successors, predecessors, and stolons.
When drawn, by convention, the orientation of the edges goes from top to bottom and in counterclockwise direction for cycles.

We consider the following subsets of~$AF$.
The set of trees~$T$ contains the connected aromatic forests with one root. The trees with up to four vertices are
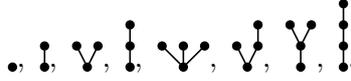
\[ \forest{b}, \ \forest{b[b]}, \ \forest{b[b,b]}, \ \forest{b[b[b]]}, \ \forest{b[b,b,b]}, \ \forest{b[b,b[b]]}, \ \forest{b[b[b,b]]}, \ \forest{b[b[b[b]]]}. \]
We obtain the set~$F$ of forests by taking all possible unordered monomials of trees, including the empty forest~$\one$.
\[F=\{ \mathbf{1}, \ \forest{b}, \ \forest{b,b}, \ \forest{b[b]}, \ \forest{b,b,b}, \ \forest{b,b[b]}, \ \forest{b[b,b]}, \ \forest{b[b[b]]},\dots\} \]
The subset of aromatic forests with~$n$ roots is written~$AF_n$, which gives us a first grading on~$AF$. In particular, the set of aromas is
\[A=AF_0=\{\mathbf{1}, \ \forest{(b)}, \ \forest{b=b}, \ \forest{(b,b)}, \ \forest{(b[b])}, \ \forest{b=b,(b)}, \ \forest{b=b,b=b}, \dots\},\]
with the double edge denoting a stolon, and the aromatic trees are in
\[AT=AF_1=\{ \forest{(b,b[b,b[b]],b[b]),(b,b),b[b]}, \ \forest{(b),(b),(b,b[b]),b[b[b]]}, \ \forest{b[b,b]},\dots \}.\]
Note that the set~$AF \simeq A \times F$ of aromatic forests is the product of the sets of aromas and forests.
The vector space spanned by a given set is written in calligraphic font. For instance, we write~$\AF=\Span_\R(AF)$.

Note that the definition of aromas taken here is more general than in~\cite{Iserles07bsm,Chartier07pfi,Bogfjellmo19aso}, as we choose to include stolons among the aromas. The reason is that the combinatorial structure of stolonic and cyclic aromas is similar, and the exotic extension of the standard aromatic forests relies in particular on stolons.

Decorated aromatic forests are aromatic forests~$\pi \in AF$ endowed with a decoration of their vertices~$\alpha : V(\pi) \to D$ for a given set~$D$. A morphism~$\varphi : (\pi_1, \alpha_1) \to (\pi_2, \alpha_2)$ between decorated aromatic forests is a graph morphism~$\varphi : \pi_1 \to \pi_2$ such that~$\alpha_1 = \alpha_2 \circ \varphi$. The group of automorphisms is denoted by~$\Aut (\pi, \alpha)$ and~$\sigma(\pi,\alpha) := |\Aut (\pi, \alpha)|$. The set of decorated aromatic forests is denoted by~$AF_D$ and the corresponding vector space by~$\AF_D$.
We will often omit writing~$\alpha$ for simplicity and will denote the elements of~$AF_D$ by~$\pi$.

The vector space~$\AF_D$ of decorated aromatic forests forms a commutative algebra~$(\AF_D, \cdot)$ which can be defined as the symmetric algebra~$S_{\AA_D}(\AT_D)$ of decorated aromatic trees~$\AT_D$ over the ring of decorated aromas~$\AA_D$.
We consider the symmetric algebra~$\CF_D := S(\AT_D)$ of decorated aromatic trees~$\AT_D$ over the base field~$\R$. The basis of~$\CF_D$ is denoted by~$CF_D$ and its elements are called clumped forests. We note that~$\forest{(b),b} \cdot \forest{b[b]} \neq \forest{b}\cdot\forest{(b),b[b]}$ in~$CF_D$.
Some elements of~$\CF_D$ are listed below:
\[ \forest{(b),b} \cdot \forest{b[b,b]}, \quad \forest{b} \cdot \forest{(b),b[b,b]}, \quad \forest{(b,b[b]),b[b,b]} \cdot \forest{b=b,b[b]} \cdot \forest{(b[b,b]),b=b[b],b}. \]

A detailed study of the algebraic structure of decorated aromatic and clumped forests is presented in Section~\ref{sec:algebraic_structure}.

\subsection{Decorated aromatic S-series}
\label{sec:S-series}

Let~$I$ be a finite set of indices. Then the elementary differential~$f^k_I$ is defined as:
\[ f^k_I := \frac{\partial^{|I|}f^k}{\prod_{i \in I} \partial x_i}.\]
We define the elementary differential map~$F_D$ by
\begin{align*}
F_D(\pi)[\phi]=\sum_{\underset{v\in V}{i_v=1,\dots,d}}
    \delta_ {I_S}\left( \prod_{v\in V} f^{i_v}_{\alpha(v), I_{\Pi(v)}} \right) \phi_{I_{R}},
\end{align*}
where~$V$ is the set of vertices,~$R$ is the set of roots,~$\Pi(v)$ is the set of predecessors of~$v$,~$\delta_{I_S}$ identifies the indices of the stolons, that is,~$i_x=i_y$ if~$(x,y)\in S$, and~$f_{\alpha(v)}$ is the vector field corresponding to the vertices decorated by~$\alpha(v)$. For example,
\[ F_D (\forest{b=w,(w),w[b,b],b[w]})[\phi] = \sum_{i,j,k,l,m,n,p=1}^d f^i g^i g^j_j f^k f^l g^m_{ij} g^n f^p_n \phi_{m,p}. \]
where~$f$ and~$g$ are vector fields corresponding to~$\bullet$ and~$\circ$, respectively.

Let us introduce S\nobreakdash-series over decorated aromatic forests. Let~$\overline{\AF}_D$ be the space of formal sums of the following form
\[ \sum_{\pi \in AF_D} a(\pi) \pi, \quad \text{with } a \in \AF_D^*.  \]
It is the completion with respect to the graduation given by the number of vertices. We extend~$F_D$ by linearity to~$\overline{\AF}_D$.
Recall the symmetry coefficient~$\sigma(\pi) := |\Aut(\pi)|$.
Let the map~$\delta_\sigma : \AF_D^* \to \overline{\AF}_D$ be the isomorphism between the dual and the completion given by
\[ \delta_\sigma (a) = \sum_{\pi \in AF_D} \frac{a(\pi)}{\sigma(\pi)} \pi. \]
We will abuse the notation and use~$\delta_\sigma$ to denote analogous isomorphisms for other spaces. The space will be made clear from the context and will always be the domain of the functional to which~$\delta_\sigma$ is applied.

Let~$\abs{.}\colon AF_D\rightarrow\N$ be a map defining a finite grading on~$AF_D$, that we call the order on decorated aromatic forests.
\begin{definition}
\label{definition:S-series}
An S\nobreakdash-series over decorated aromatic forests is defined for~$a \in \AF_D^*$ and some smooth vector fields~$f_d$,~$d \in D$, as the following formal power series
    \[ S_{(f_d)_{d\in D}}^h(a)[\phi] := \sum_{\pi \in AF_D} h^{\abs{\pi}}\frac{a(\pi)}{\sigma(\pi)} F_D (\pi)[\phi], \quad \text{for } a \in \AF_D^*.\]
\end{definition}

We often omit the dependency in the~$f_d$ and write for simplicity~$S^h(a)$ and~$S(a)=S^1(a)$.
For example, let~$D = \{\bullet, \circ\}$ and~$F_D(\bullet) = f, F_D(\circ) = g$, then,
\begin{align*}
    S_{f,g}(a)[\phi] &= \left(1 + a(\forest{(b)}) \Div(f) + a(\forest{(w)}) \Div(g) + a(\forest{b=b}) \langle f,f\rangle + \cdots \right) \phi \\
    &+ \left(1 + a(\forest{(b),b}) \Div(f) + a(\forest{(w),b}) \Div(g) + a(\forest{b=b,b}) \langle f,f\rangle + \cdots \right) \sum_{i=1}^d f^i \phi_i + \cdots,
\end{align*}
with~$\langle \blank,\blank\rangle$ the standard Euclidean scalar product.

S\nobreakdash-series are used to represent formal sums of differential operators. We use S\nobreakdash-series to represent formal sums of vector fields by using the identification between first order operators and vector fields and by requiring the corresponding functional~$a_0 : \AF_D \to \R$ to satisfy
\[ \supp (a_0) \subset \AT_D. \]
Such series are called B\nobreakdash-series, written~$B^h(a_0)$, and~$a_0$ is called an infinitesimal character. B\nobreakdash-series are used in deterministic numerical analysis~\cite{Hairer06gni} for the representation of numerical integrators $y_1 := \Psi_h (y_0)$ applied to~$y^\prime = f(y)$ as~$\Psi_h(y_0) = y_0 + B^h(a_0)(y_0)$ for an infinitesimal character~$a_0$. This is typically done by Taylor expanding the numerical method with respect to its time stepsize~$h$. Given any smooth function~$\phi: \R^d \to \R$, the composition of~$\Psi_h : \R^d \to \R^d$ with~$\phi$ is represented by the following exponential of B\nobreakdash-series:
\[
\phi \circ \Psi_h =
F_D\Big(\exp^\cdot \big(\delta_\sigma(a_0)\big)\Big)[\phi], \quad \text{with } \exp^\cdot (x) := 1 + \sum_{k=1}^\infty \frac{x^{\cdot k}}{k!}. \]
We use the formalism of aromatic forests and the convolution product~$\odot : \AF_D^* \otimes \AF_D^* \to \AF_D^*$ defined as~$a \odot b := (a \otimes b) \circ \Delta$ with~$\Delta : \AF_D \to \AF_D \otimes \AF_D$ being the deshuffle coproduct (not~$\AA_D$-linear) over decorated aromatic forests to write the exponential of B\nobreakdash-series as an S\nobreakdash-series as
\begin{equation}
\label{eq:B-S-series}
    \phi \circ \Psi_h = S^h(\exp^\odot (a_0))[\phi] \quad \text{with } \exp^\odot (x) := 1 + \sum_{k=1}^\infty \frac{x^{\odot k}}{k!}.
\end{equation}


\subsection{Exotic aromatic S-series}
\label{section:structure_EAF}

Exotic aromatic forests are a specific class of decorated aromatic forests that naturally appears in the study of numerical methods for solving stochastic differential equations with additive noise~\cite{Laurent20eab, Laurent21ocf, Laurent21ata}.
We use decorations to encode paired nodes, called lianas, that allow to represent the Laplacian of a Taylor series in a jet bundle.
The exotic aromatic forests are defined in the following way, where we mention that additional decorations could be added, in the spirit of partitioned S\nobreakdash-series.
\begin{definition}
\label{def:EAF}
An exotic aromatic forest is a decorated aromatic forest~$(\pi, \alpha)\in AF_D$ with the decorations~$D=\{\bullet\}\cup \N$,~$\N=\{1,2,3,\dots\}$, that follows the following rules. All vertices of~$\pi$ are black, except the leaves that can be either black or numbered.
If a natural number is used as a decoration, then it must decorate two leaves, that is,~$|\alpha^{-1}(n)| \in \{0,2\}$ for any~$n \in \N$. Two exotic aromatic forests~$(\pi_1, \alpha_1)$ and~$(\pi_2, \alpha_2)$ are considered to be identical if~$\pi_1 = \pi_2$ and there exists a map~$\varphi \in S_\N$ such that~$\alpha_1 \restrict{\N} = \varphi \circ \alpha_2 \restrict{\N}$. The pair of numbered leaves that correspond to the same number is called a liana, gathered in the set~$L\subset V\times V/\langle(x,y)-(y,x)\rangle$.
The order of an exotic aromatic forest~$\pi$ is the following, where~$V_{\bullet}$ denotes the number of black vertices of~$\pi$,
\[\abs{\pi}=\abs{V_{\bullet}}+\abs{L}-\abs{S},\]
where~$S \subset V \times V /\langle(x,y)-(y,x)\rangle$ is the set of stolons.
\end{definition}

The exotic aromatic trees of order up to two are the following. Note that the order does not coincide with the number of black vertices in general.
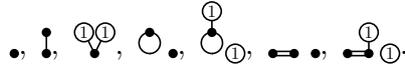
\[ \forest{b}, \ \forest{b[b]}, \ \forest{b[1,1]}, \ \forest{(b),b}, \ \forest{(b[1]),1}, \ \forest{b=b,b}, \ \forest{b=b[1],1}. \]
The set of exotic aromatic forests is denoted by~$EAF$. Exotic aromatic forests with one root are called exotic aromatic trees and form a set denoted by~$EAT$. Exotic aromatic forests without a root are called exotic aromas and form a set denoted by~$EA$. The corresponding vector spaces are denoted by~$\EAF$,~$\EAT$ and~$\EA$, respectively.
We refer to Appendix~\ref{section:tables_examples} for further examples.
We emphasize that the order of a forest is never negative.
Note also that the following exotic aromatic forests are identical:
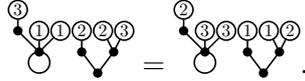
\[    \forest{(b[b[3],1,1]),b[b[2],b[2,3]]} = \forest{(b[b[2],3,3]),b[b[1],b[1,2]]}.\]

To obtain exotic aromatic S\nobreakdash-series, we define the elementary differential map by
\begin{align*}
F_\EE(\pi)[\phi]=\sum_{\underset{v\in V}{i_v=1,\dots,d}}
\delta_{I_{S\cup L}} \left(\prod_{v\in V_{\bullet}} f^{i_v}_{I_{\Pi(v)}} \right) \phi_{I_{R}},
\end{align*}
where~$R$ is the set of roots,~$\Pi(v)$ is the set of predecessors of~$v$, and~$\delta_{I_{S\cup L}}$ identifies the indices of the lianas and stolons, that is,~$i_x=i_y$ if~$(x,y)\in S\cup L$.
For instance, we have
\[
\pi=\forest{(b[1]),b=b[2],b[1],2},
\quad
F_\EE(\pi)[\phi]=\sum_{i,j,s,l_1,l_2=1}^d f^i_{il_1} f^s f^s_{l_2} f^j_{l_1} \phi_{jl_2}.
\]
Given a coefficient map~$a \in \EAF^*$, an exotic aromatic S\nobreakdash-series is the following formal series (see Definition~\ref{definition:S-series}),
\[
S^h(a) := \sum_{\pi \in EAF} h^{\abs{\pi}}\frac{a(\pi)}{\sigma(\pi)} F_{\mathcal{E}} (\pi), \quad S(a)=S^1(a),
\]
where the symmetry coefficient~$\sigma(\pi)$ counts the number of graph automorphisms leaving~$\pi$ unchanged.
The S\nobreakdash-series formalism allows us to rewrite the tedious combinatorics of numerical analysis in terms of simple graph operations that do not involve the dimension of the problem.

Theorems~\ref{theorem:exotic_composition_law} and~\ref{theorem:exotic_substitution_law} present the composition and substitution laws of exotic aromatic S\nobreakdash-series used in Section~\ref{section:numerical_analysis_application} to describe the composition of numerical methods, the construction of the modified equation of a method, and the backward error analysis. The proofs are derived straightforwardly from the analysis of Section~\ref{sec:algebraic_structure}.

\begin{definition}
\label{def:BCK}
\cite{Bronasco22ebs}
	The Butcher-Connes-Kreimer coproduct on~$\EAF$ is defined as
    \[ \Delta_{BCK} (\pi) := \sum_{\pi_0 \subset \pi} (\pi \setminus \pi_0) \otimes \pi_0, \]
	where the sum runs over all rooted subforests~$\pi_0 \in EAF$ of~$\pi$ such that~$\pi \setminus \pi_0 \in EAF$ and there are no edges going from~$\pi_0$ to~$\pi \setminus \pi_0$ in~$\pi$. We note that the paired number vertices cannot be separated across different sides of the tensor product.
\end{definition}

\begin{ex}
\label{ex:example_composition}
For example,
\begin{align*}
    \Delta_{BCK} (\forest{(b[1]),b[1,b]}) &= \mathbf{1} \otimes \forest{(b[1]),b[1,b]} + \forest{1,1} \otimes \forest{(b),b[b]} + \forest{b} \otimes \forest{(b[1]),b[1]} + \forest{(b[1]),1} \otimes \forest{b[b]} + \\
    &\quad \forest{1,1,b} \otimes \forest{(b),b} + \forest{(b[1]),1,b} \otimes \forest{b} + \forest{1,b[1,b]} \otimes \forest{(b)} + \forest{(b[1]),b[1,b]} \otimes \mathbf{1}, \\
    \Delta_{BCK} (\forest{b[1,1,2,b[2]]}) &= \mathbf{1} \otimes \forest{b[1,1,2,b[2]]} + \forest{1,1} \otimes \forest{b[2,b[2]]} + \forest{2,2} \otimes \forest{b[1,1,b]} + \forest{1,1,2,2} \otimes \forest{b[b]} + \\
    &\quad \forest{2,b[2]} \otimes \forest{b[1,1]} + \forest{1,1,2,b[2]} \otimes \forest{b} + \forest{b[1,1,2,b[2]]} \otimes \mathbf{1}.
\end{align*}
We present in Appendix~\ref{section:tables_examples} the values of~$\Delta_{BCK}$ over all connected forests and exotic aromas up to order~$3$.
\end{ex}

\begin{theorem}[Composition law]\label{theorem:exotic_composition_law}\cite{Bronasco22ebs}
    Let~$S(a)$ and~$S(b)$ be two exotic aromatic S\nobreakdash-series and let~$\phi$ be a test function. Then,
    \[ S(a)[S(b)[\phi]] = S(a \ast  b)[\phi], \quad \text{with } a \ast  b = (a \otimes b) \circ \Delta_{BCK}, \]
    with~$\Delta_{BCK}$ being the Butcher-Connes-Kreimer coproduct over exotic aromatic forests.
\end{theorem}

We define an extension~$\CEFone$ of the concept of decorated clumped forests to the exotic context as~$\CEFone := S(\EAT)$. Some elements of~$\CEFone$ are
\[ \forest{(b),b[1,1]} \cdot \forest{(b[2]),b[2]}, \quad \forest{(b[3],b[3]),b[b[b,4],4]} \cdot \forest{b=b,b[b]}. \]
Note that exotic aromatic trees are a subset of primitive elements of~$\EAF$ with the full set of primitive elements being discussed in Section~\ref{sec:exotic_algebraic_structure}.

We consider a particular case of the substitution law from Theorem~\ref{thm:substitution_law} in which the black vertices are substituted by~$\delta_\sigma(b_0) \in \overline\EAT$ while the numbered vertices are not changed, that is, a number vertex~$\forest{i_n}$ is substituted by~$\forest{i_n}$ for~$n \in \N$. This allows us to simplify the definition of the CEM coaction.

\begin{definition}
\label{def:exotic_CEM_coaction}
Let the CEM coaction~$\Delta_{CEM} : \EAF \to \CEFone \otimes \EAF$ over exotic aromatic forests be defined as 
\[ \Delta_{CEM} (\pi) := \sum_{p \subset \pi} p \otimes \pi /_p, \]
where the sum is over all clumped exotic subforests~$p \in \CEFone$ that cover all black vertices and~$\pi /_p$ is the exotic aromatic forest obtained by contracting the exotic aromatic trees of~$p$ into black vertices. If the forest~$\pi \in \EAF$ doesn't have valid subforests~$p \in \CEFone$, then~$\Delta_{CEM} (\pi) = \one \otimes \pi$.
For details see the proof of Theorem~\ref{thm:CEF_Hopf}.
\end{definition}

\begin{ex}
The substitution on exotic aromatic forests has additional difficulties compared to the deterministic context.
While the coaction~$\Delta_{CEM}$ works with a similar idea to the partitions of trees as in~\cite[Thm.\ts 3.2]{Chartier10aso}, it ensures that paired vertices stay in the same component and it sums over all possible combinations of clumped exotic forests in~$\CEFone$ (and not exotic aromatic forests) in the left side of the tensor product.
For instance, we have
\begin{align*}
  \Delta_{CEM} (\forest{(b[1]),b[1,b]}) &= \forest{(b[1]),b[1,b]} \otimes \forest{b} + \forest{(b[1]),1,b[b]} \otimes \forest{b[b]} + \forest{(b),b[b]} \otimes \forest{b[1,1]} + \forest{b,b[b]} \otimes \forest{(b[1]),b[1]} \\
  &+\forest{(b[1]),b[1],b} \otimes \forest{b[b]} + \forest{(b[1]),1,b,b} \otimes \forest{b[b,b]} + \forest{(b),b} \cdot \forest{b} \otimes \forest{b[1,1,b]} + \forest{b,b,b} \otimes \forest{(b[1]),b[1,b]}, \\
  \Delta_{CEM} (\forest{b[1,1,2,b[2]]}) &= \forest{b[1,1,2,b[2]]} \otimes \forest{b} + \forest{b[b]} \otimes \forest{b[1,1,2,2]} + \forest{b[1,1,b]} \otimes \forest{b[2,2]} \\
  &+ \forest{b[2,b[2]]} \otimes \forest{b[1,1]} + \forest{b[1,1],b} \otimes \forest{b[2,b[2]]} + \forest{b,b} \otimes \forest{b[1,1,2,b[2]]}.
\end{align*}
We present in Appendix~\ref{section:tables_examples} the values of~$\Delta_{CEM}$ over the elements of~$\EA$ and~$\PEAF$ up to order~$3$.
\end{ex}

We use the discussion from Section~\ref{sec:substitution_law} to obtain Theorem~\ref{theorem:exotic_substitution_law}.
\begin{theorem}[Substitution law]\label{theorem:exotic_substitution_law}
    Let~$a \in \EAF^*,~$~$b_0 \in \EAT^*$, then,
    \[ S_{S_f(b_0)} (a) = S_f (b_c \star a), \quad \text{with } b_c \star a = (b_c \otimes a) \circ \Delta_{CEM}, \]
    where~$b_c$ is the character of~$\CEFone$ that extends~$b_0$ and~$\Delta_{CEM}$ is defined in Definition~\ref{def:exotic_CEM_coaction}.
\end{theorem}

Theorem~\ref{thm:CEF_Hopf} provides a simplified procedure to compute the substitution law for exotic aromatic S\nobreakdash-series using the Hopf algebra structure of~$CEF$.


\section{Backward order analysis and modified equations for stochastic differential equations}
\label{section:numerical_analysis_application}

In the numerical analysis of ordinary differential equations, the composition and substitution of B\nobreakdash-series have a variety of applications, including the derivation of order conditions, the composition of numerical methods, backward error analysis, or high-order integration based on modified equations (see~\cite{Chartier10aso} and references therein).
The aromatic B\nobreakdash-series appear naturally in the study of volume-preserving integrators and we refer to~\cite{Chartier07pfi,Iserles07bsm,Bogfjellmo19aso} for details.
The stolons appear when studying projection methods on embedded manifolds, while the lianas are used in stochastic numerical analysis (see~\cite{Laurent21ata}).
In this section, we focus on the applications of the composition and substitution laws for the numerical approximation of stochastic differential equations. As described in the original references~\cite{Laurent20eab,Laurent21ocf,Bronasco22ebs}, exotic aromatic B\nobreakdash-series and S\nobreakdash-series are a crucial calculation tool for the high-order approximation of SDEs in the weak context and for the invariant measure. There are several major differences with the deterministic context, especially on the concept of backward error analysis.
Most importantly, our analysis gives an explicit expression as an exotic B\nobreakdash-series of the modified vector field for the backward error analysis and for using modified equations for the invariant measure in~$\R^d$ (see Theorem~\ref{theorem:BEA_Rd} and Theorem~\ref{theorem:modified_equation_Rd}).

\subsection{Stochastic order theory with exotic aromatic S-series}

We consider stochastic differential equations with additive noise on the~$d$-dimensional torus $\MM=\T^d$ or on smooth compact manifolds~$\MM=\{x\in\R^d,\zeta(x)=0\}$ of codimension~$1$ with a smooth constraint~$\zeta\colon\R^d\to\R$ of the form
\begin{equation}
\label{equation:Langevin}
dX(t)=\Pi_\MM(X(t)) f(X(t))dt +\Pi_\MM(X(t)) \circ dW(t),\quad X(0)=x\in \MM,
\end{equation}
where~$x$ is assumed deterministic for simplicity,~$\Pi_\MM\colon\R^d\to\R^{d\times d}$ is the orthogonal projection on the tangent bundle of the manifold~$\MM$,~$f$ is a smooth vector field,~$W$ is a standard~$d$-dimensional Brownian motion on a probability space equipped with a filtration~$(\FF_t)_{t>0}$ and fulfilling the standard assumptions.
Dynamics of the form~\eqref{equation:Langevin} include overdamped Langevin dynamics~\cite{Lelievre10fec} when~$f=-\nabla V$ for~$V$ a smooth potential.
If~$\MM=\T^d$, the orthogonal projection is~$\Pi_\MM=I_d$ and equation~\eqref{equation:Langevin} reduces to
$$
dX(t)= f(X(t))dt + dW(t).
$$

There are three main types of approximations of~\eqref{equation:Langevin}. A strong approximation focuses on approaching~$X(t)$ for a given realisation of~$W$. A weak approximation approaches averages of functionals of the solution~$\E[\phi(X(t))]$ where~$\phi\in \CC^\infty(\NN,\R)$ is a smooth test function defined on an open neighbourhood~$\NN$ of~$\MM$.
An approximation for the invariant measure approximates averages of functionals at the equilibrium for ergodic systems.
In this paper, we focus on the algebraic calculation underlying the high order approximation of~\eqref{equation:Langevin} in the weak sense and for the invariant measure. We give the definitions of order in terms of S\nobreakdash-series directly and we refer to the textbooks~\cite{Milstein04snf,Lelievre10fec,Abdulle14hon} for the detailed analysis on non-compact manifolds.

The generator of equation~\eqref{equation:Langevin} in~$\T^d$ can be represented as an exotic S\nobreakdash-series~$S(l)$ with coefficient map~$l\in \EF^*$. It is given by the following linear combination of the primitive elements of order one:
\[\delta_\sigma(l)=\forest{b}+\frac{1}{2}\forest{1,1},\quad
\LL\phi=S(l)[\phi]=\phi'f
+\frac{1}{2}\Delta \phi,\]
where we recall that~$\delta_\sigma$ is given by
\[ \delta_\sigma (a) = \sum_{\pi \in \EE\AA\FF} \frac{a(\pi)}{\sigma(\pi)} \pi. \]
In the manifold case, we add a new decoration~$\circ$ to~$\EE\AA\FF$ that represents the gradient of the constraint~$n=\nabla \zeta\colon\R^d\rightarrow\R^d$ and we extend the elementary differential with~$F_{\EE}(\circ)=n$.
The definition of the S\nobreakdash-series is
\[
S^h(a)=\sum_{\pi\in \EE\AA\FF} h^{\abs{\pi}}\frac{a(\pi)}{\sigma(\pi)\abs{n}^{\abs{V_\circ}}} F_{\EE}(\pi),\quad
\abs{\pi}=\abs{V_{\bullet}}+\frac{\abs{V_{\circ}}}{2}+\abs{L}-\abs{S},
\]
and the generator becomes
\begin{align*}
\delta_\sigma(l)&=\forest{b}
-\forest{w=b,w}
-\frac{1}{2}\forest{(w),w}
+\frac{1}{2}\forest{w=w[w],w}
+\frac{1}{2}\forest{1,1}
-\frac{1}{2}\forest{w,w},\\
\LL\phi&=S(l)[\phi]
=\phi'f
-\abs{n}^{-2}\langle n,f\rangle \phi'n
-\frac{1}{2}\abs{n}^{-2}\Div(n)\phi'n \\
&+\frac{1}{2}\abs{n}^{-4}\langle n,n'n\rangle \phi'n
+\frac{1}{2}\Delta \phi
-\frac{1}{2}\abs{n}^{-2}\phi''(n,n).
\end{align*}
The quantity of interest in the weak context is~$u(t,x)=\E[\phi(X(t))|X(0)=x]$. It satisfies the following formal expansion in a neighbourhood of~$\MM$, derived from the backward Kolmogorov equation (see, for instance,~\cite{Hasminskii80sso, Faou09csd, Abdulle14hon, Kopec15wbea, Kopec15wbeb,Laurent21ata}),
\[u(x,h) = \exp(h\LL)[\phi](x)=S^h(e)[\phi](x), \quad e=\exp^\ast (l):=\sum_{n=0}^\infty \frac{1}{n!} l^{\ast n},\]
where we recall that~$\ast$ is the composition law (see Theorem~\ref{theorem:exotic_composition_law}).
The first terms of~$e$ in~$\T^d$ are
\begin{equation}
\label{equation:exact_flow_S_Series}
\delta_\sigma(e)=\mathbf{1}
+\forest{b}
+\frac{1}{2}\forest{1,1}
+\frac{1}{2}\forest{b[b]}
+\frac{1}{2}\forest{b,b}
+\frac{1}{2}\forest{b,1,1}
+\frac{1}{4}\forest{b[1,1]}
+\frac{1}{2}\forest{b[1],1}
+\frac{1}{8}\forest{1,1,2,2}
+\dots
\end{equation}
In opposition to the deterministic formalism of Butcher trees, the S\nobreakdash-series of the exact flow is not the exponential of a combination of trees, but the exponential of a combination of forests, as the primitive elements of exotic forests do not reduce to exotic trees:
\[
\ET \subsetneq \Prim(\EF), \quad \TT=\Prim(\FF).
\]
This major difference with deterministic geometric numerical integration makes, in particular, backward error analysis much more tedious in the stochastic case, as we rely on an additional non-trivial operation, similar to the integration by parts, that transforms primitive elements into exotic trees.
In the manifold case, the first terms of~$S(e)$ can be found in~\cite{Laurent21ocf}.
Consider now a one-step integrator for solving~\eqref{equation:Langevin} of the form
\begin{equation}
\label{equation:integrator}
X_{n+1}=\Phi_h(X_n), \quad X_0=x\in \MM,
\end{equation}
where~$h$ is the stepsize of the method and the dependence in random variables is omitted for clarity.
We assume that the method~\eqref{equation:integrator} is an exotic aromatic S\nobreakdash-series method, that is, the numerical analogue of~$u(x,h)$ satisfies the following formal Talay-Tubaro expansion~\cite{Talay90eot} in a neighbourhood of~$\MM$:
\[\E[\phi(X_1)|X_0=x] = S^h(a)[\phi](x).\]
For instance, the following class of stochastic Runge-Kutta integrators presented in~\cite{Laurent21ata} satisfy this assumption naturally in~$\T^d$:
\begin{align}
  \label{equation:SRK}
  Y_n^i&=X_n+h\sum \limits_{j=1}^s a_{i,j}f(Y_n^j) +\sqrt{h}d_i \xi_n , \qquad i = 1, \dots ,s,\\
  X_{n+1}&=X_n+h\sum \limits_{i=1}^s b_i f(Y_n^i) +\sqrt{h}\xi_n, \nonumber
\end{align}
where we use one Gaussian vector~$\xi_n\sim \NN(0,I_d)$ at each step for simplicity.
Thanks to the result~\cite[Prop.\ts 4.3]{Bronasco22ebs} (see also~\cite{Hairer06gni,Rossler06rta}), the exotic S\nobreakdash-series~$S(a)$ of methods of the form~\eqref{equation:SRK} is given by the following coefficient map,
\[
\tau=\forest{b_i_270[b_j_180[1,2],1],2,b_k_270[b_l_90,b_m_90]}
\qquad
a(\tau)=\sum_{i,j,k,l,m=1}^s b_i a_{i,j} d_j^2 d_i b_k a_{k,l} a_{k,m},
\]
where the decoration of~$\tau$ by~$\{i,j,k,l,m\}$ is added for illustration purposes.
The coefficient maps~$a$ of the exact flow and of stochastic Runge-Kutta methods~\eqref{equation:SRK} are naturally characters over~$\EF$ equipped with the concatenation product, that is,~$a$ satisfies
\[
a(\pi_1\cdot \pi_2)=a(\pi_1)a(\pi_2), \quad \pi_1, \pi_2\in \EF.
\]
We denote~$\Char(\SS)$ the set of characters over a subset~$\SS$ of~$\EF$ equipped with the concatenation product.

The standard definition of weak order~\cite{Talay90eot,Milstein04snf} rewrites in the following way with exotic aromatic S\nobreakdash-series.
\begin{definition}
Let an integrator~\eqref{equation:integrator} with S\nobreakdash-series~$S^h(a)$ such that~$a-e\in \EAF^*_{\geq p+1}$, i.e., if~$a$ and~$e$ coincide on all exotic aromatic forests of order up to~$p$. Then, the integrator has (at least) weak order~$p$ for solving~\eqref{equation:Langevin}, that is, for~$T>0$ and~$h\leq h_0$ small enough with~$Nh=T$, for the initial condition~$X_0\in\MM$, for all test functions~$\phi\in\CC^\infty(\NN,\R)$, the following estimate holds
\[
\abs{\E[\phi(X_n)]-\E[\phi(X(nh))]}\leq C(h_0,T,\phi)h^p,\quad n=0,\dots,N.
\]
An integrator of at least weak order one is called consistent.
\end{definition}

\begin{ex}
Consider for instance the Euler-Maruyama method in~$\T^d$:
\begin{equation}
\label{equation:EM}
X_{n+1}=X_n+hf(X_n)+\sqrt{h}\xi_n, \quad \xi_n\sim \NN(0,I_d).
\end{equation}
The exotic aromatic S\nobreakdash-series of the Euler-Maruyama method is given by
\emph{\begin{equation}
\label{equation:example_EM_S_Series}
\exp^\cdot(\delta_\sigma(l))=\delta_\sigma(\exp^\odot(l))
=\mathbf{1}
+\forest{b}
+\frac{1}{2}\forest{1,1}
+\frac{1}{2}\forest{b,b}
+\frac{1}{2}\forest{b,1,1}
+\frac{1}{8}\forest{1,1,2,2}
+\dots,
\end{equation}}
where the convolution product is~$a \odot b := (a \otimes b) \circ \Delta$ with~$\Delta$ the deshuffle coproduct over exotic aromatic forests (see Section~\ref{sec:S-series}).
Comparing the S\nobreakdash-series \eqref{equation:exact_flow_S_Series} and~\eqref{equation:example_EM_S_Series} shows that, under the standard regularity assumptions~\cite{Talay90eot,Milstein04snf,Abdulle14hon}, the method is of weak order one.
\end{ex}

The overdamped Langevin dynamics~\eqref{equation:Langevin} with~$f=-\nabla V$ are ergodic (see, for instance, the works~\cite{Lelievre10fec,Debussche12wbe,Abdulle14hon}), that is, there exists a unique invariant measure~$d\mu_\infty$ on~$\MM$ that has a probability density~$\rho_\infty$ with respect to~$d\sigma_\MM$, the canonical measure on~$\MM$ induced by the Euclidean metric of~$\R^d$, such that for all test functions~$\phi$,
\[\lim_{T\to\infty}\frac{1}{T}\int_0^T \phi(X(t)) dt= \int_{\MM} \phi(x) d\mu_\infty(x)\quad \text{almost surely}, \quad \rho_\infty(x) \propto \exp\left(-2V(x)\right).\]
The invariant measure is the unique solution to~$\LL^* d\mu_\infty=0$.
Similarly, a consistent integrator~\eqref{equation:integrator} is ergodic if there exists a unique probability measure~$d\mu^h$ that is absolutely continuous w.r.t.\ts the measure~$d\sigma_\MM$ such that for all test functions~$\phi\in\CC^\infty(\NN,\R)$ and all initial condition~$X_0\in\MM$,
\[
\lim_{N\rightarrow\infty}\frac{1}{N+1}\sum_{n=0}^N \phi(X_n)=\int_\MM \phi(x)d\mu^h(x).
\]
We refer for instance to~\cite{Debussche12wbe} for appropriate assumptions to obtain the ergodicity of the numerical scheme.
We take over the equivalence relation~$\sim$ on exotic aromatic S\nobreakdash-series and their dual defined in~\cite{Laurent20eab,Laurent21ocf,Bronasco22ebs}.
This relation is called the integration by parts as
\[
a\sim b \quad \Rightarrow \quad \int_{\MM} S^h(a)[\phi](x) d\mu_\infty(x)=\int_{\MM} S^h(b)[\phi](x) d\mu_\infty(x),
\]
and involves detaching and grafting back edges. The integration by parts allows one to transform primitive elements into trees.
We give examples and refer to~\cite{Laurent20eab,Laurent21ocf,Bronasco22ebs} for the detailed definition:
\[
\forest{1,1} \sim -2\forest{b}, \quad
\forest{b[1],1} \sim -\forest{b[1,1]} -2\forest{b[b]}, \quad
\forest{b[1,b],1} \sim -\forest{b[1,1,b]}-\forest{b[1,b[1]]}-2\forest{b[b,b]}, \quad
\forest{b[1,2],1,2}\sim -\forest{b[1,2,2],1}-2\forest{b[1,b],1}
.
\]
The approximation for the invariant measure is defined algebraically in the following way. We shall see that it can also be understood as the stochastic extension of pseudo-volume-preserving approximations.
\begin{definition}
Consider a consistent ergodic integrator~\eqref{equation:integrator} with S\nobreakdash-series~$S^h(a)$ and invariant measure~$d\mu^h$.
Then, if~$a-\delta_{\mathbf{1}}\sim \EAF_{\geq p+1}^*$, the integrator has (at least) order~$p$ for sampling the invariant measure of~\eqref{equation:Langevin}, that is, for~$h\leq h_0$ small enough and~$\phi\in\CC^\infty(\NN,\R)$, the following estimate holds,
\[
\abs{\int_\MM \phi(x)d\mu^h(x)-\int_\MM \phi(x)d\mu_\infty(x)}\leq Ch^p.
\]
\end{definition}
An integrator of weak order~$p$ immediately has at least order~$p$ for the invariant measure as~$e\sim \delta_{\mathbf{1}}$, but there exist methods with high order for the invariant measure and weak order one (see, for instance, for the underdamped and overdamped Langevin equation~\cite{BouRabee10lra, Leimkuhler13rco, Leimkuhler16tco, Abdulle14hon, Abdulle15lta, Laurent20eab, Laurent21ocf}).
An important result for our analysis in the~$\T^d$ case with a gradient vector field~$f=-\nabla V$ is in~\cite[Thm.\ts 5.8]{Bronasco22ebs}. We mention that the assumption of this result is naturally satisfied for the stochastic Runge-Kutta integrators~\eqref{equation:SRK}.
\begin{proposition}[\cite{Laurent20eab,Bronasco22ebs}]
\label{proposition:prop_character_IBP}
There exists an algorithm that transforms~$a_c\in \Char(\EF)$ over exotic forests into an equivalent character~$a_c\circ A^*\in \Char(S(\ET))/{\sim}$ over the symmetric algebra spanned by exotic trees. This defines a map~$A(.)=. \circ A^*\colon \Char(\EF)\rightarrow \Char(S(\ET))/{\sim}$ such that~$a_c\sim A(a_c)$. Moreover,~$A$ naturally induces a map~$A\colon \Prim(\EF)^* \rightarrow \ET^*/{\sim}$ satisfying the identity~$A|_{\ET^*}=id$.
\end{proposition}

The precise algorithm for computing~$A$ is omitted for simplicity as it can be found in~\cite{Laurent20eab,Bronasco22ebs}.
The operator~$A$ is called IBP+ELI in~\cite{Bronasco22ebs} and it shares similarities with the horizontal homotopy operator presented in~\cite{Laurent23tab} (see also~\cite{Anderson89tvb, Anderson92itt}). We mention that the map~$A$ extends to the manifold case (see~\cite{Laurent21ocf}).

\begin{remark}
\label{rk:kernel_IBP}
The output of~$A$ is defined up to linear combinations of exotic trees that vanish by integration by parts. The precise description of the kernel of the integration by parts relation~$\sim$ is an important open question for the high order integration for the invariant measure. The first occurrence of a non-trivial element in the kernel appears for order four. Here is one such element:
\begin{align*}
26\forest{b[1,1,b[b]]}
&-13\forest{b[1,1,2,2]}
-5\forest{b[b[b[1]],1}
-21\forest{b[b[b,1,1]]}
+5\forest{b[b[b[1],1]]}
-5\forest{b[b[1,1],b]}
+10\forest{b[b[b,b]]}\\&
+13\forest{b[b[1],1,2,2]}
-13\forest{b[b[1,2,2],1]}
-10\forest{b[b[b],b]}
-5\forest{b[b[1],1,b]}
+5\forest{b[b[b,1],1]}
+13\forest{b[1,1,b[2,2]]}
\sim 0
\end{align*}
\end{remark}

\subsection{The composition law in stochastic numerical analysis}
\label{section:composition_SDEs}

The use of the composition rule of exotic aromatic S\nobreakdash-series in stochastic numerical analysis first appears in~\cite{Laurent20eab,Laurent21ocf,Laurent21ata} (see also~\cite{Debrabant11cos}), without the Hopf algebra formalism.

\begin{proposition}[Composition of integrators]
Consider two independent integrators~$\Phi^1_h$ and~$\Phi^2_h$ with exotic S\nobreakdash-series~$S^h(a_1)$ and~$S^h(a_2)$, then the composition of~$\Phi^1_h$ and~$\Phi^2_h$ has the following S\nobreakdash-series
\[
\E[\phi((\Phi^2_{h}\circ \Phi^1_{h})(x))] = S^h(a_1\ast a_2)[\phi](x).
\]
\end{proposition}

\begin{ex}
Consider the explicit and implicit Euler methods in~$\T^d$:
$$
\Phi^1_h(x)=x+hf(x)+\sqrt{h}\xi^1, \quad
\Phi^2_h(x)=x+hf(\Phi_2(x))+\sqrt{h}\xi^2,
$$
with independent random variables~$\xi_1$,~$\xi_2$.
Then, a calculation yields that the composed method~$\Phi$ has weak order 2 for solving equation~\eqref{equation:Langevin}:
\[
\E[\phi(\Phi_h(x))] = S^h(\mathbf{1}
+l+\frac{l^{\ast 2}}{2}+\dots)[\phi](x), \quad \Phi_h=\Phi^2_{h/2}\circ \Phi^1_{h/2}.
\]
The composed method~$\Phi$ coincides in law with the stochastic trapezoidal method
\[\Phi_h(x)=x+h\frac{f(x)+f(\Phi_h(x))}{2}+\sqrt{h}\xi.\]
\end{ex}

In the spirit of effective order for ODEs~\cite{Butcher69teo}, the postprocessing idea is used in~\cite{Vilmart15pif} for the high-order sampling of the invariant measure of~\eqref{equation:Langevin} in~$\T^d$ (or~$\R^d$). The approach with trees was then introduced in~\cite[Sec.\ts 5.3]{Laurent20eab}, where order conditions are presented for stochastic Runge-Kutta methods and postprocessors.
Consider an integrator~$\Phi$ for the invariant measure. After applying~$\Phi$ through the entire time interval, we apply a correction~$\overline{\Phi}$ at the very last step. If~$\overline{\Phi}$ is chosen carefully, the postprocessed integrator has a higher order than the original method~$\Phi$. This methodology yields a costless way to improve the order of a method.
\begin{proposition}[Postprocessed integrators]
Consider an integrator with S\nobreakdash-series~$S^h(a)$ of order~$p$ for the invariant measure of~\eqref{equation:Langevin}, and a correction with S\nobreakdash-series of the form
$$S^h(\overline{a})=\mathbf{1}+h\alpha_1 l +\dots + h^{p-1}\alpha_{p-1} l^{\ast p-1}+\dots,$$
for some constants~$\alpha_k\in\R$.
Then, the postprocessed method has order~$p+1$ for the invariant measure of~\eqref{equation:Langevin} in~$\T^d$ if the following condition is satisfied:
$$a+[l,\overline{a}]\sim \EF_{\geq p+1}^*, \quad [a,b]:=a\ast b-b\ast a.$$
\end{proposition}

\begin{ex}
The following scheme, first introduced in~\cite{Leimkuhler13rco}, has weak order one and order two for the invariant measure,
\[
\Phi_h(x)=x+hf(x+\frac{\sqrt{h}}{2}\xi)+\sqrt{h}\xi, \quad
\overline{\Phi}_h(x)=x+\frac{\sqrt{h}}{2}\xi.
\]
Further examples can be found in~\cite{Vilmart15pif,Abdulle17oes,Laurent20eab}.
\end{ex}

\subsection{The substitution law in stochastic numerical analysis}
\label{section:solving_SDEs}

The main applications of the substitution law of standard B\nobreakdash-series are backward error analysis and modified equations for ODEs. Adapting backward error analysis and modified equations in the stochastic context is challenging~\cite{Shardlow06mef} and it is an active field of research~\cite{Zygalakis11ote, Abdulle12hwo, Debussche12wbe, Abdulle14hon, Kopec15wbea, Kopec15wbeb, Li19sme, DiGiovacchino23bea, Laurent23tue}. It was in particular proven in~\cite{Debussche12wbe,Abdulle14hon} that backward error analysis rewrites nicely in the context of the invariant measure. We use the substitution law and the integration by parts operation to show that the modified vector field writes as an exotic S\nobreakdash-series, for which we provide an explicit expression.
For simplicity, we work in~$\T^d$ and discuss the manifold case at the end of the subsection.

The goal of backward error analysis is to find a modified vector field written formally as an exotic aromatic B\nobreakdash-series,
$$h\widetilde{f}=B^h(b)=hf+h^2f_1+h^3f_2+\dots, \quad b\colon \ET\rightarrow\R, \quad b(\bullet)=1,$$
for some vector fields~$f_1$,~$f_2$, \dots that typically write as polynomials in the coordinates of~$f$ and its partial derivatives, such that the invariant measure of the ergodic integrator~\eqref{equation:integrator} with S\nobreakdash-series~$S(a)$ coincides with the invariant measure of the modified dynamics:
\begin{equation}
\label{equation:modified_Langevin}
d\widetilde{X}(t)=\Pi_\MM(\widetilde{X}(t)) \widetilde{f}(\widetilde{X}(t))dt +\Pi_\MM(\widetilde{X}(t)) \circ dW(t),\quad \widetilde{X}(0)=\widetilde{X}_0\in \MM.
\end{equation}
The coefficient map~$b$ is the solution to the substitution
$
b_c \star e\sim a
$
where~$\star$ is the substitution law (see Theorem~\ref{theorem:exotic_substitution_law}).

It is known~\cite{Debussche12wbe,Abdulle14hon} that there exists a modified vector field for large classes of methods, such as stochastic Runge-Kutta methods.
The calculations are tedious and were rewritten with exotic series in~\cite{Laurent20eab, Bronasco22ebs}.
There is, however, no proof that the calculations can be carried out up to any order in these works as there is no reason in general why the modified vector field could be written as an exotic B\nobreakdash-series.
A geometric justification of the importance of writing the modified vector field as an exotic aromatic B\nobreakdash-series is given in~\cite{Laurent23tue}: it enforces that~$\widetilde{f}$ is invariant with respect to orthogonal changes of coordinates, which is a natural property in the stochastic context.
We provide here a simple and natural algebraic criterion, satisfied by large classes of methods, for the description of integrators that have a modified vector field in the form of an exotic B\nobreakdash-series. In addition, we give the first explicit expression of the modified vector field~$\widetilde{f}$ relying on the map~$A$ defined in Proposition~\ref{proposition:prop_character_IBP}. This shows in particular that exotic series are a powerful tool for the stochastic backward error analysis.

\begin{theorem}[Backward error analysis]
\label{theorem:BEA_Rd}
Consider a consistent method with the exotic S\nobreakdash-series~$S^h(a)$ for solving equation~\eqref{equation:Langevin} with~$f=-\nabla V$. Assume that~$a$ is a character of~$(\EF,\cdot)$. Then, there exists a modified vector field~$h\widetilde{f}=B^h(b)$ that can be written as an exotic B\nobreakdash-series with a coefficient map~$b\colon \ET\rightarrow\R$ satisfying~$
b_c\star e\sim a
$, and given by
    \[ b = \delta_\bullet + A \Big(\sum_{k=0}^{\infty} (-1)^k A_{\tilde\star e}^k (a - e) \Big)\restrict{\ET}, \]
    where~$A_{\tilde\star e} : \EF^* \to \EF^*$ satisfies~$A_{\tilde\star e}(x) = A(x) \tilde\star e$ and
\[ a \tilde\star e = (a \otimes e) \circ \tilde{\Delta}_{CEM}, \quad \text{and} \quad \tilde{\Delta}_{CEM} (\pi) = \Delta_{CEM} (\pi) - \bullet \otimes \pi - \pi \otimes \bullet, \]
for~$a \in \EF^*$,~$\pi \in EF$ such that~$|\pi| > 1$.
\end{theorem}

\begin{proof}
After initializing~$b_0=\delta_{\bullet}$, we construct recursively the coefficient map sequence~$(b_n)$ by
\[b_n = b_{n-1} + A(a - b_{n-1,c} \star e)\restrict{\ET}
, \quad b_{n-1,c} \star e = (b_{n-1,c} \otimes e) \circ \Delta_{CEM}, \]
with~$b_{n-1,c}:=\exp^\odot(b_{n-1})$ and the coproduct~$\Delta_{CEM} : \EF \to S(\ET) \otimes \EF$.
Assume~$b_{n-1}(\bullet)=1$, since~$A^* (\bullet) = \bullet - \forest{1,1}$ and the method associated to~$a$ is consistent, we find 
\[ b_n (\bullet) = b_{n-1}(\bullet) + (a-b_{n-1,c} \star e) (\bullet - \forest{1,1}) = 1 + a(\bullet) - a(\forest{1,1})-e(\bullet) + e(\forest{1,1}) = 1, \]
thus we obtain that~$b_n (\bullet) = 1$ for any~$n$ by induction.
For all~$\tau\in \ET$ such that~$|\tau| > 1$, using the reduced coproduct yields
\begin{align*}
    b_n(\tau)&= b_{n-1}(\tau) +A(a)(\tau) - (b_{n-1}\otimes e)(A^*(\tau)\otimes \bullet) \\
    &- (b_{n-1}\otimes e)(\bullet \otimes A^*(\tau)) - A(b_{n-1,c} \tilde\star e)(\tau) \\
    &=b_{n-1}(\tau) - A(b_{n-1})(\tau) + A(a-e - b_{n-1,c} \tilde\star e)(\tau) \\
    &= A(a-e + b_{n-1,c} \tilde\star e)(\tau),
\end{align*}
where we recall that the only exotic tree~$\tau$ for which~$|\tau| \leq 1$ is~$\tau = \bullet$ and we used~$A(b_{n-1}) = b_{n-1}$.
%
%
The first values are
\begin{align*}
    b_1 &= \delta_\bullet + A(a - e)\restrict{\ET}, \\
	b_2 &= \delta_\bullet + A(a - e)\restrict{\ET} - A(A(a - e) \tilde\star e)\restrict{\ET}, \\
	&\vdots \\
	b_n &= \delta_\bullet + A\big(\sum_{k=0}^{n-1} (-1)^k A_{\tilde\star e}^k (a - e) \big)\restrict{\ET}, \quad \text{where } A_{\tilde\star e} (x) = A(x) \tilde\star e.
\end{align*}
Since~$|A^*(\tau)|_e \leq |\tau|_e$ where~$|\tau|_e$ is the number of edges of~$\tau$, and~$\tilde{\Delta}_{CEM}^n (\tau) = 0$ if~$n \geq |\tau|_e$, we have for all~$n \geq |\tau|_e$,
\[ b_n(\tau) =  \delta_\bullet + A \big(\sum_{k=0}^{|\tau|_e - 1} (-1)^k A_{\tilde\star e}^k (a - e) \big)(\tau), \]
so that the sequence~$(b_n)$ converges to the desired coefficient map~$b$ by stationarity.
\end{proof}

\begin{ex}
The first terms of the modified vector field~$h\widetilde{f}=B^h(b)$ given by Theorem~\ref{theorem:BEA_Rd} for the Euler-Maruyama method~\eqref{equation:EM} are
\[
B^h(b)=h\forest{b}
+\frac{h^2}{2}\forest{b[b]}+\frac{h^2}{4}\forest{b[1,1]}
-\frac{h^2}{2}\forest{b[b[b]]}+\frac{h^2}{12}\forest{b[b,b]}-\frac{h^2}{4}\forest{b[b[1,1]]}-\frac{h^2}{12}\forest{b[b[1],1]}+\frac{h^2}{12}\forest{b[b,1,1]}+\frac{h^2}{12}\forest{b[1,1,2,2]}+\dots
\]
Note that removing the stochastic terms (that are, the trees with lianas) does not yield the modified vector field for the Euler method with the standard deterministic backward error analysis (see~\cite[Chap.\ts IX]{Hairer06gni}).
Indeed, high order for the invariant measure does not imply high order in the weak or strong sense \cite{Abdulle14hon}, hence the modified equations and order conditions in the deterministic sense or weak sense are not the same as for the invariant measure sampling, as highlighted in the works~\cite{Abdulle14hon, Abdulle15lta, Laurent20eab, Laurent21ocf, SanzSerna14mcm}.
\end{ex}

\begin{remark}
\label{remark:EA_bicomplex}
Similarly to the deterministic context~\cite[Chap.\ts IX]{Hairer06gni}, the properties of the scheme are observed directly on its associated modified vector field. In particular, if~$h\widetilde{f}=B^h(b)$ with a coefficient map~$b$ satisfying~$b-\delta_{\bullet}\sim \ET_{>p}^*$, then the integrator has at least order~$p$ for the invariant measure. This gives an equivalent definition of the order for the invariant measure based on the modified vector field, as observed in~\cite{SanzSerna14mcm,Abdulle15lta}.
In particular, a method is exact for the invariant measure if its modified vector field satisfies
\begin{equation}
\label{equation:stochastic_VP}
\Div(\widetilde{f}-f)+\langle f,\widetilde{f}-f\rangle=0.
\end{equation}
Equation~\eqref{equation:stochastic_VP} is analogous to the one satisfied by the modified vector field of volume-preserving aromatic B\nobreakdash-series methods for ODEs, whose understanding is already an important unsolved problem of deterministic geometric numerical integration (see, for instance, the works~\cite{Chartier07pfi,Iserles07bsm,MuntheKaas16abs,Bogfjellmo19aso, Bogfjellmo22uat, Laurent23tab}).
\end{remark}

\begin{remark}
Following~\cite{Debussche12wbe}, the exotic B\nobreakdash-series~$h\widetilde{f}=B^h(b)$ translates into a formal expansion of the invariant measure~$d\mu^h=\rho^h d\sigma_\MM$ of the integrator by solving:
$$\rho^h=\rho_\infty+h\rho_1+h^2\rho_2+\dots,\quad
\widetilde{\LL}^* \rho^h=0, \quad \widetilde{\LL}=\LL+(\widetilde{f}-f)\nabla.$$
For the first few terms, we find (see~\cite{Debussche12wbe, Abdulle14hon} for the detailed expansion),
\[
\LL^*\rho_1=-\Div(f_1\rho_\infty),
\quad
\LL^*\rho_2=-\Div(f_1\rho_1)-\Div(f_2\rho_\infty).
\]
\end{remark}

Consider now a consistent exotic aromatic B\nobreakdash-series method with S\nobreakdash-series~$S(a)$.
Similar to backward error analysis, we are interested in finding a modified vector field~$h\widetilde{f}=B^h(b)$ with~$b\in \ET^*$ and~$b(\bullet)=1$ such that
$
b_c\star a\sim \delta_{\mathbf{1}},
$
that is, the integrator applied to the modified equation~\eqref{equation:modified_Langevin} is exact.
This technique allows in particular to increase the order of a numerical method when the partial derivatives of~$f$ are not costly to evaluate (see, for instance, in the deterministic setting~\cite{Calvo94mef, Hairer06gni, Chartier07nib}).
A general expansion of the modified vector field is presented in the~$\T^d$ case in~\cite{Abdulle14hon,Laurent20eab}, but it is not an exotic B\nobreakdash-series in general and it is not unique.
Following~\cite{Bronasco22ebs}, we propose a simple criterion to obtain the existence of a modified vector field in the form of an exotic B\nobreakdash-series in the context of~$\T^d$, for which we also provide an explicit expression.
\begin{theorem}[Modified equations]
\label{theorem:modified_equation_Rd}
Consider a consistent method with the exotic S\nobreakdash-series $S^h(a)$ for solving equation~\eqref{equation:Langevin} with~$f=-\nabla V$. Assume that~$a$ is a character of~$(\EF,\cdot)$. Then, there exists a modified vector field~$h\widetilde{f}=B^h(b)$ that can be written as an exotic B\nobreakdash-series with the coefficient map~$b\colon \ET\rightarrow\R$ satisfying~$b(\bullet)=1$,
$
b_c\star a\sim \delta_{\mathbf{1}}
$, and given by
\[ b = \delta_\bullet - A \big(\sum_{k=0}^\infty (-1)^k A_{\tilde\star a}^k (a)\big). \]
\end{theorem}




To prove Theorem~\ref{theorem:modified_equation_Rd}, we use a similar reasoning as in Theorem~\ref{theorem:BEA_Rd}.
\begin{proof}
We introduce the sequence (in the spirit of the works~\cite{Abdulle14hon,Laurent20eab})
    \[ b_n = b_{n-1} - A(b_{n-1,c} \star a), \quad b_0=\delta_{\bullet}. \]
    Similarly to the proof of Theorem~\ref{theorem:BEA_Rd}, we show by induction that~$b_n (\bullet) = 1$ and we obtain the identity $b_n = -A(a + b_{n-1,c} \tilde\star a)$ using the reduced coproduct.
The sequence~$(b_n)$ takes the following first values
    \begin{align*}
        b_1 &= \delta_\bullet - A(a), \\
        b_2 &= \delta_\bullet - A(a) + A(A(a) \tilde\star a), \\
        &\vdots \\
        b_n &= \delta_\bullet - A\big(\sum_{k=0}^{n-1} (-1)^k A_{\tilde\star a}^k (a)\big), \quad \text{where } A_{\tilde\star a} (x) = A(x) \tilde\star a,
    \end{align*}
and~$(b_n)$ converges to~$b=\delta_\bullet - A\big(\sum_{k=0}^\infty (-1)^k A_{\tilde\star a}^k (a)\big)$ by stationarity.
\end{proof}

Any stochastic Runge-Kutta method~\eqref{equation:SRK} has a coefficient map that is a character, so that Theorem~\ref{theorem:modified_equation_Rd} applies and there exists a modified vector field that can be written as an exotic B\nobreakdash-series.
We refer to~\cite[Sec.\ts 5.1]{Laurent20eab} for examples in~$\T^d$.
In the manifold case, the most popular integrators for solving constrained SDEs in the weak sense or for sampling the invariant measure are projection methods (see, for instance, the textbook~\cite[Chap.\ts 3]{Lelievre10fec}).
In~\cite{Laurent21ocf}, a new class of Runge-Kutta projection methods was introduced for the high-order approximation of constrained overdamped Langevin processes in codimension one.
Theorems~\ref{theorem:BEA_Rd} and~\ref{theorem:modified_equation_Rd} do not extend straightforwardly to the manifold case as there does not exist an equivalent of Proposition~\ref{proposition:prop_character_IBP} with general exotic aromatic S\nobreakdash-series. However, one can adapt the algorithm described in the proof of Theorem~\ref{theorem:modified_equation_Rd}, taking over the integration by parts relation~$\sim$ described in~\cite{Laurent21ocf}. Define the projection operator~$\Pi\colon \EAT\rightarrow \EAT$ by~$\Pi \tau=\tau -\tau\ins_{\bullet}(\forest{w=b,w})$, with~$\ins_{\bullet}$ the decorated insertion product (see Section~\ref{sec:substitution_law}), and its dual~$\Pi^*$. Define~$b_0=\delta_{\bullet}$.
For the induction, at the step~$n$, if~$A(b_{n-1,c} \star a)$ is well-defined and if~$\Pi^* A(b_{n-1,c} \star a)=A(b_{n-1,c} \star a)$, compute \[ b_n = b_{n-1} - A(b_{n-1,c} \star a).\]
Then, applying the integrator to the equation~\eqref{equation:modified_Langevin} with the modified vector field~$B_{h,f}(b_n)$ yields a method of order~$n+1$ for the invariant measure of the original problem~\eqref{equation:Langevin}.

\begin{ex}
For the sake of simplicity, we consider the unit sphere~$\MM=\{x\in\R^d, \abs{x}^2=1\}$, with~$n(x)=x$.
The constrained Euler scheme with implicit projection direction:
\[
X_{n+1}=X_n+hf(X_{n+1})+\sqrt{h}\xi_n+\lambda_n X_{n+1},\quad \abs{X_{n+1}}=1.
\]
has order 2 for sampling the invariant measure of~\eqref{equation:Langevin} when applied to the modified equation~\eqref{equation:modified_Langevin} with the modified vector field
\begin{align*}
\widetilde{f}&=f+h\bigg[
-\frac{1}{2}f'f
-\frac{1}{4}\Delta f
+\frac{3}{4}f
-\frac{1}{4}\Div(n)f
-\frac{1}{2} \langle n,f\rangle f
-\frac{1}{4} f'n\\&
-\frac{1}{4}\Div(n)f'n
-\frac{1}{2} \langle n,f\rangle f'n
-\frac{1}{4} f''(n,n)
+\frac{1}{2} \langle n,f\rangle \langle n,f'n\rangle n\\&
+\frac{1}{4} \langle n,f''(n,n)\rangle n
-\frac{1}{4} \Div(f)'n n
-\frac{1}{2} \langle n,f'f\rangle n
+\frac{1}{4} \langle n,f'n\rangle n\\&
+\frac{1}{4}\Div(n)(n,f'n\rangle n
+\frac{1}{2}\langle n,f\rangle^2 n
-\frac{3}{4}\langle n,f\rangle n
+\frac{1}{4}\Div(n)\langle n,f\rangle n
\bigg].
\end{align*}
%
\end{ex}

\section{Algebraic structure of decorated and exotic aromatic forests}
\label{sec:algebraic_structure}

This section is devoted to a detailed study of the algebraic structures of decorated aromatic and exotic aromatic forests, which are used to describe the corresponding substitution laws presented in Theorems~\ref{theorem:exotic_substitution_law} and~\ref{thm:substitution_law}. Section~\ref{sec:D-algebra} defines a free D-algebra~\cite{Munthe-Kaas-Wright-Dalgebra} over decorated aromatic forests and describes the corresponding Grossman-Larson Hopf algebroid and pre-Hopf algebroid, generalizing the Grossman-Larson Hopf algebra and pre-Hopf algebra~\cite{post-Hopf-algebra} structures. The Grossman-Larson Hopf algebroid is closely related to the composition of differential operators and, consequently, to the composition law of S\nobreakdash-series.

Section~\ref{sec:decorated_clumped_forests} focuses on the algebraic structure of decorated clumped forests as well as the relation between clumped and aromatic forests. The introduction of decorated clumped forests is necessary due to the fact that the substitution law is described using a homomorphism with respect to a product which is not~$\AA_D$-bilinear, and therefore, we need to attach aromas to rooted components. Section~\ref{sec:substitution_law} introduces the substitution law for decorated aromatic forests, while Section~\ref{sec:exotic_algebraic_structure} details the algebraic results presented in this section in the context of exotic aromatic forests.

\subsection{D-algebra of decorated aromatic forests}
\label{sec:D-algebra}

In this section, we introduce algebraic structures over decorated aromatic forests that are relevant in our description of the substitution law. 
Let~$\graft \,: \AT_D \otimes \AT_D \to \AT_D$ denote the grafting product over decorated aromatic trees. The grafting product~$\tau \graft \gamma$ is the sum over all ways to attach the root of~$\tau$ to a vertex of~$\gamma$, for example,
\[ \forest{(b),b[b]} \graft \forest{(b,b),b} = 2 \, \forest{(b),(b,b[b[b]]),b} + \forest{(b),(b,b),b[b[b]]}. \]
Divergence of an aromatic tree is defined to be a map~$d : \AT_D \to \AA_D$ such that~$d(\tau)$ is a sum over all ways to attach the root of~$\tau$ to one of its vertices, for example,
\[ d(\forest{(b),b[b,b]}) = \forest{(b[b[b,b]])} + \forest{(b),(b[b,b])} + 2 \forest{(b),(b,b[b])}. \]
Let~$El_{\AA_D} (\AT_D)$ be the algebra of endomorphisms~$\blank \graft \tau : \AT_D \to \AT_D$. The pair~$(\AT_D, \AA_D)$ is the tracial pre-Lie-Rinehart algebra generated by the set~$D$, that is, it satisfies the following properties:
\begin{enumerate}
    \item~$\AA_D$ is a unital commutative algebra with concatenation product~$\cdot : \AA_D \otimes \AA_D \to \AA_D$,
    \item~$\AT_D$ is an~$\AA_D$-module with a pre-Lie product~$\graft \,: \AT_D \otimes \AT_D \to \AT_D$, that is,
    \[ \tau \graft (\gamma \graft \nu) - (\tau \graft \gamma) \graft \nu = \gamma \graft (\tau \graft \nu) - (\gamma \graft \tau) \graft \nu, \]
    for~$\tau, \gamma \in \AT_D, \nu \in \AT_D \sqcup \AA_D$,
\item for any~$\tau \in \AT_D$, the map~$\tau \graft \blank : \AA_D \to \AA_D$ is a derivation and the Leibniz rule holds,
    \[ \tau \graft (\omega \cdot \nu) = (\tau \graft \omega) \cdot \nu + \omega \cdot (\tau \graft \nu), \]
    for~$\omega \in \AA_D$ and~$\nu \in \AA_D \sqcup \AT_D$.
    \item there exists a map~$t : El_{\AA_D}(\AT_D) \to \AA_D$ called a trace that satisfies~$t(\tau \graft \tilde\gamma) = \tau \graft t(\tilde\gamma)$ and~$t(\tilde\tau \circ \tilde\gamma) = t(\tilde\gamma \circ \tilde\tau)$ with~$\tau \in \AT_D, \tilde\tau, \tilde\gamma \in El_{\AA_D} (\AT_D)$ and~$\circ$ the composition of endomorphisms. The divergence is then defined as~$d(\tau) := t(\blank \graft \tau)$.
\end{enumerate} 
More details can be found in~\cite{Floystad20tup} where it is proven that~$(\AT_D, \AA_D)$ is a free tracial pre-Lie-Rinehart algebra.
We extend the structure of the tracial pre-Lie-Rinehart algebra by considering a symmetric~$\AA_D$-bilinear form~$\langle \blank, \blank \rangle : \AT_D \otimes \AT_D \to \AA_D$ with the Leibniz rule:
\[ \tau \graft \langle \gamma, \nu \rangle = \langle \tau \graft \gamma, \nu \rangle + \langle \gamma, \tau \graft \nu \rangle, \quad \text{for } \tau, \gamma, \nu \in \AT_D. \]
The aroma~$\langle \gamma, \nu \rangle \in \AA_D$ is called a stolon and is denoted by a horizontal double edge that connects the roots of the corresponding trees, for example,~$\langle \forest{(b),b}, \forest{b[b]} \rangle = \forest{(b),b=b[b]}$.
\begin{proposition}
\label{proposition:freeness}
  The tracial stolonic pre-Lie-Rinehart algebra~$(\AT_D, \AA_D)$ is free.
\end{proposition}

The classical Guin-Oudom process~\cite{OudomGuin} extends uniquely a pre-Lie product over a vector space~$V$ to a product over the symmetric algebra~$S(V)$. We generalize and use this process to extend uniquely the pre-Lie-Rinehart product~$\graft$ over a~$\AA_D$-module~$\AT_D$ to the symmetric algebra~$\AF_D := S_{\AA_D} (\AT_D)$ of decorated aromatic forests.

\begin{proposition}
\label{prop:GuinOudom-pLR}
    There exists a unique extension of the~$\graft$ product to~$\AF_D$ such that
    \begin{enumerate}
        \item[(i)]~$\omega \mathbf{1} \graft \nu = \omega \nu, \quad$ for~$\omega \in \AA_D, \nu \in \AA_D \sqcup \AT_D$,
        \item[(ii)]~$(\tau \cdot \pi) \graft \nu = \tau \graft (\pi \graft \nu) - (\tau \graft \pi) \graft \nu, \quad$ for~$\tau \in \AT_D, \pi \in \AF_D$,
        \item[(iii)]~$\pi \graft (\mu_1 \cdot \mu_2) = \sum_{(\pi)} (\pi_{(1)} \graft \mu_1) \cdot (\pi_{(2)} \graft \mu_2), \quad$ for~$\mu_1, \mu_2 \in \AF_D$,
    \end{enumerate}
    with deshuffle coproduct~$\Delta_{\AA_D}(\pi) = \sum_{(\pi)} \pi_{(1)} \otimes_{\AA_D} \pi_{(2)}$.
\end{proposition}
We follow the structure of the proof of Proposition 2.7 of~\cite{OudomGuin} which proves an analogous statement for~$S(\AT_D)$. We check that the relations~$(i),(ii),(iii)$ are well-defined over~$S_{\AA_D}(\AT_D)$.
\begin{proof}
    It follows from~$(i)$,~$(iii)$, and the coassociativity of~$\Delta_{\AA_D}$ that
    \[ \tau \graft \mathbf{1} = 0, \quad \text{and} \quad \tau \graft (\pi_1 \cdots \pi_n) = \sum_{k=1}^n \pi_1 \cdots (\tau \graft \pi_k) \cdots \pi_n. \]
    The relation~$(ii)$ is well-defined with respect to the choice of~$\tau$ using the Lemma 2.5 of~\cite{OudomGuin} which is based on induction on the length of the monomial and the pre-Lie relation. This means that~$(\tau \cdot \pi) \graft \nu$ is well-defined for~$\tau \cdot \pi \in S(\AT_D)$ using~$(i)$ and~$(ii)$. Let~$\JJ$ be an ideal of~$S(\AT_D)$
    \[ \JJ := \langle (\omega \tau \cdot \gamma - \tau \cdot \omega \gamma) \cdot \pi : \omega \in \AA_D, \tau,\gamma \in \AT_D, \pi \in S(\AT_D)\rangle. \] 
    It remains to show that~$\JJ \graft \nu = 0$ which follows from the property~$(\omega \tau \cdot \pi) \graft \nu = \omega (\tau \cdot \pi \graft \nu)$, with~$\omega \in \AA_D$, proved by induction on the length of the monomial~$\pi$. The initial step is shown below for~$\gamma \in \AT_D$ using the~$\AA_D$-linearity in the left operand of~$\graft$:
    \[ (\omega \tau \cdot \gamma) \graft \nu = \omega \tau \graft (\gamma \graft \nu) - (\omega \tau \graft \gamma) \graft \nu = \omega (\tau \cdot \gamma \graft \nu). \]
    Assume the property to be true for monomials shorter than~$\pi = \tau_1 \cdots \tau_n \in S(\AT_D)$ and recall that
    \[ \omega \tau \graft \pi = \sum_{k=1}^n \omega (\tau \graft \tau_k) \cdot \pi_{\hat{k}}, \quad \text{where } \pi_{\hat{k}} := \tau_1 \cdots \tau_{k-1} \tau_{k+1} \cdots \tau_n.  \]
    Then, by induction, we have,
\begin{align*}
(\omega \tau \graft \pi) \graft \nu &= \big( \sum_{k=1}^n \omega (\tau \graft \tau_k) \cdot \pi_{\hat{k}} \big) \graft \nu \\&= \omega \big( \sum_{k=1}^n (\tau \graft \tau_k) \cdot \pi_{\hat{k}} \big) \graft \nu \\&= \omega ((\tau \graft \pi) \graft \nu).
\end{align*}
    This allows us to prove the inductive step:
    \[ (\omega \tau \cdot \pi) \graft \nu = \omega \tau \graft (\pi \graft \nu) - (\omega \tau \graft \pi) \graft \nu = \omega (\tau \cdot \pi \graft \nu). \]
    Therefore, relations~$(i)$ and~$(ii)$ extend~$\graft$ to~$S_{\AA_D}(\AT_D) \otimes (\AA_D \oplus \AT_D) \to (\AA_D \oplus \AT_D)$. Due to the Leibniz rule, the property~$(\omega \pi) \graft \nu = \omega (\pi \graft \nu)$, the cocommutativity and coassociativity of~$\Delta_{\AA_D}$, the relation~$(iii)$ is well-defined. Therefore, it defines~$\graft$ on~$\AF_D = S_{\AA_D}(\AT_D)$.
\end{proof}

Following the D-algebra structure from~\cite{Munthe-Kaas-Wright-Dalgebra}, see also~\cite{Lundervold11hao,Lundervold13bea,Lundervold15oas}, we define the tracial commutative D-algebra~$(\AF_D,\graft)$ of decorated aromatic forests graded by the number of roots.

\begin{definition}
    Let~$(A, \cdot)$ be a unital commutative graded algebra with unit~$\mathbf{1}$. Let~$A$ be equipped with a non-associative product~$\graft$ and let~$x \in A_1$ and~$a, b \in A$ satisfy the relation 
    \[ x \graft (a \cdot b) = (x \graft a) \cdot b + a \cdot (x \graft b). \]
    The triple~$(A, \cdot, \graft)$ is a commutative D-algebra if the following identities are satisfied
    \begin{align*}
        \mathbf{1} \graft a &= a, \\
        a \graft x &\in A_1, \\
        (\omega \cdot a) \graft b &= \omega \cdot (a \graft b), \\
        (x \cdot a) \graft b &= x \graft (a \graft b) - (x \graft a) \graft b,
    \end{align*}
    for~$\omega \in A_0, \; x \in A_1, \; a, b \in A$. It is called tracial if there exists a trace~$t : El_{A_0} (A_1) \to A_0$.
\end{definition}

For a tracial commutative D-algebra~$A$, the pair~$(A_1, A_0)$ is a tracial pre-Lie-Rinehart algebra. This implies that~$(\AF_D, \graft)$ is the free tracial commutative D-algebra since~$(\AT_D, \AA_D)$ is the free tracial pre-Lie-Rinehart algebra and~$(\AF_D, \graft)$ is obtained uniquely using the Guin-Oudom process (Proposition~\ref{prop:GuinOudom-pLR}). We extend the structure of the D-algebra~$(\AF_D, \graft)$ with the~$\AA_D$-bilinear form~$\langle \blank, \blank \rangle : \AT_D \otimes \AT_D \to \AA_D$ and see that it remains free.

A map~$\varphi : A \to A^\prime$ between two tracial commutative D-algebras~$A$ and~$A^\prime$ is a D-algebra morphism if~$\varphi(A_1) \subset A^\prime_1$ and
\begin{align*}
    \varphi (a \cdot b) = \varphi(a) \cdot \varphi(b), &\quad \varphi (a \graft b) = \varphi(a) \graft \varphi(b), \\
    \varphi (\langle x, y \rangle) = \langle \varphi(x), \varphi(y) \rangle, &\quad \varphi(t(\tilde x)) = t(\varphi(\tilde x)),
\end{align*}
for~$a,b \in A$,~$x, y \in A_1$,~$\tilde x \in El_{A_0} (A_1)$.

\begin{ex}
\label{example:D-algebra_vector_fields}
    Let~$\XX$ be the space of vector fields~$\R^d \to \R^d$. The symmetric algebra~$S_{\CC^\infty(\R^d)}(\XX)$ over the ring of~$\CC^\infty(\R^d)$ maps is a commutative tracial D-algebra and represents the space of differential operators in~$\R^d$. The non-associative product is given by the covariant derivation, for example, let~$f,g,h : \R^d \to \R^d$, then,
  \[ (f g) [h] = \sum_{i, j=1}^d f^i g^j h_{ij}, \quad \text{where } h_{ij} := \frac{\partial^2 h}{\partial x_i \partial x_j}. \]
    Divergence of a vector field is~$d(f) = \sum_{i=1}^d f^i_i$ and the bilinear product is the inner product, i.e.,~$\langle f,g \rangle = \sum_{i=1}^d f^i g^i$.
\end{ex}

We note that the algebra of endomorphisms generated by~$\pi \graft \blank : \AF_D \to \AF_D$ for $\pi \in \AF_D$ is isomorphic to the algebra~$(\AF_D, \gl)$ with~$\gl$ denoting the Grossman-Larson product defined as
\[ \pi_1 \graft (\pi_2 \graft \blank) = (\pi_1 \gl \pi_2) \graft \blank , \quad \pi_1, \pi_2 \in \AF_D. \]
Let us obtain a Grossman-Larson~$\AA/_\R$-bialgebra~\cite{RinehartBialgebra} from the commutative D-algebra~$\AF_D$.

\begin{proposition}
\label{GL-Rinehart-bialgebra}
    Let~$\epsilon_{\AA_D} : \AF_D \to \AA_D$ be a~$\AA_D$-linear map such that~$\epsilon_{\AA_D} (\mathbf{1}) = \mathbf{1}$ and $\epsilon_{\AA_D}(\pi) = 0$ for~$\pi \notin \AF_{D,0}$, then~$B_{GL} := (\AF_D, \gl, \mathbf{1}, \Delta_{\AA_D}, \epsilon_{\AA_D})$ is a~$\AA/_\R$-bialgebra~\cite{RinehartBialgebra} called Grossman-Larson~$\AA/_\R$-bialgebra. That is, it satisfies the following properties:
    \begin{itemize}
        \begin{minipage}{0.4\linewidth}
            \item[(1)]~$\epsilon_{\AA_D} (\mathbf{1}) = \mathbf{1}$,
            \item[(2)]~$\Delta_{\AA_D} (\mathbf{1}) = \mathbf{1} \otimes_{\AA_D} \mathbf{1}$,
        \end{minipage}
        \begin{minipage}{0.6\linewidth}
            \item[(3)]~$\epsilon_{\AA_D} (\pi \gl \mu) = \epsilon_{\AA_D} (\pi \gl \epsilon_{\AA_D} (\mu))$,
            \item[(4)]~$\Delta_{\AA_D} (\pi \gl \mu) = \Delta_{\AA_D} (\pi) \gl_\otimes \Delta_{\AA_D}(\mu)$,
        \end{minipage}
    \end{itemize}
    with~$(\pi_{(1)} \otimes_{\AA_D} \pi_{(2)}) \gl_\otimes (\mu_{(1)} \otimes_{\AA_D} \mu_{(2)}) = (\pi_{(1)} \gl \mu_{(1)}) \otimes_{\AA_D} (\pi_{(2)} \gl \mu_{(2)})~$.
\end{proposition}
\begin{proof}
    Since Grossman-Larson product is isomorphic to the composition of maps, it is associative. The properties~$(1)$ and~$(2)$ follow from the definition of the counit and coproduct. Property~$(4)$ can be proven using the definition of the Grossman-Larson product together with the relation~$(iii)$ from Proposition~\ref{prop:GuinOudom-pLR}. To prove the property~$(3)$, we note that~$\epsilon_{\AA_D}(\pi_1 \gl \pi_2)$ is non-zero if and only if~$\pi_2 \in \AA_D$ and is equal to the aromas obtained by grafting all trees of~$\pi_1$ onto~$\pi_2$ in all possible ways.
\end{proof}

Proposition~\ref{GL-Rinehart-bialgebra} can be applied to a general commutative D-algebra if the associativity of the Grossman-Larson product is proven analogously to~\cite[Lemma 2.10]{OudomGuin}. We note that if we exclude all aromas, i.e.~$\AA_D = \{ \mathbf{1}\}$, the~$\AA/_\R$-bialgebra~$B_{GL}$ becomes a graded connected bialgebra which is also a Hopf algebra.
We also note that the product~$\gl_\otimes$ is well-defined over $\Delta_{\AA_D}(\AF_D) \subset \AF_D \otimes_{\AA_D} \AF_D$, but not over~$\AF_D \otimes_{\AA_D} \AF_D$ since~$\gl$ is not~$\AA_D$-linear in its right operand.

\begin{remark}
    We say that~$B_{GL}$ is cocomplete since~$B_{GL} = \bigcup_{n=0}^\infty B_{GL,n}$ and is a free~$\AA_D$-module. An analog of the Cartier-Milnor-Moore theorem is proven in~\cite{RinehartBialgebra} which states that a cocomplete and graded projective~$\AA/_\R$-bialgebra is the universal enveloping Lie-Rinehart algebra of the Lie-Rinehart algebra of its primitive elements. The freeness of the~$\AA_D$-module~$B_{GL}$ implies its graded projectiveness, therefore, the~$\AA/_\R$-bialgebra~$B_{GL}$ is the universal enveloping Lie-Rinehart algebra of the pre-Lie-Rinehart algebra~$(\AT_D, \AA_D)$.
\end{remark}

\begin{remark}
\label{remark:GL_Hopf_algebra}
    If we replace the coalgebra structure of~$B_{GL}$ by the deshuffle coproduct $\Delta : \AF_D \to \AF_D \otimes \AF_D$ and~$\epsilon : \AF_D \to \R$, then we get a Grossman-Larson Hopf algebra dual (up to the symmetry coefficients) to the Hopf algebra mentioned in~\cite[Thm.\ts 4.4]{Bogfjellmo19aso}.
\end{remark}

$\AA/_\R$-bialgebras are also called cocommutative bialgebroids. We show that~$B_{GL}$ defines a Hopf algebroid as introduced in~\cite{Lu_Hopf_algebroids}.
\begin{proposition}
\label{prop:antipode}
    Let~$S_\gl : (\AF_D, \gl) \to (\AF_D, \gl)$ be the algebra anti-isomorphism defined as
    \begin{itemize}
        \item[(i)]~$S_\gl(\omega) := \omega, \quad S_\gl(\omega \pi) := S_\gl(\pi) \gl \omega, \quad$ for~$\omega \in \AA_D, \pi \in \AF_D$,
        \item[(ii)]~$S_\gl(\tau) := - \tau, \quad S_\gl(\tau \pi) := -S_\gl(\pi) \gl \tau - S_\gl(\tau \graft \pi), \quad$ for~$\tau \in \TT_D$.
    \end{itemize}
    Then,~$H_{GL} := (B_{GL}, S_\gl)$ is the Grossman-Larson Hopf algebroid with~$S_\gl$ being called an antipode and satisfying the following conditions where~$\pi \in \AF_D$ and~$\Delta_{\AA_D} (\pi) = \sum_{(\pi)} \pi_{(1)} \otimes_{\AA_D} \pi_{(2)}$,
    \begin{itemize}
        \item[(1)]~$\sum_{(\pi)} S_\gl(\pi_{(1)}) \gl \pi_{(2)} = \mathbf{1}\epsilon_{\AA_D} (S_\gl(\pi))$,
        \item[(2)]~$\sum_{(\pi)} \hat{\pi}_{(1)} \gl S_\gl(\hat{\pi}_{(2)}) = \mathbf{1}\epsilon_{\AA_D} (\pi)$, with~$\gamma (\pi_{(1)} \otimes_{\AA_D} \pi_{(2)}) = \hat{\pi}_{(1)} \otimes \hat{\pi}_{(2)}$,
    \end{itemize}
    where~$\gamma$ is the section of the projection~$P : \AF_D \otimes \AF_D \to \AF_D \otimes_{\AA_D} \AF_D$ that places all aromas on the left side of the tensor product, that is,~$\hat{\pi}_{(2)}$ has no aromas.
\end{proposition}
\begin{proof}
    It can be seen that~$S_\gl$ defined this way is an anti-isomorphism due to the associativity of the Grossman-Larson product and the definition which can be rewritten as
    \[ S_\gl(\omega \gl \pi) = S_\gl(\pi) \gl S_\gl(\omega), \quad S_\gl(\tau \gl \pi) = S_\gl(\pi) \gl S_\gl(\tau). \]
    Let us prove that~$S_\gl$ satisfies the condition~$(1)$ and~$(2)$. We start with~$(1)$ by noting that the operation $S_\gl(\blank) \gl \blank : \AF_D \otimes_{\AA_D} \AF_D \to \AF_D$ is well-defined since 
    \[ S_\gl(\omega \pi_{(1)}) \gl \pi_{(2)} = S_\gl(\pi_{(1)}) \gl \omega \gl \pi_{(2)} = S_\gl(\pi_{(1)}) \gl \omega \pi_{(2)}. \] 
    We see that~$(1)$ is satisfied for~$\pi = \omega \in \AA_D$. Let us check that~$(1)$ is satisfied for~$\pi = \omega \tau \in \AT_D$:
    \[ S_\gl(\tau) \gl \omega + S_\gl(\mathbf{1}) \gl \omega \tau = -\tau \omega - \tau \graft \omega + \omega \tau = S_\gl(\tau) \graft \omega = \mathbf{1}\epsilon_{\AA_D} (S_\gl(\omega \tau)), \]
    where~$\tau \in \TT_D$. Let~$\tau \in \TT_D, \pi \in \FF_D, \omega \in \AA_D$, then,~$\tau \pi \omega = \tau \gl \pi \omega - \tau \graft \pi \omega$. We use induction on the length of~$\pi$ and assume that the condition~$(1)$ is satisfied for~$\tau \graft \pi \omega$. We check that the left-hand side of~$(1)$ applied to~$\tau \gl \pi \omega$ is~$0$:
    \begin{align*}
        \sum_{(\tau \gl \pi \omega)} S_\gl((\tau \gl \pi \omega)_{(1)}) \gl (\tau \gl \pi \omega)_{(2)} &=  \sum_{(\tau),(\pi \omega)} S_\gl(\tau_{(1)} \gl \pi_{(1)}) \gl \tau_{(2)} \gl \pi_{(2)} \omega \\
        &= \sum_{(\pi \omega)} S_\gl(\pi_{(1)}) \gl \big(\sum_{(\tau)} S_\gl(\tau_{(1)}) \gl \tau_{(2)} \big) \gl \pi_{(2)} \omega = 0,
    \end{align*}
    since~$\sum_{(\tau)} S_\gl(\tau_{(1)}) \gl \tau_{(2)} = 0$. This implies that left-hand side of~$(1)$ applied to~$\tau \pi \omega$ is
    \begin{align*}
        \sum_{(\tau \pi \omega)} S_\gl((\tau \pi \omega)_{(1)}) \gl (\tau \pi \omega)_{(2)} &= - \mathbf{1} \epsilon_{\AA_D} (S_\gl(\tau \graft \pi \omega)) \\
        &= \mathbf{1} \epsilon_{\AA_D} (S_\gl(\tau \pi \omega)) + \mathbf{1} \epsilon_{\AA_D} (S_\gl(\pi \omega) \gl \tau)\\
        &= \mathbf{1}\epsilon_{\AA_D} (S_\gl(\tau \pi \omega)).
    \end{align*}
    This proves that~$S_\gl$ satisfies the condition~$(1)$. To prove condition~$(2)$, we recall that the Grossman-Larson product is~$\AA_D$-linear in its left operand and that~$\gamma$ places all aromas on the left side of the tensor product. This implies that both sides of the condition~$(2)$ are~$\AA_D$-linear and the condition is reduced to the analogous condition over the Grossman-Larson Hopf algebra over~$\FF_D$. This proves that~$S_\gl$ satisfies the condition~$(2)$.
\end{proof}

The space of decorated aromatic forests with commutative product~$\cdot$ and deshuffle coproduct forms a Hopf algebra~$H := (\AF_D, \cdot, \mathbf{1}, \Delta_{\AA_D}, \epsilon_{\AA_D}, S)$. The Hopf algebra~$H$ together with the product~$\graft : H \otimes H \to H$ forms a pre-Hopf algebroid which is a generalization of the pre-Hopf algebra~\cite{post-Hopf-algebra} that satisfies the following conditions for~$\pi, \mu, \eta \in H$, a map~$\beta : H \to H$, and a section~$\gamma$ of the projection~$P : H \otimes H \to H \otimes_{\AA_D} H$ with~$\gamma(\pi_{(1)} \otimes_{\AA_D} \pi_{(2)}) = \hat{\pi}_{(1)} \otimes \hat{\pi}_{(2)}$:
\begin{itemize}
    \item[(1)]~$\pi \graft (\mu \cdot \eta) = (\pi_{(1)} \graft \mu) \cdot (\pi_{(2)} \graft \eta)$,
    \item[(2)]~$\pi \graft (\mu \graft \eta) = (\pi_{(1)} \cdot (\pi_{(2)} \graft \mu)) \graft \eta$,
    \item[(3)]~$\hat{\pi}_{(1)} \graft (\beta(\hat{\pi}_{(2)}) \graft \blank) = \epsilon_{\AA_D}(\pi) \id$,
    \item[(4)]~$\beta(\pi_{(1)}) \graft (\pi_{(2)} \graft \blank) = \epsilon_{\AA_D}(\beta(\pi)) \id$.
\end{itemize}
Conditions~$(1)$ and~$(2)$ follow from the definition of the D-algebra and the conditions~$(3)$ and~$(4)$ are satisfied for~$\beta$ and~$\gamma$ being the anitipode~$S_\gl$ and~$\gamma$ from Proposition~\ref{prop:antipode}. We present an alternative proof for~$(8)$ of~\cite[Lemma 2.3]{post-Hopf-algebra}.
\begin{lemma}
    For all~$\pi, \mu \in H$, we have~$\pi \graft S(\mu) = S(\pi \graft \mu)$.
\end{lemma}
\begin{proof}
    Consider~$\mu = \tau \in \AT_D$, then~$\pi \graft \tau \in \AT_D$ and~$\pi \graft S(\tau) = - \pi \graft \tau = S(\pi \graft \tau)$. Assume the statement is true for all monomials shorter than~$\mu = \tau_1 \cdots \tau_n$ for~$\tau_i \in \AT_D$, then,
    \begin{align*}
        S(\pi \graft \tau_1 \cdots \tau_n) &= - \sum_{(\pi)} (\pi_{(1)} \graft \tau_1) S(\pi_{(2)} \graft \tau_2 \cdots \tau_n) \\
        &= - \sum_{(\pi)} (\pi_{(1)} \graft \tau_1) (\pi_{(2)} \graft S(\tau_2 \cdots \tau_n)) \\
        &= - \pi \graft (\tau_1 S(\tau_2 \cdots \tau_n)) = \pi \graft S(\mu).
    \end{align*}
    This finishes the proof.
\end{proof}

We check that the subadjacent Hopf algebra with antipode~$\hat{S}_\gl$ defined in~\cite[Thm.\ts 2.4]{post-Hopf-algebra} corresponds to the Grossman-Larson Hopf algebroid by showing that the antipodes coincide.
\begin{lemma}
  Let~$\hat{S}_\gl$ be the antipode defined in~\cite[Thm.\ts 2.4]{post-Hopf-algebra}, that is,
    \[ \hat{S}_\gl (\pi) := \sum_{(\pi)} \hat{S}_\gl (\pi_{(1)}) \graft S(\pi_{(2)}), \quad \text{for } \pi \in H. \]
    Then,~$\hat{S}_\gl = S_\gl$ where~$S_\gl$ is defined in Proposition~\ref{prop:antipode}.
\end{lemma}
\begin{proof}
  We see that~$\hat{S}_\gl (\omega) = \omega$ and~$\hat{S}_\gl (\tau) = -\tau$ for~$\omega \in \AA_D$ and~$\tau \in \TT_D$. We check by induction and using the fact that~$S_\gl$ is a coalgebra homomorphism that~$\hat{S}_\gl (\omega \pi) = S_\gl (\omega \pi)$ for~$\pi \in \AF_D$:
    \begin{align*}
        \hat{S}_\gl (\omega \pi) &= \sum_{(\pi)} S_\gl (\pi_{(1)}) \graft S(\omega \pi_{(2)}) = \sum_{(\pi)^2} (S_\gl (\pi_{(1)}) \graft \omega) (S_\gl (\pi_{(2)}) \graft S(\pi_{(3)})) \\
        &= \sum_{(\pi)} (S_\gl (\pi_{(1)}) \graft \omega) S_\gl (\pi_{(2)}) = S_\gl (\pi) \gl \omega.
    \end{align*}
    We use the same properties to show that~$\hat{S}_\gl (\tau \pi) = S_\gl (\tau \pi)$:
    \begin{align*}
        \hat{S}_\gl (\tau \pi) &= S_\gl (\pi_{(1)}) \graft S(\tau \pi_{(2)}) + S_\gl (\tau \pi_{(1)}) \graft S(\pi_{(2)}) \\
        &= - S_\gl (\pi_{(1)}) \graft \tau S(\pi_{(2)}) - (S_\gl (\pi_{(1)}) \gl \tau) \graft S(\pi_{(2)}) - S_\gl(\tau \graft \pi_{(1)}) \graft S(\pi_{(2)}) \\
        &= - S_\gl (\pi) \gl \tau - S_\gl (\pi_{(1)}) \graft S(\tau \graft \pi_{(2)}) - S_\gl(\tau \graft \pi_{(1)}) \graft S(\pi_{(2)}) \\
        &= - S_\gl (\pi) \gl \tau - S_\gl (\tau \graft \pi),
    \end{align*}
    where we omit writting the sums to simplify the notation.
    Therefore,~$\hat{S}_\gl = S_\gl$ following the definition from Proposition~\ref{prop:antipode}.
\end{proof}
Following~\cite[Def.\ts 1.1]{Hopf-braces}, the Hopf algebra~$H$ together with the Hopf algebroid~$H_{GL}$ forms a generalization of the Hopf brace, that is, the following compatibility condition is satisfied
\[ \pi \gl (\mu \cdot \eta) = \sum_{(\pi)^2} (\pi_{(1)} \gl \mu) \cdot S(\pi_{(2)}) \cdot (\pi_{(3)} \gl \eta), \]
with~$\pi, \mu, \eta \in \AF_D$ and~$\Delta^2 (\pi) = \sum_{(\pi)^2} \pi_{(1)} \otimes \pi_{(2)} \otimes \pi_{(3)}$. A proof can be found in~\cite[Thm.\ts 2.13]{post-Hopf-algebra}.

\subsection{Decorated clumped forests}
\label{sec:decorated_clumped_forests}

Recall that decorated clumped forests are defined as a symmetric algebra~$\CF_D = S(\AT_D)$ over~$\R$ in Section~\ref{section:Decorated aromatic forests}.
We use the Guin-Oudom process~\cite{OudomGuin} to define the product~$\graft$ on~$\CF_D$. The commutative D-algebra~$(\CF_D, \graft)$ that we obtain in this way has~$\CF_{D,0} = \{\mathbf{1}\}$ and is in many ways similar to the commutative D-algebra of classical forests~$(\FF_D, \graft)$.
Decorated clumped forests have a convenient algebraic structure described in the following result.
\begin{theorem}
\label{thm:CF_Hopf_alg}
    The Grossman-Larson~$\AA/_\R$-bialgebroid of~$(\CF_D, \graft)$ is a Hopf algebra dual up to the symmetry to the Butcher-Connes-Kreimer Hopf algebra over clumped forests.
\end{theorem}
\begin{proof}
    As~$\CF_{D,0} = \{\one\}$, the~$\AA/_\R$-bialgebroid structure reduces to a graded connected bialgebra, that is, to a Hopf algebra. Its duality to the corresponding Butcher-Connes-Kreimer Hopf algebra over clumped forests can be seen by following the proof for classical forests from~\cite{Hoffman}.
\end{proof}
 
Let us consider the following inner product for~$\pi, \mu \in \CF_D$ or~$\AF_D$:
\[ \langle \pi, \mu \rangle_\sigma := \begin{cases} \sigma(\pi), &\text{if } \pi = \mu, \\ 0, &\text{otherwise}. \end{cases} \] 
We use it to obtain the following duality between concatenation product and deshuffle coproduct.
\begin{lemma}
\label{lemma:concat_deshuffle_duality}
    Let~$\cdot$ be the concatenation product and let~$a,b \in \CF^*_D$ or~$\AF^*_D$, then,
    \[ \delta_\sigma(a) \cdot \delta_\sigma(b) = \delta_\sigma(a \odot b), \quad \text{where } a \odot b = (a \otimes b) \circ \Delta, \]
    where~$\Delta: \CF_D \to \CF_D \otimes \CF_D$ or~$\Delta: \AF_D \to \AF_D \otimes \AF_D$ is the~$\R$-linear deshuffle coproduct.
\end{lemma}
\begin{proof}
    We check that~$\langle \pi \cdot \mu, \eta \rangle_\sigma = \langle \pi \otimes \mu, \Delta(\eta) \rangle_\sigma$ for~$\pi, \mu, \eta \in \CF_D$ or~$\AF_D$.
\end{proof}

Let the map~$\Phi : (\CF_D, \graft) \to (\AF_D, \graft)$ be a commutative D-algebra morphism that "forgets" the clumping, for example,
\[ \Phi(\forest{(b),b} \cdot \forest{b[b]}) = \Phi(\forest{b}\cdot\forest{(b),b[b]}) = \forest{(b),b,b[b]}. \]
We define~$\Phi^* : \AF_D \to \CF_D$ as~$(\Phi \circ \delta_\sigma)(a) = \delta_\sigma(a \circ \Phi^*)$ with~$a \in \CF_D^*$, in particular,~$\Phi^*$ is the adjoint of~$\Phi$ with respect to the~$\langle \blank, \blank \rangle_\sigma$ inner product.
Let us consider the exponential maps
\begin{align*}
    \exp^\cdot : \AT_D \to \AF_D, & \quad \exp^\odot : \AT_D^* \to \AF_D^*, \\
    \exp^\cdot_C : \AT_D \to \CF_D, & \quad \exp^\odot_C : \AT_D^* \to \CF_D^*.
\end{align*}
Using Lemma~\ref{lemma:concat_deshuffle_duality}, we obtain for~$a_0 \in \AT_D^*$ the following identities
\[ \exp^\cdot(\delta_\sigma(a_0)) = \delta_\sigma(\exp^\odot(a_0)), \quad \exp^\cdot_C(\delta_\sigma(a_0)) = \delta_\sigma(\exp^\odot_C(a_0)), \]
with the functionals~$\exp^\odot(a_0)$ and~$\exp^\odot_C(a_0)$ being characterized in Propositon~\ref{prop:exp_odot}.
\begin{proposition}
\label{prop:exp_odot}
    Let~$a_0 \in \AT^*$ be an infinitisimal character and let
    \[ a_e := \exp^\odot (a_0), \quad a_c := \exp^\odot_C(a_0). \]
    Then,~$a_c \in \CF_D^*$ is a character of~$(\CF_D, \cdot)$ and~$a_e = a_c \circ \Phi^*$ with~$a_e \in \AF_D^*$.
\end{proposition}
\begin{proof}
    We prove that~$a_c := \exp^\odot_C (a_0)$ is a character of~$\CF_D$ by considering a decorated clumped forest~$\pi = \tau_1 \cdots \tau_n \in \CF_D$ where~$\tau_i \in \AT_D$. Then,
    \begin{align*}
        \exp_C^\odot(a_0) (\pi) &= \frac{1}{n!} \cdot (a_0 \otimes \cdots \otimes a_0) (\Delta^{n-1} (\pi)) \\
        &= \frac{1}{n!} \cdot \sum_{\sigma \in S_n} a_0(\tau_{\sigma(1)}) \cdots a_0 (\tau_{\sigma(n)}) = a_0(\tau_1) \cdots a_0(\tau_n).
    \end{align*}
    We prove~$a_e = a_c \circ \Phi^*$ by using the identity~$\exp^\cdot(\delta_\sigma(a_0)) = \Phi \big( \exp^\cdot_C(\delta_\sigma(a_0)) \big)$.
\end{proof}

We define~$F_D$ from Section~\ref{sec:S-series} over~$\CF_D$ by~$F_D := F_D \circ \Phi$ where we use the same notation for the morphism over clumped and aromatic forests. This way, we obtain S\nobreakdash-series over decorated clumped forests. We note that, following the definition of~$\Phi^*$, any S\nobreakdash-series~$S(a)$ with~$a \in \CF_D^*$ is identical to the S\nobreakdash-series over decorated aromatic forests~$S(a \circ \Phi^*)$. Moreover, given any functional~$a \in \AF_D^*$, there exists a functional~$a_C \in \CF_D^*$ such that~$a = a_C \circ \Phi^*$, since~$\Phi^*$ is injective. A possible definition of~$a_C$ is
\[    a_C(\pi) = \frac{1}{n^m} a(\Phi(\pi)), \quad \text{for } \pi \in \CF_D,\]
where~$n$ is the number of rooted components and~$m$ is the number of aromas.

We note that a tree~$\tau \in \TT_D$ induces a map~$\tau^\prime : \AA_D \to \AT_D$ with~$\tau^\prime(\omega) = \omega \tau$ for~$\omega \in \AA_D$. We can extend~$(\blank)^\prime$ to~$\FF_D$ in two possible ways, that is, for a~$\pi \in \FF_D$, we have two maps,
\[ \pi^\prime : \AA_D \to \AF_D, \quad \pi^{\prime \prime} : \AA_D \to \CF_D, \]
defined as
\[ (\pi \cdot \mu)^\prime(\omega) = \omega \pi \mu, \quad (\pi \cdot \mu)^\dprime (\omega) = (\pi^\dprime \odot \mu^\dprime) (\omega), \]
with~$\pi^\dprime \odot \mu^\dprime := \cdot \circ (\pi^\dprime \otimes \mu^\dprime) \circ \Delta$ where~$\Delta$ is a deshuffle coproduct over~$\AA_D$. For example,
\[ (\forest{b,b[b]})^\prime (\forest{(b)}) = \forest{(b),b,b[b]}, \quad (\forest{b,b[b]})^{\dprime} (\forest{(b)}) = \forest{(b),b} \cdot \forest{b[b]} + \forest{b} \cdot \forest{(b),b[b]}. \]
The following result presents a convenient method to compute~$\Phi^*$.
\begin{proposition}
\label{prop:Phi_duality}
    Given a decorated forest~$\pi \in \FF_D$ and~$\omega \in \AA_D$, we have the following identity
    \[ \Phi^*(\pi^\prime(\omega)) = \pi^\dprime(\omega). \]
\end{proposition}
\begin{proof}
  We take the adjoints of~$\pi^\prime$ and~$\pi^{\prime \prime}$ with respect to the~$\langle \blank, \blank \rangle_\sigma$ inner product and denote them by~$\pi^{\prime*} : \AF_D \to \AA_D$ and~$\pi^{\dprime*} : \CF_D \to \AA_D$. Then,~$\pi^{\prime*} (\omega \pi) = \sigma(\pi) \omega$ where~$\omega \in \AA_D$ and $\pi \in \FF_D$. To prove the statement, we have to show that~$\pi^{\dprime*}(\pi_\omega) = \sigma(\pi) \omega$ where~$\pi_\omega \in \CF_D$ is a decorated clumped forest that occurs as a term in~$\pi^\dprime(\omega)$, that is, we need to show that $\pi^{\prime *}(\Phi(\pi_\omega)) = \pi^{\dprime*} (\pi_\omega)$.

    We note that for~$\tau \in \TT_D$,~$\tau^{\dprime*}(\omega) = \tau^{\prime*}(\omega) = \sigma(\tau) \omega$ and the statement is true. We use an inductive assumption and Lemma~\ref{lemma:concat_deshuffle_duality} to obtain
    \[ (\pi \cdot \mu)^{\dprime*} (\eta_\omega) = (\cdot \circ (\pi^{\dprime*} \otimes \mu^{\dprime*}) \circ \Delta) (\eta_\omega) = \sigma(\pi) \sigma(\mu) |A| \omega, \]
    where~$A$ is the set of~$\eta_{\omega,(1)} \otimes \eta_{\omega,(2)}$ such that the rooted components of~$\eta_{\omega, (1)}$ and~$\eta_{\omega, (2)}$ are isomorphic to~$\pi$ and~$\mu$, respectively. We note that~$\sigma(\pi) \sigma(\mu) |A| = \sigma(\pi \cdot \mu)$ and the proof is finished.
\end{proof}
 For example, let~$\forest{(b),b=b,b,b[b]} \in \AF_D$, then,
\[ \Phi^* (\forest{(b),b=b,b,b[b]}) = \forest{(b),b=b,b} \cdot \forest{b[b]} + \forest{b=b,b} \cdot \forest{(b),b[b]} + \forest{(b),b} \cdot \forest{b=b,b[b]} + \forest{b} \cdot \forest{(b),b=b,b[b]}. \]
We note that Propositions~\ref{prop:exp_odot} and~\ref{prop:Phi_duality} imply the following identity
\[ a_e (\omega \pi) = a_c (\pi^{\dprime} (\omega)), \quad \text{for } \pi \in \FF_D, \quad \omega \in \AA_D, \]
a version of which is presented in~\cite{Bogfjellmo19aso} for classical aromatic forests.

\subsection{Substitution law for decorated aromatic S-series}
\label{sec:substitution_law}

We introduce the decorated insertion product~$\ins_d : \AT_D \otimes \AT_D \to \AT_D$ which inserts the aromatic tree from the left operand into the vertices decorated by~$d \in D$ of the right operand in all possible ways. For example, we recall that~$\forest{(w),b[w[w]]} = \text{div} (\circ) (\circ \graft \circ) \graft \bullet$, then,
\begin{align*}
    \forest{(b),b[w]} \ins_\circ \forest{(w),b[w[w]]} &= \text{div} (\forest{(b),b[w]}) (\circ \graft \circ) \graft \bullet + \text{div} (\circ) (\forest{(b),b[w]} \graft \circ) \graft \bullet + \text{div} (\circ) (\circ \graft \forest{(b),b[w]}) \graft \bullet \\
    &= \forest{(b[b[w]]),b[w[w]]} + \forest{(b),(b[w]),b[w[w]]} + \forest{(b),(b,w),b[w[w]]} + \forest{(b),b[w[b[w]]]} + \forest{(b),b[b[w[w]]]} + \forest{(b),b[b[w,w]]} + \forest{(b[w]),b[b[w]]}.
\end{align*}
The product~$\ins_d$ generalizes the insertion product studied in~\cite{Saidi10oap,Manchon11lpl,Saidi11adh}. Then, the family $(\AT_D, (\ins_d)_{d \in D})$ is a multi-pre-Lie algebra~\cite{BrunedManchon22AlgDefSPDE,Foissy21aso}, that is, for~$\tau, \gamma, \nu \in \AT_D$ and~$d,e \in D$ we have
\[ \tau \ins_d (\gamma \ins_e \nu) - (\tau \ins_d \gamma) \ins_e \nu = \gamma \ins_e (\tau \ins_d \nu) - (\gamma \ins_e \tau) \ins_d \nu. \]

Let~$\AT_D^{\oplus D} := \AT_D \otimes \R[D] = \bigoplus_{d \in D} \AT_D \iota_d$ where~$\iota_d$ for~$d \in D$ form the basis of~$\R[D]$ and let us define the action~$\ins : \AT_D^{\oplus D} \otimes \AT_D \to \AT_D$ by
\[ \tau \iota_d \ins \gamma = \tau \ins_d \gamma. \]
Let us consider~$S(\AT_D^{\oplus D})$ which becomes~$\CF_D^{\otimes D} := \bigotimes_{d \in D} \CF_D \iota_d$ after we assume the identity $\pi \iota_d \cdot \mu \iota_d = (\pi \cdot \mu) \iota_d$. 
We check that the Guin-Oudom process for multi-pre-Lie products~\cite[Thm.\ts 2.4]{Foissy21aso} is well-defined and use it to define
\[ \ins : \CF_D^{\otimes D} \otimes \CF_D \to \CF_D.\] 
For example, let~$D := \{\bullet, \circ \}$, then,~$(\forest{b[b]} \iota_\circ) \cdot (\forest{w,w} \iota_\bullet) \ins \forest{w[b,b]} = 2 \forest{b[w,w,b]} + 4 \forest{b[w,b[w]]} + 2 \forest{b[b[w,w]]}$.
We note that for~$\pi, \mu \in \CF_D$, we have~$\pi \iota_d \ins \mu = 0$ if the number~$|\pi|_{AT}$ of aromatic trees in~$\pi$ is greater than the number~$|\mu|_d$ of vertices decorated by~$d$ in~$\mu$. We define now the substitution action $\subs : \CF_D^{\otimes D} \otimes \CF_D \to \CF_D$ in the following way
\[ (\otimes_{d \in D} \pi_d \iota_d) \subs \mu := \begin{cases} 
  (\otimes_{d \in D} \pi_d \iota_d) \ins \mu, &\text{if } |\pi_d|_{AT} = |\mu|_d \text { for all } d \in D, \\ 
  0, &\text{otherwise.} 
\end{cases} \]
The substitution action substitutes all vertices in the right operand by the aromatic trees from the left operand. 
We recall that~$CF_D$ is a free D-algebra generated by the set~$D$, therefore, given a map~$\varphi : D \to \AT_D$, there exists a unique morphism~$A_\varphi : \CF_D \to \CF_D$. The morphism~$A_\varphi$ can be written using the substitution action.

\begin{lemma}
\label{lemma:Aphi_exp}
  Given a morphism~$A_\varphi : \CF_D \to \CF_D$ that acts on generators~$D$ as~$\varphi$, we have
  \[ A_\varphi (\pi) = \otimes_{d\in D} \exp^\cdot_C(\varphi(d) \iota_d) \subs \pi. \]
\end{lemma}
\begin{proof}
  Assume~$\varphi(d) = \tau_d \in \AT_D$, then,
  \[ \otimes_{d\in D} \exp^\cdot_C(\tau_d \iota_d) \subs \pi = \frac{1}{\prod_{d \in D}|\pi|_d!} (\otimes_{d \in D} \tau_d^{|\pi|_d} \iota_d) \ins \pi = A_\varphi(\pi), \]
  where we use the definition of~$\ins$.
\end{proof}

Let~$a^{\otimes D} := \otimes_{d \in D} a_d \iota_d$ with~$a_d \in \CF^*_D$ be a functional over~$\CF_D^{\otimes D}$ defined as
\[ a^{\otimes D} (\otimes_{d \in D} \pi_d \iota_d) = \prod_{d \in D} a_d (\pi_d). \]
For example, let~$D = \{\bullet, \circ\}$, then,
$(a\iota_\bullet b\iota_\circ) \big(\forest{b[b,w[b]],w[b]}\iota_\bullet \cdot \forest{b,b[w]}\iota_\circ \big) = a(\forest{b[b,w[b]],w[b]}) b(\forest{b,b[w]})$.
We define a coaction $\Delta_{CEM} : \CF_D \to \CF_D^{\otimes D} \otimes \CF_D$ and show in Proposition~\ref{prop:subs_cemdual} that it is dual to the substitution action up to the symmetry coefficient, that is, we have
\[ \delta_\sigma(b^{\otimes D} \star a) = \delta_\sigma(b^{\otimes D}) \subs \delta_\sigma(a) \, \quad \text{where } b^{\otimes D} \star a = (b^{\otimes D} \otimes a) \circ \Delta_{CEM}, \]
for~$a \in \CF_D^*$,~$b^{\otimes D} \in \CF_D^{\otimes D *}$, and where we denote~$\sigma^{\otimes D}$ by~$\sigma$ for simplicity. The coaction~$\Delta_{CEM}$ is a generalization of the Calaque--Ebrahimi-Fard--Manchon (CEM) coproduct from~\cite{Calaque11tih}.

\begin{definition}
  Define the CEM coaction~$\Delta_{CEM} : \CF_D \to \CF_D^{\otimes D} \otimes \CF_D$ as
    \[ \Delta_{CEM} (\pi) = \sum_{\otimes_{d \in D} p_d \iota_d} (\otimes_{d \in D} p_d \iota_d) \otimes \pi/_{\otimes_{d \in D} p_d \iota_d}, \quad \text{for } \pi \in \CF_D, \]
    where the sum is over all monomials of disjoint decorated clumped subforests~$p_d \iota_d$ that partition~$\pi$, and~$\pi/_{\otimes_{d \in D} p_d \iota_d}$ denotes the decorated clumped forest obtained by contracting the decorated aromatic trees of~$p_d \iota_d$ into vertices decorated by~$d$. 
\end{definition}
For example, let~$D = \{\bullet, \circ \}$, then,
\begin{align*}
    \Delta_{CEM} (\forest{b}) &= \forest{b} \iota_\bullet \otimes \forest{b} + \forest{b} \iota_\circ \otimes \forest{w}, \\
    \Delta_{CEM} (\forest{b[w]}) &= \forest{b[w]} \iota_\bullet \otimes \forest{b} + \forest{b[w]} \iota_\circ \otimes \forest{w} + \forest{b,w} \iota_\bullet \otimes \forest{b[b]} + \forest{b,w} \iota_\circ \otimes \forest{w[w]} \\
    &\quad \forest{b} \iota_\bullet \cdot \forest{w} \iota_\circ \otimes \forest{b[w]} + \forest{b} \iota_\circ \cdot \forest{w} \iota_\bullet \otimes \forest{w[b]}, \\
    \Delta_{CEM} (\forest{b[b,w]}) &= \forest{b[b,w]} \iota_\bullet \otimes \forest{b} + \forest{b[b],w} \iota_\bullet \otimes \forest{b[b]} + \forest{b[w],b} \iota_\bullet \otimes \forest{b[b]} + \forest{b,b,w} \iota_\bullet \otimes \forest{b[b,b]} + \\
    &\quad \forest{b[b,w]} \iota_\circ \otimes \forest{w} + \forest{b[b],w} \iota_\circ \otimes \forest{w[w]} + \forest{b[w],b} \iota_\circ \otimes \forest{w[w]} + \forest{b,b,w} \iota_\circ \otimes \forest{w[w,w]} + \\
    &\quad \forest{b[b]} \iota_\bullet \cdot \forest{w} \iota_\circ \otimes \forest{b[w]} + \forest{b[b]} \iota_\circ \cdot \forest{w} \iota_\bullet \otimes \forest{w[b]} + \forest{b[w]} \iota_\bullet \cdot \forest{b} \iota_\circ \otimes \forest{b[w]} + \forest{b[w]} \iota_\circ \cdot \forest{b} \iota_\bullet \otimes \forest{w[b]} + \\
    &\quad \forest{b} \iota_\circ \cdot \forest{b,w} \iota_\bullet \otimes \forest{b[w,b]} + \forest{b} \iota_\circ \cdot \forest{b,w} \iota_\bullet \otimes \forest{w[b,b]} + \forest{b,b} \iota_\bullet \cdot \forest{w} \iota_\circ \otimes \forest{b[b,w]} + \\
    &\quad \forest{b} \iota_\bullet \cdot \forest{b,w} \iota_\circ \otimes \forest{w[b,w]} + \forest{b} \iota_\bullet \cdot \forest{b,w} \iota_\circ \otimes \forest{b[w,w]} + \forest{b,b} \iota_\circ \cdot \forest{w} \iota_\bullet \otimes \forest{w[w,b]}.
\end{align*}

\begin{proposition}
\label{prop:subs_cemdual}
  Let~$a \in \CF^*_D$,~$b^{\otimes D} \in \CF_D^{\otimes D *}$, and~$b^{\otimes D} \star a = (b^{\otimes D} \otimes a) \circ \Delta_{CEM}$, then,
    \[ \delta_\sigma(b^{\otimes D} \star a) = \delta_\sigma(b^{\otimes D}) \subs \delta_\sigma(a). \]
\end{proposition}
\begin{proof}
    We prove the statement in two steps. First, we prove the statement over the space of ordered clumped forests. Ordered clumped forest are forests in which all vertices are totally ordered. Next, we describe the relationships between the product and coproduct over ordered and non-ordered clumped forests and use them to finish the proof.

    \textbf{1.} The space~$\CF_{DO}$ is defined by assigning every decorated clumped forest~$\pi$ a total order over the vertices. We note that a forest~$\pi \in \CF_D$ corresponds to~$|\pi|!$ ordered forests in~$\CF_{DO}$. The symmetry of any element of~$\CF_{DO}$ is equal to~$1$. Let~$\CF_{DO}^{\otimes D} := \bigotimes_{d \in D} \CF_{DO} \iota_d$.
    
    We define~$\subs^O : \CF_{DO}^{\otimes D} \otimes \CF_D \to \CF_{DO}$ and~$\Delta_{CEM}^O : \CF_{DO} \to \CF_{DO}^{\otimes D} \otimes \CF_D$ to be the natural extensions of~$\subs$ and~$\Delta_{CEM}$. The map~$\subs^O$ substitutes the vertices decorated by~$d$ by the trees of~$\pi_d \iota_d$ and chooses a total order between the vertices of~$\pi_d$ and~$\pi_l$ with~$d \neq l$ in all possible ways. We note that there are~$\big( \sum_{d \in D} |\pi_d| \big) ! / \prod_{d \in D} |\pi_d|!$ ways to choose a total order between the vertices of~$\pi_d$ and~$\pi_l$ with~$d \neq l$. Let~$\pi \in \CF_{DO}^{\otimes D}$ and~$\mu \in \CF_D$, then,
    \[ \pi \subs^O \mu = \sum_{\eta \in \CF_{DO}} N(\pi, \mu, \eta) \eta, \]
    where~$N(\pi, \mu, \eta)$ is the number of ways to substitute the vertices of~$\mu$ by trees in~$\pi$ to obtain~$\eta$. We note that there are~$|\Aut(\mu)|$ ways to substitute the vertices of~$\mu$ to obtain the same ordered forest~$\eta$, therefore,~$N(\pi, \mu, \eta) = \sigma(\mu)$.
    The symmetry is~$1$ for an ordered clumped forest, and, since all terms in~$\Delta_{CEM}^O (\pi)$ have coefficient~$1$, we obtain
    \[ \delta_\sigma^O (b^{\otimes D}) \subs^O \delta_\sigma (a) = \delta_\sigma^O (b^{\otimes D} \star^O a), \quad \text{where } b^{\otimes D} \star^O a = (b^{\otimes D} \otimes a) \circ \Delta_{CEM}^O, \]
    for~$b^{\otimes D} \in \CF_{DO}^{\otimes D *}$ and~$a \in \CF_D^*$. The image of~$\delta_\sigma^O$ is a formal sum over ordered clumped forests.

    \textbf{2.} We define the map~$\varphi$ that forgets the ordering of the vertices and let 
    \[ \hat{\varphi} (\pi) := \frac{\varphi(\pi)}{|\pi|!}, \quad \varphi^{\otimes D}(\otimes_{d \in D} \pi_d \iota_d) = \otimes_{d \in D} \varphi(\pi_d) \iota_d, \] 
    where~$\varphi, \hat\varphi : \CF_{DO} \to \CF_D$ and~$\varphi^{\otimes D}, \hat\varphi^{\otimes D} : \CF_{DO}^{\otimes D} \to \CF_D^{\otimes D}$. We have the following properties
    \[ \hat\varphi \circ \subs^O = \subs \circ (\hat\varphi^{\otimes D} \otimes \id), \quad (\varphi^{\otimes D} \otimes \id) \circ \Delta_{CEM}^O = \Delta_{CEM} \circ \varphi.  \]
    We also note that for a functional~$b^{\otimes D} \in \CF_D^{\otimes D *}$,~$\delta_\sigma (b^{\otimes D}) = \hat\varphi^{\otimes D} (\delta_\sigma^O (b^{\otimes D} \circ \varphi^{\otimes D}))$, therefore,
    \begin{align*}
      \delta_\sigma (b^{\otimes D}) \subs \delta_\sigma (a) &= \hat\varphi^{\otimes D} \big( \delta_\sigma^O (b^{\otimes D} \circ \varphi^{\otimes D}) \big) \subs \delta_\sigma(a) \\ 
        &= \hat\varphi \big( \delta_\sigma^O (b^{\otimes D} \circ \varphi^{\otimes D}) \subs^O \delta_\sigma(a) \big) \\
        &= \hat\varphi \big( \delta_\sigma^O ( (b^{\otimes D} \circ \varphi^{\otimes D}) \star^O a) \big) \\ 
        &= \hat\varphi \big( \delta_\sigma^O ( (b^{\otimes D} \star a) \circ \varphi ) \big) = \delta_\sigma (b^{\otimes D} \star a).
    \end{align*}
    This proves that~$\delta_\sigma$ is an algebra morphism.
\end{proof}

We now extend the insertion action~$\ins : \AT_D^{\oplus D} \otimes \AT_D \to \AT_D$ into the insertion product $\ins : \AT_D^{\oplus D} \otimes \AT_D^{\oplus D} \to \AT_D^{\oplus D}$ by
\[ \tau \iota_d \ins \gamma \iota_l = (\tau \ins_d \gamma) \iota_l. \]
We recall that~$\CF_D^{\otimes D} = S(\AT_D^{\oplus D})$ which allows us to use the Guin-Oudom process and define $\ins : \CF_D^{\otimes D} \otimes \CF_D^{\otimes D} \to \CF_D^{\otimes D}$ and the substitution product~$\subs : \CF_D^{\otimes D} \otimes \CF_D^{\otimes D} \to \CF_D^{\otimes D}$. The algebra~$(\CF_D^{\otimes D}, \subs)$ is endowed with the deshuffle coproduct~$\Delta$ of~$S(\AT_D^{\oplus D})$.
We present a generalization of the result obtained in~\cite{Calaque11tih} for classical forests.

\begin{theorem}
\label{thm:clumped_CEM_bialgebra}
  The algebra~$B_\subs := (\CF_D^{\otimes D}, \subs)$ endowed with the deshuffle coproduct~$\Delta$ is a bialgebra dual to the CEM bialgebra~$B_{CEM} := (\CF_D^{\otimes D}, \cdot, \Delta_{CEM})$ with respect to the~$\langle \blank, \blank \rangle_\sigma$ inner product. Moreover,~$H_{CEM} := B_{CEM} / \langle (1 - \forest{b_d_90}) \iota_d \, : d \in D \rangle$ is a Hopf algebra dual to the Grossman-Larson Hopf algebra~$H_\subs$, which is an appropriate sub-bialgebra of~$B_\subs$ and a universal enveloping algebra of~$(\AT_D^{\oplus D}, \ins)$.
\end{theorem}
\begin{proof}
  Let us start by proving that~$B_\subs$ is a bialgebra. The unit is given by~$u := \otimes_{d \in D} \exp^\cdot_C(\forest{b_d_90} \iota_d)$ and counit is given by~$\epsilon := \one^* \in \CF_D^{\otimes D *}$. It is straightforward to check that the counit is compatible with the product and unit. We use the property~$\Delta(\exp^\cdot_C (\tau)) = \exp_C^\cdot(\tau) \otimes \exp_C^\cdot(\tau)$ to check that the unit is compatible with the coproduct, that is,~$\Delta(u) = u \otimes u$. We use
    \[ \pi \subs (\blank \cdot \blank) = \sum_{(\pi)} (\pi_{(1)} \subs \blank) \cdot (\pi_{(2)} \subs \blank) \]
    and the associativity of~$\subs$ to obtain the identity
    \[ \sum_{(\pi),(\mu)} (\pi_{(1)} \subs \mu_{(1)} \subs \blank) (\pi_{(2)} \subs \mu_{(2)} \subs \blank) = \sum_{(\pi \subs \mu)} ((\pi \subs \mu)_{(1)} \subs \blank) ((\pi \subs \mu)_{(2)} \subs \blank). \]
    Therefore, the compatibility of the product and coproduct is necessary. This proves that~$B_\subs$ is a bialgebra. Lemma~\ref{lemma:concat_deshuffle_duality} and Proposition~\ref{prop:subs_cemdual} show that~$B_{CEM}$ is the dual bialgebra of~$B_\subs$.

    Similarly to~\cite{Calaque11tih}, we obtain a Hopf algebra once we take the quotient of~$B_{CEM}$ by the ideal~$\langle (\one - \forest{b_d_90}) \iota_d \, : d \in D \rangle$. The Hopf algebra~$H_{CEM}$ thus obtained is dual to the Hopf algebra~$H_\subs$ which is isomorphic to the Grossman-Larson Hopf algebra obtained using Guin-Oudom process applied to the insertion pre-Lie product~$\ins$. The elements of~$H_\subs$ have the form
    \[ \otimes_{d \in D} \exp^\cdot_C(\forest{b_d_90} \iota_d) \cdot \pi, \quad \text{for } \pi \in \CF_D^{\otimes D}. \]
    Following Guin-Oudom, the Grossman-Larson algebra is a universal enveloping algebra of the respective pre-Lie algebra.
\end{proof}

\begin{theorem}[Substitution law]\label{thm:substitution_law}
    Let~$a \in \CF_D^*$ and let~$\{ b_{0,d} \in \AT_D^* \; : \; d \in D\}$ be a set of infinitesimal characters. Let us consider a map~$\varphi : D \to \AT_D$ with~$\varphi(d) = \delta_\sigma (b_{0,d})$, let~$F_D (d) = f_d$ for~$d \in D$, and let~$S_\varphi = F_D \circ A_\varphi \circ \delta_\sigma$, then,
    \[ S_\varphi (a) = S(b_c^{\otimes D} \star a), \]
    where~$b_c^{\otimes D} = \otimes_{d \in D} b_{c,d} \iota_d$ with~$b_{c,d}$ being a character of~$\CF_D$ that extends~$b_{0,d}$.
\end{theorem}
\begin{proof}
  We use Lemma~\ref{lemma:Aphi_exp} and Proposition~\ref{prop:exp_odot} to write~$A_\varphi$ as
    \[ A_\varphi (\pi) = \delta_\sigma (b_c^{\otimes D}) \subs \pi, \quad \text{where } b_c^{\otimes D} = \otimes_{d \in D} b_{c,d} \iota_d, \]
  with~$b_{c,d}$ being a character of~$\CF_D$ that extends~$b_{0,d}$.
  We use the fact that~$S_\varphi = F_D \circ A_\varphi \circ \delta_\sigma$ and Proposition~\ref{prop:subs_cemdual} to obtain
  \[ S_\varphi (a) = F_D \big( \delta_\sigma (b_c^{\otimes D}) \subs \delta_\sigma (a) \big) = F_D \big( \delta_\sigma (b_c^{\otimes D} \star a) \big) = S(b_c^{\otimes D} \star a). \]
  This finishes the proof.
\end{proof}

Analogously to~\cite{Lundervold13bea}, due to the relation between the homomorphism $A_\varphi : \CF_D \to \CF_D$ of D-algebra and the corresponding map~$b_c^{\otimes D} \star : \CF_D^* \to \CF_D^*$, we obtain Corollary~\ref{corr:cointeraction}.
\begin{corollary}[Cointeraction]\label{corr:cointeraction}
    There is the following cointeraction between the substitution law and the composition law:
    \[ a_c^{\otimes D} \star (b \ast  c) = (a_c^{\otimes D} \star b) \ast (a_c^{\otimes D} \star c), \]
    for~$a_c^{\otimes D} = \otimes_{d \in D} a_{c,d}\iota_d~$ where~$a_{c,d}$ are characters of~$\CF_D$ and~$b, c \in \CF_D^*$.
\end{corollary}
\begin{proof}
Let~$\varphi(d) = \delta_\sigma (a_{0,d})$, then we have
    \begin{align*}
      \delta_\sigma (a_c^{\otimes D} \star (b \ast c)) 
      &= A_\varphi ( \delta_\sigma (b \ast  c) ) \\
      &= A_\varphi (\delta_\sigma (b)) \gl A_\varphi (\delta_\sigma (c)) \\
      &= \delta_\sigma ((a_c^{\otimes D} \star b) \ast  (a_c^{\otimes D} \star c)).
    \end{align*}
    Since~$\delta_\sigma$ is an isomorphism, the statement is proved.
\end{proof}

We extend the definition of~$\ins$ to decorated aromatic forests as~$\ins : \CF_D^{\otimes D} \otimes \AF_D \to \AF_D$ and note that~$\Phi$ is a~$\CF_D^{\otimes D}$-module morphism, that is,
\[ \Phi(\pi \ins \mu) = \pi \ins \Phi(\mu). \]
We extend the substitution product~$\subs : \CF_D^{\otimes D} \otimes \AF_D \to \AF_D$ and the coproduct $\Delta_{CEM} : \AF_D \to \CF_D^{\otimes D} \otimes \AF_D$ and use Proposition~\ref{prop:subs_cemdual} to see that~$\Phi^*$ is a~$\CF_D^{\otimes D}$-comodule morphism, that is,
\[ \Delta_{CEM} \circ \Phi^* = (\id \otimes \Phi^*) \circ \Delta_{CEM}. \]
Using the freeness of the tracial commutative D-algebra~$(\AF_D, \graft)$, we note that for every map~$\varphi : D \to \AT_D$, there exists a corresponding~$A_\varphi : \AF_D \to \AF_D$ with the following property
\[ \Phi \circ A_\varphi = A_\varphi \circ \Phi, \]
where we use the same notation for the~$A_\varphi$ over decorated clumped and aromatic forests.
We describe the substitution law for S\nobreakdash-series over decorated aromatic forests by using the fact that any functional over decorated aromatic forests~$a \in \AF_D^*$ can be written as an image of the map~$\Phi^*$, that is, there exists a functional~$a_C \in \CF_D^*$ such that~$a = a_C \circ \Phi^*$.

\begin{theorem}
    Using the notation from Theorem~\ref{thm:substitution_law} and given~$a \in \AF_D$, we have
    \[ S_\varphi (a) = S(b_c^{\otimes D} \star a), \]
    with convolution product~$\star$ with respect to~$\Delta_{CEM} : \AF_D \to \CF_D^{\otimes D} \otimes \AF_D$.
\end{theorem}
\begin{proof}
  We note that given a functional~$a \in \AF_D^*$, we can define a functional~$a_C \in \CF_D^*$ such that~$a = a_C \circ \Phi^*$, moreover,~$S(a) = S(a_C)$. The statement follows from the following identities,
    \[ S_\varphi (a) = S_\varphi (a_C) = S(b_c^{\otimes D} \star a_C) = S((b_c^{\otimes D} \star a_C) \circ \Phi^*) = S(b_c^{\otimes D} \star (a_C \circ \Phi^*)) = S(b_c^{\otimes D} \star a), \]
     with~$b_c^{\otimes D} := \otimes_{d \in D} b_{c,d} \iota_d$.
\end{proof}

Corollary~\ref{corr:cointeraction} is extended to functionals over decorated aromatic forests in a similar fashion.

\subsection{Algebraic structures of exotic aromatic forests}
\label{sec:exotic_algebraic_structure}

We extend the results of Section~\ref{sec:algebraic_structure} to the space~$\EAF$ of exotic aromatic forests defined in Definition~\ref{def:EAF}. To do so, we build the commutative tracial D-algebra~$(\EAF, \graft)$ of exotic aromatic forests using a commutative tracial D-algebra~$(\AF_{\bullet\N}, \graft)$ of decorated aromatic forests.

We define the D-algebra~$(\AF_{\bullet\N}, \graft)$ to be a vector space spanned by aromatic forests with black vertices and vertices decorated by natural numbers~$\N$. The decoration is denoted by $\alpha : V(\pi) \to \{\bullet\} \cup \N$ and we assume that every number decorates an even number of vertices, that is,~$\alpha^{-1}(n)$ is even for~$n \in \N$, and a vertex decorated by a number cannot have an incoming edge. We note that our assumptions on the structure of~$\pi \in \AF_{\bullet\N}$ are compatible with the commutative tracial D-algebra structure. For example,
\[ \forest{b[1,1,2,b[1]],b[2,1]} \cdot \forest{b[1,1,2,2]} = \forest{b[1,1,2,b[1]],b[2,1],b[1,1,2,2]}, \quad \forest{b[1,1]} \graft \forest{b[b[1],1]} = \forest{b[b[1,1],b[1],1]} + \forest{b[b[b[1,1],1],1]}. \]
We can obtain the D-algebra~$(\AF_{\bullet\N}, \graft)$ by choosing an appropriate sub-D-algebra and taking the quotient over an appropriate ideal of the free D-algebra. We define the commutative tracial D-algebra of exotic forests~$(\EAF, \graft)$ as a sub-D-algebra of~$(\overline\AF_{\bullet\N}, \graft)$ spanned by the elements
\begin{equation}
\label{eq:EAF_def}
  (\pi, \alpha_e) := \sum_{\alpha \in P(\alpha_e)} (\pi, \alpha) \  \in \  \overline{\AF}_{\bullet\N},
\end{equation}
with~$|\alpha_e^{-1}(n)| \in \{0, 2\}$ for~$n \in \mathbb{N}$.~$P(\alpha_e)$ is the set of decorations~$\alpha$ with~$\alpha^{-1}(\bullet) = \alpha_e^{-1}(\bullet)$ and 
\[ \alpha(v_1) = \alpha(v_2) \quad \text{if } \quad \alpha_e(v_1) = \alpha_e(v_2), \quad \text{for } v_1, v_2 \in V(\pi). \]
We note that the exotic aromatic forests constructed in this way agree with Definition~\ref{def:EAF}.
The D-algebra~$(\EAF, \graft)$ of exotic aromatic forests is graded by the number of roots. 
The algebraic structures related to D-algebras are studied in Section~\ref{sec:D-algebra}. 

An exotic aromatic forest is connected if it cannot be written as a concatenation of non-trivial exotic aromatic forests. This notion of connectedness coincides with the one found in~\cite{Laurent23tue}. We note that due to the pairings of the number vertices (that are also called \textit{lianas}), exotic aromatic forests can contain connected components which have more than one root, which is a major difference with the standard Butcher trees and forests. For example, the following exotic aromatic forests are connected:
\[ \forest{1,1}, \quad \forest{b[b,1],b[b,b[1]]}, \quad \forest{b[1],b[1,2],b[2,3],3}. \]
We consider the coalgebra~$(\EAF, \Delta_\EA)$ with~$\EA$-linear deshuffle coproduct~$\Delta_{\EA}$ whose $\EA$-module of primitive elements~$\Prim(\EAF, \Delta_\EA)$ is spanned by the connected exotic aromatic forests, for example,
\[ \Delta_{\EA} (\forest{(b[1,1]),b[2],b[2,3],3}) = \forest{(b[1,1]),b[2],b[2,3],3} \otimes \mathbf{1} + \mathbf{1} \otimes \forest{(b[1,1]),b[2],b[2,3],3}. \]

We note that the space of exotic aromatic forests~$\EAF$ can be defined as the symmetric algebra~$S_\EA (\PEAF)$ over the ring of exotic aromas. Analogously to Section~\ref{sec:decorated_clumped_forests}, we define two possible extensions of the concept of clumped forests to the exotic context,
\[ \CEF := S(\PEAF), \quad \CEFone := S(\EAT).\]
We note that both symmetric algebras are over the base field~$\R$, meaning that the exotic aromas are attached to the rooted components. We recall that~$\EAT$ is the space of exotic aromatic forests with one root and~$\CEFone \subset \CEF$.
Some elements of~$\CEF$ are
\[ \forest{(b),b[1],1} \cdot \forest{b[b,b[2]],b[2]} \cdot \forest{b[3],b[3,4],4}, \quad \forest{b[1]=b[1],b[b,2],2}. \]
We note that~$\forest{(b),b[1,1]} \cdot \forest{(b[2]),b[2]} \neq \forest{b[1,1]} \cdot \forest{(b),(b[2]),b[2]}$ in~$\CEF$.

Let~$\Phi : \CEF \to \EAF$ denote the map that forgets the "clumping", that is, it satisfies $\ker{\Phi} = \langle \omega \tau \cdot \gamma \cdot \pi - \tau \cdot \omega \gamma \cdot \pi \; : \; \omega \in \EA, \tau, \gamma \in \EAT, \pi \in \CEF \rangle$.
The results of Section~\ref{section:Decorated aromatic forests} yield the following structure on exotic aromatic and exotic clumped forests.
\begin{theorem}
  The space of exotic aromatic forests forms a Grossman-Larson Hopf algebroid
  \[ (\EAF, \mathbf{1}, \gl, \epsilon_{\EA}, \Delta_{\EA}, S_\gl),\]
  and the space of clumped exotic forests forms a Grossman-Larson Hopf algebra
  \[ (\CEF, \mathbf{1}, \gl, \epsilon, \Delta, S_\gl^C).\]
  Moreover,~$\Phi : \CEF \to \EAF$ is a surjective algebra morphism.
\end{theorem}
\begin{proof}
  We note that, according to the definition (\ref{eq:EAF_def}),~$\EAF$ is obtained by taking a sub-D-algebra of~$\overline{\AF}_{\bullet\N}$, therefore, we can use the analysis from Section~\ref{sec:D-algebra} to build a~$\AA/_\R$-bialgebra $(\EAF, \mathbf{1}, \gl, \epsilon_{\EA}, \Delta_{\EA})$. We obtain a bialgebra~$(\CEF, \mathbf{1}, \gl, \epsilon, \Delta)$ by noting that~$\CEF$ can be obtained by taking a sub-D-algebra of D-algebra~$\overline{\CF}_{\bullet\N}$ which is a special case of~$\CF_D$~(Section~\ref{sec:decorated_clumped_forests}).

  Following~\cite{LodayRonco}, the primitive elements~$\PEAF$ are endowed with a pre-Lie product:
\[ \pi_1 \tildegraft \pi_2 := \pi_1 \gl \pi_2 - \pi_1 \cdot \pi_2, \quad \text{for } \pi_1, \pi_2 \in \PEAF. \]
  We use the pre-Lie product~$\tildegraft$ to define the antipode~$S_\gl$ for~$\EAF$ by replacing all instances of~$\graft$ in Proposition~\ref{prop:antipode} by the product~$\tildegraft$, all instances of~$\tau \in \TT_D$ by~$\tau \in \Prim(\EF, \Delta)$, and~$\omega \in \AA_D$ by~$\omega \in \EA$ where~$\EF$ denotes the space of exotic forests, that is, exotic aromatic forests without aromas. The antipode~$S_\gl^C$ for~$\CEF$ is obtained in a similar way, but the trees~$\tau \in \TT_D$ are replaced by~$\tau \in \PEAF$ and the identities~$(i)$ are ignored. This proves that we have a Hopf algebroid and Hopf algebra structures over~$\EAF$ and~$\CEF$, respectively. 
  We note that~$\Phi(\one) = \one$ and~$\Phi(\pi \gl \mu) = \Phi(\pi) \gl \Phi(\mu)$, so~$\Phi$ is a surjective algebra morphism.
\end{proof}

\begin{theorem}
\label{thm:CEF_Hopf}
  The space of clumped exotic forests~$\CEF$ forms a Hopf algebra
  \[(\CEF, \one, \cdot, \one^*, \Delta_{CEM}, S),\]
  where~$\Delta_{CEM}$ is the coproduct extended from~$\Prim(\EAF)$ to~$\CEF$ by respecting the concatenation product.
  Moreover,~$\Phi^* : \EAF \to \CEF$ is a~$\CEFone$-comodule morphism where~$\Phi^*$ is the adjoint of~$\Phi$, that is,~$\Phi \circ \delta_\sigma = \delta_\sigma \circ \Phi^*$.
\end{theorem}
\begin{proof}
  The space~$\CEF$ can be defined analogously to~$\EAF$ using (\ref{eq:EAF_def}) such that~$\CEF \subset \overline\CF_D$ with~$D = \{ \bullet \} \cup \N$. We recall that, following Theorem~\ref{thm:clumped_CEM_bialgebra},~$B_{CEM} := (\CF_D^{\otimes D}, \cdot, \Delta_{CEM})$ is a bialgebra which becomes a Hopf algebra~$H_{CEM} := (\CF_D, \cdot, \Delta_{CEM})$ once we take the quotient of~$B_{CEM}$ by the ideal~$\langle (\one - \forest{b})\iota_\bullet \rangle + \JJ$ defined as
  \[ \JJ := \langle (\one - \forest{i_k}) \iota_{\forest{i_k}}, \pi \iota_{\forest{i_k}} \; : \; \pi \notin \{ \one, \forest{i_k} \}, k \in \N \rangle. \]
  The obtained coproduct of~$H_{CEM}$ can now be described as
  \[ \Delta_{CEM} (\pi) = \sum_{p \subset \pi} p \otimes \pi /_p, \]
  where the sum is over all clumped subforests~$p \in \CF_D$ that cover all black vertices and~$\pi /_p$ is the clumped forest obtained by contracting the aromatic trees of~$p$ into black vertices. If the forest~$\pi \in \CF_D$ doesn't have valid subforests~$p \in \CF_D$, then~$\Delta_{CEM} (\pi) = \one \otimes \pi$.
  We note that the construction of~$\CEF$ with the CEM Hopf algebra structure as a quotient and subspace of~$\overline\CF_D$ must be dual to the one seen in (\ref{eq:EAF_def}). We obtain the Hopf algebra
  \[(\CEF, \one, \cdot, \one^*, \Delta_{CEM}, S),\]
  with~$\Delta_{CEM} : \CEF \to \CEF \otimes \CEF$. We recall the analysis of Section~\ref{sec:substitution_law} and note that since~$\exp^\odot_C (b_0)$ is a character of~$\CEFone$, we consider~$\Delta_{CEM} : \CEF \to \CEFone \otimes \CEF$. We refer to the discussion from Section~\ref{sec:substitution_law} for showing that~$\Phi^*$ is a~$\CEFone$-comodule morphism.
\end{proof}

We can simplify the computation of the substitution law for exotic aromatic S\nobreakdash-series using the map~$\Phi$ and Theorem~\ref{thm:CEF_Hopf}.
Given a functional~$a \in \EAF^*$, we use the discussion from Section~\ref{sec:decorated_clumped_forests} to define a functional~$a_C \in \CEF$ over clumped exotic forests as
\[ a_C (\pi) := \frac{1}{n^m} a(\Phi(\pi)), \]
where~$n$ is the number of rooted components and~$m$ is the number of aromas. We have the property~$a = a_C \circ \Phi^*$ which we use to compute the substitution law as follows
\[ b_c \star a = b_c \star (a_C \circ \Phi^*) = (b_c \star a_C) \circ \Phi^*. \]
The substitution law~$b_c \star a_C$ over~$\CEF$ is easier to compute since~$\Delta_{CEM}$ over~$\CEF$ respects concatenation. Therefore, computing the values of~$\Delta_{CEM}$ on~$\PEAF$ is enough to obtain its values over all clumped exotic forests.

\begin{ex}
  Let us compute the CEM coproduct over~$\EAF$ using the comodule morphism~$\Phi^*$.
\begin{align*}
    (\Delta_{CEM} \circ \Phi^*) (\forest{(b),b[1],1,b}) &= \Delta_{CEM} (\forest{(b),b[1],1} \cdot \forest{b}) + \Delta_{CEM} (\forest{b[1],1} \cdot \forest{(b),b}) \\
    &= \Delta_{CEM} (\forest{(b),b[1],1}) \hat\cdot \Delta_{CEM} (\forest{b}) + \Delta_{CEM} (\forest{b[1],1}) \hat\cdot \Delta_{CEM} (\forest{(b),b}),
\end{align*}
where~$(\pi_{(1)} \otimes \pi_{(2)}) \hat\cdot (\mu_{(1)} \otimes \mu_{(2)}) = (\pi_{(1)} \cdot \mu_{(1)}) \otimes (\pi_{(2)} \cdot \mu_{(2)})$ and
\vspace{-0.5cm}
\begin{center}
\begin{minipage}{0.4\linewidth}
    \begin{align*}
        \Delta_{CEM} (\forest{(b),b[1],1}) &= \forest{b,b} \otimes \forest{(b),b[1],1} + \forest{(b),b} \otimes \forest{b[1],1}, \\
        \Delta_{CEM} (\forest{(b),b}) &= \forest{b,b} \otimes \forest{(b),b} + \forest{(b),b} \otimes \forest{b},
    \end{align*}
\end{minipage}
\begin{minipage}{0.4\linewidth}
    \begin{align*}
        \Delta_{CEM} (\forest{b}) &= \forest{b} \otimes \forest{b}, \\
        \Delta_{CEM} (\forest{b[1],1}) &= \forest{b} \otimes \forest{b[1],1}.
    \end{align*}
\end{minipage}
\end{center}
Therefore, we have
\begin{align*}
    (\Delta_{CEM} \circ \Phi^*) (\forest{(b),b[1],1,b}) 
    &= \forest{b,b,b} \otimes \forest{(b),b[1],1} \cdot \forest{b} + \forest{(b),b} \cdot \forest{b} \otimes \forest{b[1],1,b} \\
    &+ \forest{b,b,b} \otimes \forest{b[1],1} \cdot \forest{(b),b} + \forest{(b),b} \cdot \forest{b} \otimes \forest{b[1],1,b} \\
    &= \forest{b,b,b} \otimes \Phi^* (\forest{(b),b[1],1,b}) + 2 \forest{(b),b} \cdot \forest{b} \otimes \Phi^*(\forest{b[1],1,b}),
\end{align*}
which agrees with a direct computation that gives
\[ \Delta_{CEM} (\forest{(b),b[1],1,b}) = \forest{b,b,b} \otimes \forest{(b),b[1],1,b} + 2 \forest{(b),b} \cdot \forest{b} \otimes \forest{b[1],1,b}. \]
\end{ex}

\section{Conclusion}

In this paper, we uncover the fundamental algebraic structures that govern the approximations in the weak sense and for the invariant measure of stochastic dynamics with additive noise in both Euclidean and manifold settings.
We present in particular the novel concept of clumping for the Hopf algebra structures associated with substitution using aromas. Additionally, we introduce the free commutative tracial D-algebra of decorated aromatic forests, along with the related Grossman-Larson Hopf algebroid and pre-Hopf algebroid.
This algebraic study enables us to describe the substitution law for decorated aromatic forests by linking substitution to D-algebra homomorphisms.
Understanding the substitution law for exotic aromatic S\nobreakdash-series provides an elegant algebraic framework for the backward error analysis of stochastic dynamics at any order, with an explicit expression for the modified vector field.

There are a handful of challenging open questions that follow from the present work.
The deeper understanding of the integration by parts operation given in Proposition~\ref{proposition:prop_character_IBP} in the Euclidean setting is key in extending Theorems~\ref{theorem:BEA_Rd} and~\ref{theorem:modified_equation_Rd} to the manifold case for the creation of modified equations and the backward error analysis of projection methods. An explicit description of the combination of forests that vanish by integration by parts (see Remark~\ref{rk:kernel_IBP}) is necessary for the derivation of order conditions for the invariant measure at any order.
Following Remark~\ref{remark:EA_bicomplex}, an alternate approach could be to characterise the modified vector field of methods that preserve the invariant measure exactly in the spirit of~\cite{Laurent23tab,Laurent23tld} in the context of volume-preserving methods. The exotic aromatic formalism is the natural tool in this context as the operations~$\Div$ and~$\langle f,\blank \rangle$ of equation \eqref{equation:stochastic_VP} precisely generate the aromas.
This calls for future works in the spirit of~\cite{Laurent23tab,Laurent23tld}.

The Butcher forests and their extensions have been successfully applied to rough paths with the construction of branched rough paths~\cite{Gubinelli10ror} (see also~\cite{Hairer15gvn,Friz20aco}), planarly branched rough paths~\cite{MuntheKaas08oth,Curry20pbr,Ebrahimi23pei}, and aromatic rough paths~\cite{Lejay22cgr}. In this context, the substitution law corresponds to the translation of rough paths, and the notion of composition of exotic aromatic trees~\cite{MuntheKaas16abs,Laurent23tue} is closely linked to the Hopf algebra of multi-indices~\cite{Linares21tsg,Jacques23pla}.
The creation and study of rough paths structures arising from exotic aromatic or exotic clumped forests is natural and is matter for later work.
Moreover, we use the clumping idea to explain the algebraic behaviour brought by the divergence and the scalar product operators. One could consider general multilinear maps that transform trees into rootless graphs. The questions of the freeness, the general algebraic structure, and the potential applications are all interesting matters for follow ups.

Following the recent work~\cite{Bharath23sae}, the approach with projection methods for the sampling of ergodic dynamics on manifolds could be replaced with an intrinsic approach with Lie-group methods. This new approach could yield simple and efficient intrinsic discretisations of high-order for sampling the invariant measure, which could be combined with other popular techniques such as postprocessing, Metropolisation,\dots
It could also greatly simplify the tedious algebraic structure associated to projection methods and yield a structure similar to the ones appearing in the study of Lie group methods~\cite{Iserles00lgm}. This will be studied in upcoming works.

\bigskip

\noindent \textbf{Acknowledgements.}\
The authors would like to thank Dominique Manchon, Hans Munthe-Kaas, and Gilles Vilmart for helpful discussions.
The authors acknowledge the support of the Swiss National Science Foundation, projects No 200020\_214819, and No. 200020\_192129, of the Research Council of Norway through project 302831 ``Computational Dynamics and Stochastics on Manifolds'' (CODYSMA), and of the French program ANR-11-LABX-0020-0 (Labex Lebesgue).

\small
\bibliography{Ma_Bibliographie,diss}

\vskip-1ex
\normalsize
\begin{appendices}

\section{Examples for forests of small order}
\label{section:tables_examples}

We present in Tables~\ref{appendix:BCK_coproduct} and~\ref{appendix:CEM_coaction} the Butcher-Connes-Kreimer coproduct and the substitution law of all primitive exotic aromatic forests~$\Prim(\EAF, \Delta_{\EA})$ of order up to three. 
    To compute the Butcher-Connes-Kreimer coproduct efficiently, we use Remark~\ref{remark:GL_Hopf_algebra} which implies the property
    \[ \Delta_{BCK} (\pi_1 \cdot \pi_2) = \Delta_{BCK} (\pi_1) \hat\cdot \Delta_{BCK} (\pi_2), \]
    where~$(x \otimes y) \hat\cdot (u \cdot v) = (x \cdot u) \otimes (y \cdot v)$ and~$\pi_1, \pi_2 \in \AF_D$. Therefore, the values of the Butcher-Connes-Kreimer coproduct on primitive exotic aromatic forests are enough to compute the values for general exotic aromatic forests in a straightforward way.

    The forests in the table are ordered according to their order as defined in Definition~\ref{def:EAF} and by the number of roots.

\renewcommand*{\arraystretch}{1.4}
\begin{longtable}{c|c}
    \caption{Butcher-Connes-Kreimer coproduct for all primitive exotic aromatic forests up to order three}
    \label{appendix:BCK_coproduct} \\
   ~$\pi$ &~$\Delta_{BCK}(\pi)$ \\
    \hhline{=|=}
    \forest{(b)} &~$\one \otimes \forest{(b)} + \forest{(b)} \otimes \one$ \\
    \forest{b=b} &~$\one \otimes \forest{b=b} + \forest{b=b} \otimes \one$ \\
    \forest{b} &~$\one \otimes \forest{b} + \forest{b} \otimes \one$ \\
    \forest{1,1} &~$\one \otimes \forest{1,1} + \forest{1,1} \otimes \one$ \\
    \hline
    \forest{(b[b])} &~$\one \otimes \forest{(b[b])} + \forest{b} \otimes \forest{(b)} + \forest{(b[b])} \otimes \one$ \\
    \forest{(b,b)} &~$\one \otimes \forest{(b,b)} + \forest{(b,b)} \otimes \one$ \\
    \forest{(b[1,1])} &~$\one \otimes \forest{(b[1,1])} + \forest{1,1} \otimes \forest{(b)} + \forest{(b[1,1])} \otimes \one$ \\
    \forest{b=b[b]} &~$\one \otimes \forest{b=b[b]} + \forest{b} \otimes \forest{b=b} + \forest{b=b[b]} \otimes \one$ \\
    \forest{b=b[1,1]} &~$\one \otimes \forest{b=b[1,1]} + \forest{1,1} \otimes \forest{b=b} + \forest{b=b[1,1]} \otimes \one$ \\
    \forest{b[1]=b[1]} &~$\one \otimes \forest{b[1]=b[1]} + \forest{1,1} \otimes \forest{b=b} + \forest{b[1]=b[1]} \otimes \one$ \\
    \forest{b[b]} &~$\one \otimes \forest{b[b]} + \forest{b} \otimes \forest{b} + \forest{b[b]} \otimes \one$ \\
    \forest{b[1,1]} &~$\one \otimes \forest{b[1,1]} + \forest{1,1} \otimes \forest{b} + \forest{b[1,1]} \otimes \one$ \\
    \forest{(b[1]),1} &~$\one \otimes \forest{(b[1]),1} + \forest{1,1} \otimes \forest{(b)} + \forest{(b[1]),1} \otimes \one$ \\
    \forest{b=b[1],1} &~$\one \otimes \forest{b=b[1],1} + \forest{1,1} \otimes \forest{b=b} + \forest{b=b[1],1} \otimes \one$ \\
    \forest{b[1],1} &~$\one \otimes \forest{b[1],1} + \forest{1,1} \otimes \forest{b} + \forest{b[1],1} \otimes \one$ \\
    \hline
    \forest{(b[b[b]])} &~$\one \otimes \forest{(b[b[b]])} + \forest{b} \otimes \forest{(b[b])} + \forest{b[b]} \otimes \forest{(b)} + \forest{(b[b[b]])} \otimes \one$ \\
    \forest{(b[b,b])} &~$\one \otimes \forest{(b[b,b])} + 2 \forest{b} \otimes \forest{(b[b])} + \forest{b,b} \otimes \forest{(b)} + \forest{(b[b,b])} \otimes \one$ \\
    \forest{(b[b],b)} &~$\one \otimes \forest{(b[b],b)} + \forest{b} \otimes \forest{(b,b)} + \forest{(b[b],b)} \otimes \one$ \\
    \forest{(b,b,b)} &~$\one \otimes \forest{(b,b,b)} + \forest{(b,b,b)} \otimes \one$ \\
    \forest{(b[b[1,1]])} &~$\one \otimes \forest{(b[b[1,1]])} + \forest{1,1} \otimes \forest{(b[b])} + \forest{b[1,1]} \otimes \forest{(b)} + \forest{(b[b[1,1]])} \otimes \one$ \\
    \forest{(b[b[1],1])} &~$\one \otimes \forest{(b[b[1],1])} + \forest{1,1} \otimes \forest{(b[b])} + \forest{b[1],1} \otimes \forest{(b)} + \forest{(b[b[1],1])} \otimes \one$ \\
    \forest{(b[b,1,1])} &~$\one \otimes \forest{(b[b,1,1])} + \forest{1,1} \otimes \forest{(b[b])} + \forest{b} \otimes \forest{(b[1,1])} + \forest{b,1,1} \otimes \forest{(b)} + \forest{(b[b,1,1])} \otimes \one$ \\
    \forest{(b[1]),(b[1])} &~$\one \otimes \forest{(b[1]),(b[1])} + \forest{1,1} \otimes \forest{(b),(b)} + \forest{(b[1]),1} \otimes \forest{(b)} + \forest{(b[1]),1} \otimes \forest{(b)} + \forest{(b[1]),(b[1])} \otimes \one$ \\
    \forest{(b[1],b[1])} &~$\one \otimes \forest{(b[1],b[1])} + \forest{1,1} \otimes \forest{(b,b)} + \forest{(b[1],b[1])} \otimes \one$ \\
    \forest{(b[1,1],b)} &~$\one \otimes \forest{(b[1,1],b)} + \forest{1,1} \otimes \forest{(b,b)} + \forest{(b[1,1],b)} \otimes \one$ \\
    \forest{(b[1,1,2,2])} &~$\one \otimes \forest{(b[1,1,2,2])} + 2 \forest{1,1} \otimes \forest{(b[1,1])} + \forest{1,1,2,2} \otimes \forest{(b)} + \forest{(b[1,1,2,2])} \otimes \one$ \\
    \forest{b=b[b[b]]} &~$\one \otimes \forest{b=b[b[b]]} + \forest{b} \otimes \forest{b=b[b]} + \forest{b[b]} \otimes \forest{b=b} + \forest{b=b[b[b]]} \otimes \one$ \\
    \forest{b=b[b,b]} &~$\one \otimes \forest{b=b[b,b]} + 2 \forest{b} \otimes \forest{b=b[b]} + \forest{b,b} \otimes \forest{b=b} + \forest{b=b[b,b]} \otimes \one$ \\
    \forest{b[b]=b[b]} &~$\one \otimes \forest{b[b]=b[b]} + 2 \forest{b} \otimes \forest{b=b[b]} + \forest{b,b} \otimes \forest{b=b} + \forest{b[b]=b[b]} \otimes \one$ \\
    \forest{b=b[b[1,1]]} &~$\one \otimes \forest{b=b[b[1,1]]} + \forest{1,1} \otimes \forest{b=b[b]} + \forest{b[1,1]} \otimes \forest{b=b} + \forest{b=b[b[1,1]]} \otimes \one$ \\
    \forest{b=b[b[1],1]} &~$\one \otimes \forest{b=b[b[1],1]} + \forest{1,1} \otimes \forest{b=b[b]} + \forest{b[1],1} \otimes \forest{b=b} + \forest{b=b[b[1],1]} \otimes \one$ \\
    \forest{b[1]=b[b[1]]} &~$\one \otimes \forest{b[1]=b[b[1]]} + \forest{1,1} \otimes \forest{b=b[b]} + \forest{b[1],1} \otimes \forest{b=b} + \forest{b[1]=b[b[1]]} \otimes \one$ \\
    \forest{b=b[b,1,1]} &~$\one \otimes \forest{b=b[b,1,1]} + \forest{1,1} \otimes \forest{b=b[b]} + \forest{b} \otimes \forest{b=b[1,1]} + \forest{b,1,1} \otimes \forest{b=b} + \forest{b=b[b,1,1]} \otimes \one$ \\
    \forest{b[1]=b[b,1]} &~$\one \otimes \forest{b[1]=b[b,1]} + \forest{1,1} \otimes \forest{b=b[b]} + \forest{b} \otimes \forest{b[1]=b[1]} + \forest{b,1,1} \otimes \forest{b=b} + \forest{b[1]=b[b,1]} \otimes \one$ \\
    \forest{b[1,1]=b[b]} &~$\one \otimes \forest{b[1,1]=b[b]} + \forest{1,1} \otimes \forest{b=b[b]} + \forest{b} \otimes \forest{b[1,1]=b} + \forest{b,1,1} \otimes \forest{b=b} + \forest{b[1,1]=b[b]} \otimes \one$ \\
    \forest{b=b[1,1,2,2]} &~$\one \otimes \forest{b=b[1,1,2,2]} + 2 \forest{1,1} \otimes \forest{b=b[1,1]} + \forest{1,1,2,2} \otimes \forest{b=b} + \forest{b=b[1,1,2,2]} \otimes \one$ \\
    \forest{b[2]=b[1,1,2]} &~$\one \otimes \forest{b[2]=b[1,1,2]} + \forest{1,1} \otimes \forest{b=b[1,1]} + \forest{1,1} \otimes \forest{b[1]=b[1]} + \forest{1,1,2,2} \otimes \forest{b=b} + \forest{b[2]=b[1,1,2]} \otimes \one$ \\
    \forest{b[2,2]=b[1,1]} &~$\one \otimes \forest{b[2,2]=b[1,1]} + 2 \forest{1,1} \otimes \forest{b=b[1,1]} + \forest{1,1,2,2} \otimes \forest{b=b} + \forest{b[2,2]=b[1,1]} \otimes \one$ \\
    \forest{b[1,2]=b[1,2]} &~$\one \otimes \forest{b[1,2]=b[1,2]} + 2 \forest{1,1} \otimes \forest{b[1]=b[1]} + \forest{1,1,2,2} \otimes \forest{b=b} + \forest{b[1,2]=b[1,2]} \otimes \one$ \\
    \forest{b=b[1],(b[1])} &~$\one \otimes \forest{b=b[1],(b[1])} + \forest{1,1} \otimes \forest{b=b,(b)} + \forest{b=b[1],1} \otimes \forest{(b)} + \forest{(b[1]),1} \otimes \forest{b=b} + \forest{b=b[1],(b[1])} \otimes \one$ \\
    \forest{b=b[1],b[1]=b} &~$\one \otimes \forest{b=b[1],b[1]=b} + \forest{1,1} \otimes \forest{b=b,b=b} + 2 \forest{b=b[1],1} \otimes \forest{b=b} + \forest{b=b[1],b[1]=b} \otimes \one$ \\
    \forest{b[b[b]]} &~$\one \otimes \forest{b[b[b]]} + \forest{b} \otimes \forest{b[b]} + \forest{b[b]} \otimes \forest{b} + \forest{b[b[b]]} \otimes \one$ \\
    \forest{b[b,b]} &~$\one \otimes \forest{b[b,b]} + 2 \forest{b} \otimes \forest{b[b]} + \forest{b,b} \otimes \forest{b} + \forest{b[b,b]} \otimes \one$ \\
    \forest{b[b[1,1]]} &~$\one \otimes \forest{b[b[1,1]]} + \forest{1,1} \otimes \forest{b[b]} + \forest{b[1,1]} \otimes \forest{b} + \forest{b[b[1,1]]} \otimes \one$ \\
    \forest{b[b[1],1]} &~$\one \otimes \forest{b[b[1],1]} + \forest{1,1} \otimes \forest{b[b]} + \forest{b[1],1} \otimes \forest{b} + \forest{b[b[1],1]} \otimes \one$ \\
    \forest{b[b,1,1]} &~$\one \otimes \forest{b[b,1,1]} + \forest{1,1} \otimes \forest{b[b]} + \forest{b} \otimes \forest{b[1,1]} + \forest{b,1,1} \otimes \forest{b} + \forest{b[b,1,1]} \otimes \one$ \\
    \forest{b[1,1,2,2]} &~$\one \otimes \forest{b[1,1,2,2]} + 2 \forest{1,1} \otimes \forest{b[1,1]} + \forest{1,1,2,2} \otimes \forest{b} + \forest{b[1,1,2,2]} \otimes \one$ \\
    \forest{(b[b[1]]),1} &~$\one \otimes \forest{(b[b[1]]),1} + \forest{1,1} \otimes \forest{(b[b])} + \forest{b[1],1} \otimes \forest{(b)} + \forest{(b[b[1]]),1} \otimes \one$ \\
    \forest{(b[b,1]),1} &~$\one \otimes \forest{(b[b,1]),1} + \forest{1,1} \otimes \forest{(b[b])} + \forest{b} \otimes \forest{(b[1]),1} + \forest{b,1,1} \otimes \forest{(b)} + \forest{(b[b,1]),1} \otimes \one$ \\
    \forest{(b[1]),b[1]} &~$\one \otimes \forest{(b[1]),b[1]} + \forest{1,1} \otimes \forest{(b),b} + \forest{b[1],1} \otimes \forest{(b)} + \forest{(b[1]),1} \otimes \forest{b} + \forest{(b[1]),b[1]} \otimes \one$ \\
    \forest{(b,b[1]),1} &~$\one \otimes \forest{(b,b[1]),1} + \forest{1,1} \otimes \forest{(b,b)} + \forest{(b,b[1]),1} \otimes \one$ \\
    \forest{(b[1,2,2]),1} &~$\one \otimes \forest{(b[1,2,2]),1} + \forest{1,1} \otimes \forest{(b[1]),1} + \forest{1,1} \otimes \forest{(b[1,1])} + \forest{1,1,2,2} \otimes \forest{(b)} + \forest{(b[1,2,2]),1} \otimes \one$ \\
    \forest{b=b[b[1]],1} &~$\one \otimes \forest{b=b[b[1]],1} + \forest{1,1} \otimes \forest{b=b[b]} + \forest{b[1],1} \otimes \forest{b=b} + \forest{b=b[b[1]],1} \otimes \one$ \\
    \forest{b[b[1]],1} &~$\one \otimes \forest{b[b[1]],1} + \forest{1,1} \otimes \forest{b[b]} + \forest{b[1],1} \otimes \forest{b} + \forest{b[b[1]],1} \otimes \one$ \\
    \forest{b[b,1],1} &~$\one \otimes \forest{b[b,1],1} + \forest{1,1} \otimes \forest{b[b]} + \forest{b} \otimes \forest{b[1],1} + \forest{b,1,1} \otimes \forest{b} + \forest{b[b,1],1} \otimes \one$ \\
    \forest{b[1],b[1]} &~$\one \forest{b[1],b[1]} + \forest{1,1} \otimes \forest{b,b} + 2 \forest{b[1],1} \otimes \forest{b} + \forest{b[1],b[1]} \otimes \one$ \\
    \forest{b[1,2,2],1} &~$\one \otimes \forest{b[1,2,2],1} + \forest{1,1} \otimes \forest{b[1],1} + \forest{1,1} \otimes \forest{b[1,1]} + \forest{1,1,2,2} \otimes \forest{b} + \forest{b[1,2,2],1} \otimes \one$ \\
    \forest{b=b[b,1],1} &~$\one \otimes \forest{b=b[b,1],1} + \forest{1,1} \otimes \forest{b=b[b]} + \forest{b} \otimes \forest{b=b[1],1} + \forest{b,1,1} \otimes \forest{b=b} + \forest{b=b[b,1],1} \otimes \one$ \\
    \forest{b[1]=b[b],1} &~$\one \otimes \forest{b[1]=b[b],1} + \forest{1,1} \otimes \forest{b=b[b]} + \forest{b} \otimes \forest{b[1]=b,1} + \forest{b,1,1} \otimes \forest{b=b} + \forest{b[1]=b[b],1} \otimes \one$ \\
    \forest{b=b[1,1,2],2} &~$\one \otimes \forest{b=b[1,1,2],2} + \forest{1,1} \otimes \forest{b=b[1,1]} + \forest{1,1} \otimes \forest{b=b[1],1} + \forest{1,1,2,2} \otimes \forest{b=b} + \forest{b=b[1,1,2],2} \otimes \one$ \\
    \forest{b[2]=b[1,1],2} &~$\one \otimes \forest{b[2]=b[1,1],2} + \forest{1,1} \otimes \forest{b=b[1,1]} + \forest{1,1} \otimes \forest{b[1]=b,1} + \forest{1,1,2,2} \otimes \forest{b=b} + \forest{b[2]=b[1,1],2} \otimes \one$ \\
    \forest{b[1]=b[1,2],2} &~$\one \otimes \forest{b[1]=b[1,2],2} + \forest{1,1} \otimes \forest{b[1]=b[1]} + \forest{1,1} \otimes \forest{b=b[1],1} + \forest{1,1,2,2} \otimes \forest{b=b} + \forest{b[1]=b[1,2],2} \otimes \one$ \\
    \forest{(b[1,2]),1,2} &~$\one \otimes \forest{(b[1,2]),1,2} + 2 \forest{1,1} \otimes \forest{(b[1]),1} + \forest{1,1,2,2} \otimes \forest{(b)} + \forest{(b[1,2]),1,2} \otimes \one$ \\
    \forest{b[1,2],1,2} &~$\one \otimes \forest{b[1,2],1,2} + 2 \forest{1,1} \otimes \forest{b[1],1} + \forest{1,1,2,2} \otimes \forest{b} + \forest{b[1,2],1,2} \otimes \one$ \\
    \forest{b[1]=b[2],1,2} &~$\one \otimes \forest{b[1]=b[2],1,2} + 2 \forest{1,1} \otimes \forest{b[1]=b,1} + \forest{1,1,2,2} \otimes \forest{b=b} + \forest{b[1]=b[2],1,2} \otimes \one$ \\
    \forest{b=b[1,2],1,2} &~$\one \otimes \forest{b=b[1,2],1,2} + 2 \forest{1,1} \otimes \forest{b=b[1],1} + \forest{1,1,2,2} \otimes \forest{b=b} + \forest{b=b[1,2],1,2} \otimes \one$
\end{longtable}

We refer to the discussion from Section~\ref{section:structure_EAF} for an efficient way to compute the CEM coaction. The computations can be checked using Proposition~\ref{prop:subs_cemdual} with the property
\[ \langle \Delta_{CEM} (\pi), \pi_1 \otimes \pi_2 \rangle = \frac{\sigma(\pi) \langle \pi_1 \subs \pi_2, \pi \rangle}{\sigma(\pi_1) \sigma(\pi_2)}. \]

\renewcommand*{\arraystretch}{1.4}
\begin{longtable}{c|c}
    \caption{CEM coaction for all exotic aromas and primitive exotic aromatic forests up to order three}
    \label{appendix:CEM_coaction}\\
   ~$\pi$ &~$\Delta_{CEM}(\pi)$ \\
	\hhline{=|=}
    \forest{(b)} &~$\forest{b} \otimes \forest{(b)}$ \\
    \forest{b=b} &~$\forest{b,b} \otimes \forest{b=b}$ \\
    \forest{b} &~$\forest{b} \otimes \forest{b}$ \\
    \forest{1,1} &~$\one \otimes \forest{1,1}$ \\
    \hline
    \forest{(b[b])} &~$\forest{b,b} \otimes \forest{(b[b])} + \forest{(b),b} \otimes \forest{(b)} + \forest{b[b]} \otimes \forest{(b)}$ \\
    \forest{(b,b)} &~$\forest{b,b} \otimes \forest{(b,b)} + 2 \forest{b[b]} \otimes \forest{(b)}$ \\
    \forest{(b[1,1])} &~$\forest{b} \otimes \forest{(b[1,1])} + 2 \forest{(b[1]),1} \otimes \forest{(b)} + \forest{b[1,1]} \otimes \forest{(b)}$ \\
    \forest{b=b[b]} &~$\forest{b,b,b} \otimes \forest{b=b[b]} + \forest{b=b,b} \otimes \forest{(b)} + \forest{b,b[b]} \otimes \forest{b=b}$ \\
    \forest{b=b[1,1]} &~$\forest{b,b} \otimes \forest{b=b[1,1]} + 2 \forest{b=b[1],1} \otimes \forest{(b)} + \forest{b,b[1,1]} \otimes \forest{b=b}$ \\
    \forest{b[1]=b[1]} &~$\forest{b,b} \otimes \forest{b[1]=b[1]} + 2 \forest{b=b[1],1} \otimes \forest{(b)}$ \\
    \forest{b[b]} &~$\forest{b,b} \otimes \forest{b[b]} + \forest{b[b]} \otimes \forest{b}$ \\
    \forest{b[1,1]} &~$\forest{b} \otimes \forest{b[1,1]} + \forest{b[1,1]} \otimes \forest{b}$ \\
    \forest{(b),b} &~$\forest{b,b} \otimes \forest{(b),b} + \forest{(b),b} \otimes \forest{b}$ \\
    \forest{(b[1]),1} &~$\forest{b} \otimes \forest{(b[1]),1} + \forest{(b[1]),1} \otimes \forest{b}$ \\
    \forest{b=b,b} &~$\forest{b,b,b} \otimes \forest{b=b,b} + \forest{b=b,b} \otimes \forest{b}$ \\
    \forest{b=b[1],1} &~$\forest{b,b} \otimes \forest{b=b[1],1} + \forest{b=b[1],1} \otimes \forest{b}$ \\
    \forest{b[1],1} &~$\forest{b} \otimes \forest{b[1],1}$ \\
    \forest{(b),1,1} &~$\forest{b} \otimes \forest{(b),1,1}$ \\
    \forest{b=b,1,1} &~$\forest{b,b} \otimes \forest{b=b,1,1}$ \\
    \hline
    \forest{(b[b[b]])} &~$\forest{b,b,b} \otimes \forest{(b[b[b]])} + 2 \forest{b[b],b} \otimes \forest{(b[b])} + \forest{(b),b[b]} \otimes \forest{(b)} + \forest{(b[b]),b} \otimes \forest{(b)} + \forest{b[b[b]]} \otimes \forest{(b)}$ \\
    \forest{(b[b,b])} &~$\forest{b,b,b} \otimes \forest{(b[b,b])} + 2 \forest{b[b],b} \otimes \forest{(b[b])} + 2 \forest{(b[b]),b} \otimes \forest{(b)} + \forest{b[b,b]} \otimes \forest{(b)}$ \\
    \forest{(b[b],b)} & \makecell{$\forest{b,b,b} \otimes \forest{(b[b],b)} + \forest{b[b],b} \otimes \forest{(b,b)} + 2 \forest{b[b],b} \otimes \forest{(b[b])} + \forest{b[b,b]} \otimes \forest{(b)} + \forest{(b,b),b} \otimes \forest{(b)}$ \\~$ + \forest{b[b[b]]} \otimes \forest{(b)}$} \\
    \forest{(b,b,b)} &~$\forest{b,b,b} \otimes \forest{(b,b,b)} + 3 \forest{b[b],b} \otimes \forest{(b,b)} + 3 \forest{b[b[b]]} \otimes \forest{(b)}$ \\
    \forest{(b[b[1,1]])} & \makecell{$\forest{b,b} \otimes \forest{(b[b[1,1]])} + \forest{b[1,1],b} \otimes \forest{(b[b])} + \forest{b[b]} \otimes \forest{(b[1,1])} + \forest{(b),b[1,1]} \otimes \forest{(b)}$ \\~$ + 2 \forest{(b[b[1]]),1} \otimes \forest{(b)} + \forest{b[b[1,1]]} \otimes \forest{(b)}$} \\
    \forest{(b[b[1],1])} & \makecell{$\forest{b,b} \otimes \forest{(b[b[1],1])} + \forest{b[b]} \otimes \forest{(b[1,1])} + \forest{(b[b[1]]),1} \otimes \forest{(b)} + \forest{(b[b,1]),1} \otimes \forest{(b)}$ \\~$ + \forest{b[b[1],1]} \otimes \forest{(b)}$} \\
    \forest{(b[b,1,1])} & \makecell{$\forest{b,b} \otimes \forest{(b[b,1,1])} + \forest{b[b]} \otimes \forest{(b[1,1])} + \forest{b[1,1],b} \otimes \forest{(b[b])} + \forest{(b[1,1]),b} \otimes \forest{(b)}$ \\~$ + 2 \forest{(b[b,1]),1} \otimes \forest{(b)} + \forest{b[b,1,1]} \otimes \forest{(b)}$} \\
    \forest{(b[1]),(b[1])} & \makecell{$\forest{b,b} \otimes \forest{(b[1]),(b[1])} + 2 \forest{(b),b} \otimes \forest{(b[1,1])} + 2 \forest{(b[1]),1,b} \otimes \forest{(b[b])} + 2 \forest{(b[1]),(b),1} \otimes \forest{(b)}$ \\~$ + 2 \forest{(b[1]),b[1]} \otimes \forest{(b)}$} \\
    \forest{(b[1],b[1])} &~$\forest{b,b} \otimes \forest{(b[1],b[1])} + 2 \forest{b[b[1],1]} \otimes \forest{(b)} + 2 \forest{(b[1],b),1} \otimes \forest{(b)} + 2 \forest{b[b]} \otimes \forest{(b[1,1])}$ \\
    \forest{(b[1,1],b)} & \makecell{$\forest{b,b} \otimes \forest{(b[1,1],b)} + \forest{b[1,1],b} \otimes \forest{(b,b)} + \forest{b[b]} \otimes \forest{(b[1,1])} + \forest{b[b[1,1]]} \otimes \forest{(b)}$ \\~$ + 2 \forest{(b[1],b),1} \otimes \forest{(b)} + \forest{b[1,1,b]} \otimes \forest{(b)}$} \\
    \forest{(b[1,1,2,2])} &~$\forest{b} \otimes \forest{(b[1,1,2,2])} + 2 \forest{b[1,1]} \otimes \forest{(b[1,1])} + 4 \forest{(b[1,1,2]),2} \otimes \forest{(b)} + \forest{b[1,1,2,2]} \otimes \forest{(b)}$ \\
    \forest{b=b[b[b]]} & \makecell{$\forest{b,b,b,b} \otimes \forest{b=b[b[b]]} + 2 \forest{b[b],b,b} \otimes \forest{b=b[b]} + \forest{b=b,b[b]} \otimes \forest{(b)} + \forest{b=b[b],b} \otimes \forest{(b)}$ \\~$ + \forest{b,b[b[b]]} \otimes \forest{b=b}$} \\
    \forest{b=b[b,b]} &~$\forest{b,b,b,b} \otimes \forest{b=b[b,b]} + 2 \forest{b[b],b,b} \otimes \forest{b=b[b]} + 2 \forest{b=b[b],b} \otimes \forest{(b)} + \forest{b,b[b,b]} \otimes \forest{b=b}$ \\
    \forest{b[b]=b[b]} &~$\forest{b,b,b,b} \otimes \forest{b[b]=b[b]} + \forest{b[b],b,b} \otimes \forest{b=b[b]} + 2 \forest{b=b[b],b} \otimes \forest{(b)} + \forest{b[b],b[b]} \otimes \forest{b=b}$ \\
    \forest{b=b[b[1,1]]} & \makecell{$\forest{b,b,b} \otimes \forest{b=b[b[1,1]]} + \forest{b[1,1],b,b} \otimes \forest{b=b[b]} + \forest{b[b],b} \otimes \forest{b=b[1,1]} + \forest{b=b,b[1,1]} \otimes \forest{(b)}$ \\~$ + \forest{b=b[b[1]],1} \otimes \forest{(b)} + \forest{b,b[b[1,1]]} \otimes \forest{b=b}$} \\
    \forest{b=b[b[1],1]} & \makecell{$\forest{b,b,b} \otimes \forest{b=b[b[1],1]} + \forest{b[b],b} \otimes \forest{b=b[1,1]} + \forest{b=b[b[1]],1} \otimes \forest{(b)} + \forest{b=b[b,1],1} \otimes \forest{(b)}$ \\~$ + \forest{b=b[1],b[1]} \otimes \forest{(b)} + \forest{b,b[b[1],1]} \otimes \forest{b=b}$} \\
    \forest{b[1]=b[b[1]]} & \makecell{$\forest{b,b,b} \otimes \forest{b[1]=b[b[1]]} + \forest{b=b[b[1]],1} \otimes \forest{(b)} + \forest{b[1]=b[b],1} \otimes \forest{(b)} + \forest{b[1]=b,b[1]} \otimes \forest{(b)}$ \\~$ + \forest{b[b],b} \otimes \forest{b[1]=b[1]}$} \\
    \forest{b=b[b,1,1]} & \makecell{$\forest{b,b,b} \otimes \forest{b=b[b,1,1]} + \forest{b,b[b]} \otimes \forest{b=b[1,1]} + \forest{b,b,b[1,1]} \otimes \forest{b=b[b]} + \forest{b=b[1],1,b} \otimes \forest{(b[b])}$ \\~$ + \forest{b=b[1,1],b} \otimes \forest{(b)} + 2 \forest{b=b[b,1],1} \otimes \forest{(b)} + \forest{b,b[b,1,1]} \otimes \forest{b=b}$} \\
    \forest{b[1]=b[b,1]} & \makecell{$\forest{b,b,b} \otimes \forest{b[1]=b[b,1]} + \forest{b,b[b]} \otimes \forest{b[1]=b[1]} + \forest{b[1]=b,1,b} \otimes \forest{(b[b])} + \forest{b=b[b,1],1} \otimes \forest{(b)}$ \\~$ + \forest{b[1]=b[b],1} \otimes \forest{(b)} + \forest{b[1]=b[1],b} \otimes \forest{(b)}$} \\
    \forest{b[1,1]=b[b]} & \makecell{$\forest{b,b,b} \otimes \forest{b[1,1]=b[b]} + \forest{b,b[b]} \otimes \forest{b[1,1]=b} + \forest{b[1,1],b,b} \otimes \forest{b=b[b]} + \forest{b[1]=b,1,b} \otimes \forest{(b[b])}$ \\~$ + \forest{b[1,1]=b,b} \otimes \forest{(b)} + 2 \forest{b[1]=b[b],1} \otimes \forest{(b)} + \forest{b[1,1],b[b]} \otimes \forest{b=b}$} \\
    \forest{b=b[1,1,2,2]} &~$\forest{b,b} \otimes \forest{b=b[1,1,2,2]} + 2 \forest{b[1,1],b} \otimes \forest{b=b[1,1]} + 4 \forest{b=b[1,1,2],2} \otimes \forest{(b)} + \forest{b,b[1,1,2,2]} \otimes \forest{b=b}$ \\
    \forest{b[2]=b[1,1,2]} & \makecell{$\forest{b,b} \otimes \forest{b[2]=b[1,1,2]} + \forest{b[1,1],b} \otimes \forest{b[1]=b[1]} + \forest{b=b[1,1,2],2} \otimes \forest{(b)} + 2 \forest{b[2]=b[1,2],1} \otimes \forest{(b)}$ \\~$ + \forest{b[2]=b[1,1],2} \otimes \forest{(b)}$} \\
    \forest{b[2,2]=b[1,1]} &~$\forest{b,b} \otimes \forest{b[2,2]=b[1,1]} + 2 \forest{b[1,1],b} \otimes \forest{b=b[1,1]} + 4 \forest{b[2]=b[1,1],2} \otimes \forest{(b)} + \forest{b[1,1],b[2,2]} \otimes \forest{b=b}$ \\
    \forest{b[1,2]=b[1,2]} &~$\forest{b,b} \otimes \forest{b[1,2]=b[1,2]} + 4 \forest{b[2]=b[1,1],2} \otimes \forest{(b)}$ \\
    \forest{b=b[1],(b[1])} & \makecell{$\forest{b,b,b} \otimes \forest{b=b[1],(b[1])} + \forest{b=b[1],1,b} \otimes \forest{(b[b])} + \forest{b,b,(b[1]),1} \otimes \forest{b=b[b]} + \forest{(b),b} \cdot \forest{b} \otimes \forest{b=b[1,1]}$ \\~$ + \forest{b=b,1,(b[1])} \otimes \forest{(b)} + \forest{(b),1,b=b[1]} \otimes \forest{(b)} + \forest{b,(b[1]),b[1]} \otimes \forest{b=b}$} \\
    \forest{b=b[1],b[1]=b} &~$\forest{b,b,b,b} \otimes \forest{b=b[1],b[1]=b} + 2 \forest{b=b[1],1,b,b} \otimes \forest{b=b[b]} + 2 \forest{b=b,b=b[1],1} \otimes \forest{(b)}$ \\
    \forest{b[b[b]]} &~$\forest{b,b,b} \otimes \forest{b[b[b]]} + 2 \forest{b[b],b} \otimes \forest{b[b]} + \forest{b[b[b]]} \otimes \forest{b}$ \\
    \forest{b[b,b]} &~$\forest{b,b,b} \otimes \forest{b[b,b]} + 2 \forest{b[b],b} \otimes \forest{b[b]} + \forest{b[b,b]} \otimes \forest{b}$ \\
    \forest{b[b[1,1]]} &~$\forest{b,b} \otimes \forest{b[b[1,1]]} + \forest{b[1,1],b} \otimes \forest{b[b]} + \forest{b[b]} \otimes \forest{b[1,1]} + \forest{b[b[1,1]]} \otimes \forest{b}$ \\
    \forest{b[b[1],1]} &~$\forest{b,b} \otimes \forest{b[b[1],1]} + \forest{b[b]} \otimes \forest{b[1,1]} + \forest{b[b[1],1]} \otimes \forest{b}$ \\
    \forest{b[b,1,1]} &~$\forest{b,b} \otimes \forest{b[b,1,1]} + \forest{b[1,1],b} \otimes \forest{b[b]} + \forest{b[b]} \otimes \forest{b[1,1]} + \forest{b[b,1,1]} \otimes \forest{b}$ \\
    \forest{b[1,1,2,2]} &~$\forest{b} \otimes \forest{b[1,1,2,2]} + 2 \forest{b[1,1]} \otimes \forest{b[1,1]} + \forest{b[1,1,2,2]} \otimes \forest{b}$ \\
    \forest{(b),b[b]} &~$\forest{b,b,b} \otimes \forest{(b),b[b]} + 2 \forest{(b),b} \cdot \forest{b} \otimes \forest{b[b]} + \forest{b[b]} \otimes \forest{(b),b} + \forest{(b),b[b]} \otimes \forest{b}$ \\
    \forest{(b[b]),b} &~$\forest{b,b,b} \otimes \forest{(b[b]),b} + \forest{b[b],b} \otimes \forest{(b),b} + \forest{(b),b} \cdot \forest{b} \otimes \forest{(b),b} + \forest{(b[b]),b} \otimes \forest{b}$ \\
    \forest{(b,b),b} &~$\forest{b,b,b} \otimes \forest{(b,b),b} + 2 \forest{b[b],b} \otimes \forest{(b),b} + \forest{(b,b),b} \otimes \forest{b}$ \\
    \forest{(b),b[1,1]} &~$\forest{b,b} \otimes \forest{(b),b[1,1]} + \forest{(b),b} \otimes \forest{b[1,1]} + \forest{b[1,1]} \otimes \forest{(b),b} + \forest{(b),b[1,1]} \otimes \forest{b}$ \\
    \forest{(b),(b),b} &~$\forest{b,b,b} \otimes \forest{(b),(b),b} + 4 \forest{(b),b} \cdot \forest{b} \otimes \forest{(b),b} + \forest{(b),(b),b} \otimes \forest{b}$ \\
    \forest{(b[1,1]),b} &~$\forest{b,b} \otimes \forest{(b[1,1]),b} + \forest{b[1,1],b} \otimes \forest{(b),b} + 2 \forest{(b[1]),1,b} \otimes \forest{(b),b} + \forest{(b),b} \otimes \forest{b[1,1]} + \forest{(b[1,1]),b} \otimes \forest{b}$ \\
    \forest{(b),(b[1]),1} &~$\forest{b,b} \otimes \forest{(b),(b[1]),1} + \forest{(b),b} \otimes \forest{(b[1]),1} + \forest{b,(b[1]),1} \otimes \forest{(b),b} + \forest{(b),(b[1]),1} \otimes \forest{b}$ \\
    \forest{(b[b[1]]),1} &~$\forest{b,b} \otimes \forest{(b[b[1]]),1} + \forest{b[b]} \otimes \forest{(b[1]),1} + \forest{(b[b[1]]),1} \otimes \forest{b}$ \\
    \forest{(b[b,1]),1} &~$\forest{b,b} \otimes \forest{(b[b,1]),1} + \forest{b[b]} \otimes \forest{(b[1]),1} + \forest{(b[b,1]),1} \otimes \forest{b}$ \\
    \forest{(b[1]),b[1]} &~$\forest{b,b} \otimes \forest{(b[1]),b[1]} + \forest{(b),b} \otimes \forest{b[1,1]} + \forest{(b[1]),1,b} \otimes \forest{b[b]} + \forest{(b[1]),b[1]} \otimes \forest{b}$ \\
    \forest{(b,b[1]),1} &~$\forest{b,b} \otimes \forest{(b,b[1]),1} + 2 \forest{b[b]} \otimes \forest{(b[1]),1} + \forest{(b,b[1]),1} \otimes \forest{b}$ \\
    \forest{(b[1,2,2]),1} &~$\forest{b} \otimes \forest{(b[1,2,2]),1} + \forest{b[1,1]} \otimes \forest{(b[1]),1} + \forest{(b[1]),1} \otimes \forest{b[1,1]} + \forest{(b[1,2,2]),1} \otimes \forest{b}$ \\
    \forest{b=b,b[b]} &~$\forest{b,b,b,b} \otimes \forest{b=b,b[b]} + \forest{b,b,b[b]} \otimes \forest{b=b,b} + 2 \forest{b=b,b} \cdot \forest{b} \otimes \forest{b[b]} + \forest{b=b,b[b]} \otimes \forest{b}$ \\
    \forest{b=b[b],b} &~$\forest{b,b,b,b} \otimes \forest{b=b[b],b} + \forest{b,b,b[b]} \otimes \forest{b=b,b} + \forest{b=b,b} \cdot \forest{b} \otimes \forest{b[b]} + \forest{b=b,b} \cdot \forest{b} \otimes \forest{(b),b} + \forest{b=b[b],b} \otimes \forest{b}$ \\
    \forest{b=b[1,1],b} &~$\forest{b,b,b} \otimes \forest{b=b[1,1],b} + \forest{b,b,b[1,1]} \otimes \forest{b=b,b} + \forest{b=b,b} \otimes \forest{b[1,1]} + 2 \forest{b=b[1],1,b} \otimes \forest{(b),b} + \forest{b=b[1,1],b} \otimes \forest{b}$ \\
    \forest{b[1]=b[1],b} &~$\forest{b,b,b} \otimes \forest{b[1]=b[1],b} + \forest{b=b,b} \otimes \forest{b[1,1]} + 2 \forest{b=b[1],1,b} \otimes \forest{(b),b} + \forest{b[1]=b[1],b} \otimes \forest{b}$ \\
    \forest{b=b,b[1,1]} &~$\forest{b,b,b} \otimes \forest{b=b,b[1,1]} + \forest{b=b,b} \otimes \forest{b[1,1]} + \forest{b[1,1],b,b} \otimes \forest{b=b,b} + \forest{b=b,b[1,1]} \otimes \forest{b}$ \\
    \forest{(b),b=b,b} &~$\forest{b,b,b,b} \otimes \forest{(b),b=b,b} + 3 \forest{(b),b} \cdot \forest{b,b} \otimes \forest{b=b,b} + 2 \forest{b=b,b} \cdot \forest{b,b} \otimes \forest{(b),b} + \forest{(b),b=b,b} \otimes \forest{b}$ \\
    \forest{b=b,b=b,b} &~$\forest{b,b,b,b} \otimes \forest{b=b,b=b,b} + 6 \forest{b=b,b} \cdot \forest{b,b} \otimes \forest{b=b,b} + \forest{b=b,b=b,b} \otimes \forest{b}$ \\
    \forest{b=b,(b[1]),1} &~$\forest{b,b,b} \otimes \forest{b=b,(b[1]),1} + \forest{b=b,b} \otimes \forest{(b[1]),1} + \forest{b,b,(b[1]),1} \otimes \forest{b=b,b} + \forest{b=b,(b[1]),1} \otimes \forest{b}$ \\
    \forest{(b),b=b[1],1} &~$\forest{b,b,b} \otimes \forest{(b),b=b[1],1} + 2 \forest{(b),b} \cdot \forest{b} \otimes \forest{b=b[1],1} + \forest{b,b=b[1],1} \otimes \forest{(b),b} + \forest{(b),b=b[1],1} \otimes \forest{b}$ \\
    \forest{b=b,b=b[1],1} &~$\forest{b,b,b,b} \otimes \forest{b=b,b=b[1],1} + 4 \forest{b=b,b} \cdot \forest{b} \otimes \forest{b=b[1],1} + \forest{b=b[1],1,b,b} \otimes \forest{b=b,b} + \forest{b=b,b=b[1],1} \otimes \forest{b}$ \\
    \forest{b=b[b[1]],1} &~$\forest{b,b,b} \otimes \forest{b=b[b[1]],1} + \forest{b[b],b} \otimes \forest{b=b[1],1} + \forest{b=b[b[1]],1} \otimes \forest{b}$ \\
    \forest{b=b[b,1],1} &~$\forest{b,b,b} \otimes \forest{b=b[b,1],1} + \forest{b,b[b]} \otimes \forest{b=b[1],1} + \forest{b=b[1],1,b} \otimes \forest{b[b]} + \forest{b=b[b,1],1} \otimes \forest{b}$ \\
    \forest{b[1]=b[b],1} &~$\forest{b,b,b} \otimes \forest{b[1]=b[b],1} + \forest{b,b[b]} \otimes \forest{b[1]=b,1} + \forest{b[1]=b,1,b} \otimes \forest{b[b]} + \forest{b[1]=b[b],1} \otimes \forest{b}$ \\
    \forest{b=b[1,1,2],2} &~$\forest{b,b} \otimes \forest{b=b[1,1,2],2} + \forest{b[1,1],b} \otimes \forest{b=b[1],1} + \forest{b=b[1],1,b} \otimes \forest{b[1,1]} + \forest{b,b[1,1,2],2} \otimes \forest{b}$ \\
    \forest{b[2]=b[1,1],2} &~$\forest{b,b} \otimes \forest{b[2]=b[1,1],2} + \forest{b[1,1],b} \otimes \forest{b[1]=b,1} + \forest{b[2]=b,2} \otimes \forest{b[1,1]} + \forest{b[2]=b[1,1],2} \otimes \forest{b}$ \\
    \forest{b[1]=b[1,2],2} &~$\forest{b,b} \otimes \forest{b[1]=b[1,2],2} + \forest{b[1]=b[1,2],2} \otimes \forest{b}$ \\
    \forest{b[b[1]],1} &~$\forest{b,b} \otimes \forest{b[b[1]],1} + \forest{b[b]} \otimes \forest{b[1],1}$ \\
    \forest{b[b,1],1} &~$\forest{b,b} \otimes \forest{b[b,1],1} + \forest{b[b]} \otimes \forest{b[1],1}$ \\
    \forest{b[1],b[1]} &~$\forest{b,b} \otimes \forest{b[1],b[1]}$ \\
    \forest{b[1,2,2],1} &~$\forest{b} \otimes \forest{b[1,2,2],1} + \forest{b[1,1]} \otimes \forest{b[1],1}$ \\
    \forest{(b),(b[1]),1} &~$\forest{b,b} \otimes \forest{(b),(b[1]),1} + \forest{(b),b} \otimes \forest{(b[1]),1} + \forest{b,(b[1]),1} \otimes \forest{(b),b} \forest{(b),(b[1]),1} \otimes \forest{b}$ \\
    \forest{(b[b]),1,1} &~$\forest{b,b} \otimes \forest{(b[b]),1,1} + \forest{(b),b} \otimes \forest{(b),1,1}$ \\
    \forest{(b,b),1,1} &~$\forest{b,b} \otimes \forest{(b,b),1,1} + 2 \forest{b[b]} \otimes \forest{(b),1,1}$ \\
    \forest{(b),(b),1,1} &~$\forest{b,b} \otimes \forest{(b),(b),1,1} + 2 \forest{(b),b} \otimes \forest{(b),1,1}$ \\
    \forest{(b[1,1]),2,2} &~$\forest{b} \otimes \forest{(b[1,1]),2,2} + \forest{b[1,1]} \otimes \forest{(b),1,1} + 2 \forest{(b[1]),1} \otimes \forest{(b),1,1}$ \\
    \forest{(b[1,2]),1,2} &~$\forest{b} \otimes \forest{(b[1,2]),1,2} + 2 \forest{(b[1]),1} \otimes \forest{b[1],1}$ \\
    \forest{b=b,b[1],1} &~$\forest{b,b,b} \otimes \forest{b=b,b[1],1} + \forest{b=b,b} \otimes \forest{b[1],1}$ \\
    \forest{b=b[b],1,1} &~$\forest{b,b,b} \otimes \forest{b=b[b],1,1} + \forest{b=b,b} \otimes \forest{(b),1,1}$ \\
    \forest{b=b[1,1],2,2} &~$\forest{b,b} \otimes \forest{b=b[1,1],2,2} + 2 \forest{b=b[1],1} \otimes \forest{(b),1,1}$ \\
    \forest{b[1]=b[1],2,2} &~$\forest{b,b} \otimes \forest{b[1]=b[1],2,2} + 2 \forest{b=b[1],1} \otimes \forest{(b),1,1}$ \\
    \forest{(b),b=b,1,1} &~$\forest{b,b,b} \otimes \forest{(b),b=b,1,1} + 2 \forest{(b),b} \cdot \forest{b} \otimes \forest{b=b,1,1} + \forest{b=b,b} \otimes \forest{(b),1,1}$ \\
    \forest{b=b,b=b,1,1} &~$\forest{b,b,b,b} \otimes \forest{b=b,b=b,1,1} + 4 \forest{b=b,b} \cdot \forest{b} \otimes \forest{b=b,1,1}$ \\
    \forest{b[1]=b[2],1,2} &~$\forest{b,b} \otimes \forest{b[1]=b[2],1,2} + 2 \forest{b[1]=b,1} \otimes \forest{b[1],1}$ \\
    \forest{b=b[1,2],1,2} &~$\forest{b,b} \otimes \forest{b=b[1,2],1,2} + 2 \forest{b[1]=b,1} \otimes \forest{b[1],1}$ \\
    \forest{b[1,2],1,2} &~$\forest{b} \otimes \forest{b[1,2],1,2}$
\end{longtable}

\end{appendices}

\end{document}